\documentclass[12pt,leqno]{article}
\usepackage{amsmath,amsthm,amssymb}
\usepackage[T1]{fontenc}
\usepackage{graphics}
\usepackage{graphicx}

\newtheorem{Th}{Theorem}[section]
\newtheorem{Prop}[Th]{Proposition}
\newtheorem{Lemma}[Th]{Lemma}
\newtheorem{Cor}[Th]{Corollary}

\theoremstyle{remark}
\newtheorem{Remark}{Remark}

\numberwithin{equation}{section}

\def\un{{\bf 1}}
\newcommand{\koniec}{\nopagebreak\hspace*{\fill}
\nolinebreak$\square$\vspace{5mm}\par}

\newcommand{\jed}{\mbox{\boldmath$1$}}

\newcommand{\cd}{{\cal D}}

\newcommand{\cu}{{\cal U}}

\newcommand{\cb}{{\cal B}}
\newcommand{\ca}{{\cal A}}
\newcommand{\ch}{{\cal H}}
\newcommand{\cf}{{\cal F}}

\newcommand{\cp}{{\cal P}}

\newcommand{\bs}{\mathbb{S}}

\newcommand{\Q}{\mathbb{Q}}

\newcommand{\C}{{\mathbb{C}}}
\newcommand{\Z}{{\mathbb{Z}}}
\newcommand{\N}{{\mathbb{N}}}
\newcommand{\cir}{{\mathbb{S}}}
\newcommand{\xbm}{(X,{\cal B},\mu)}

\newcommand{\ov}{\overline}
\newcommand{\beq}{\begin{equation}}
\newcommand{\eeq}{\end{equation}}
\newcommand{\vep}{\varepsilon}
\newcommand{\va}{\varphi}

\newcommand{\ot}{\otimes}
\newcommand{\la}{\lambda}
\newcommand{\wt}{\widetilde}

\begin{document}
\title{Polynomial actions of unitary operators and
idempotent ultrafilters\footnote{MSC classes:  37A45 (primary), 05D10, 47A35, 47B15}}

\author{V. Bergelson\thanks{Research partially supported
by NSF grants DMS-0901106 and DMS-1162073}\and  S. Kasjan  \and M.
Lema\'nczyk\thanks{Research supported  by Polish National Science Center grant DEC-2011/03/B/ST1/00407}}

\maketitle

\thispagestyle{empty}

\begin{abstract}Let $p$ be an idempotent ultrafilter over
$\N$. For a positive integer $N$, let ${\cal P}_{\leq N}$ denote the additive group of polynomials $P\in\Z[x]$ with ${\rm deg}\, P\leq N$ and $P(0)=0$. Given a unitary operator $U$ on a Hilbert space ${\cal H}$,
we prove, for each $N\geq1$, the existence of a unique
decomposition ${\cal H}=\bigoplus_{r\geq 1}{\cal H}^{(N)}_r$ into
closed, $U$-invariant subspaces such that \begin{itemize} \item
for any polynomial $P\in{\cal P}_{\leq N}$,  we have
$$p\, \text{-}\!\lim_{n\in\N} \left(U|_{{\cal
H}_r^{(N)}}\right)^{P(n)}=0_{{\cal H}_r^{(N)}}\;\mbox{or}\;
Id_{{\cal H}_r^{(N)}},\; \mbox{for each}\; r\geq1 ;$$

\item for
each $r\neq s$ there exists $Q\in{\cal P}_{\leq N}$ such that
$$p\,\text{-}\!\lim_{n\in\N} \left(U|_{{\cal H}_r^{(N)}}\right)^{Q(n)}\neq
p\,\text{-}\!\lim_{n\in\N} \left(U|_{{\cal H}_s^{(N)}}\right)^{Q(n)}.$$
\end{itemize}
In connection with this result we introduce the notion of rigidity
group. Namely, a subgroup $G\subset {\cal P}_{\leq N}$ is called
an $N$-rigidity group if there exist an idempotent ultrafilter $p$
over $\N$ and a unitary operator $U$ on a Hilbert space $\cal H$
such that \beq\label{ab1} G=\{P\in{\cal P}_{\leq N}:\:
p\,\text{-}\!\lim_{n\in\N} U ^{P(n)}=Id\}\eeq and \beq\label{ab2}
p\,\text{-}\!\lim_{n\in\N} U ^{Q(n)}=0\;\;\mbox{for each}\;\;Q\in{\cal
P}_{\leq N}\setminus G.\eeq
 The main result of the paper states that a subgroup $G\subset
{\cal P}_{\leq N}$ satisfying $\max\{{\rm deg}\, P:\:P\in G\}=N$
is an $N$-rigidity group if and only if $G$ has finite index in
${\cal P}_{\leq N}$.
\end{abstract}

\tableofcontents

\section*{Introduction} One of the goals of this paper is to establish a new Hilbert space decomposition theorem for polynomial actions of unitary operators which may be seen as a far reaching refinement of classical splitting results summarized in the following theorem.

\begin{Th}\label{tw1wstep}
Let $\ch$ be a Hilbert space~\footnote{We are tacitly assuming that $\ch$ is separable. It is not hard to see that this assumption can be made without the loss of generality. The theorems in this paper which  pertain to unitary operators on separable Hilbert spaces hold for non-separable spaces as well.} and $U:\ch\to\ch$ a unitary operator (we write for short $U\in\cu(\ch)$). Then
\beq\label{wstep1}\ch=\ch_{inv}\oplus\ch_{erg},
\eeq
where $\ch_{inv}=\{f\in\ch:\:Uf=f\}$ and $\ch_{erg}=\{f\in\ch:\lim_{N-M\to\infty}\frac1{N-M}\sum_{n=M}^{N-1}U^nf=0\}$;
\beq\label{wstep2}
\ch=\ch_{comp}\oplus\ch_{wm},\eeq
where $$\ch_{comp}=\{f\in\ch:\:\ov{\{U^nf\}}_{n\in\Z}\;\mbox{is compact in the norm topology}\}=
\ov{\mbox{\rm span}}\{f\in\ch:\:(\exists \lambda\in\C)\;Uf=\lambda f\}$$ and $$\ch_{wm}=\{f\in\ch:\:(\forall g\in\ch)\;\lim_{N-M\to\infty}\frac1{N-M}\sum_{n=M}^{N-1}|\langle U^nf,g\rangle|=0\}.$$
\end{Th}
The decomposition~(\ref{wstep1}) in the above theorem is the classical {\em ergodic} Hilbert space decomposition which is behind the von Neumann's ergodic theorem (see, e.g.\ \cite{Ha}). The more interesting decomposition~(\ref{wstep2}) is a special case of Jacobs-Glicksberg-de Leeuw decomposition~\cite{NN2} which is connected with the notion of {\em weak mixing}.

Recall that a unitary operator $U\in\cu(\ch)$ is called {\em weakly mixing} if it has no non-trivial eigenvectors, meaning that if for some $\lambda\in \C$ and $f\in\ch$ one has $Uf=\lambda f$, then $f=0$. The notion of weak mixing was introduced in~\cite{Ko-Ne} and has a multitude of equivalent formulations (see, e.g.\ \cite{Be1} or \cite{Be-Go}). An important class of weakly mixing operators has its origins in the theory of measure-preserving systems. Given a
%standard Borel~\footnote{The standardness assumption of the underlying probability space is a usual and often %necessary assumption in ergodic theory; one advantage of it is that the corresponding Hilbert space $L^2\xbm$ is %separable. This assumption is superfluous in some recurrent type results, see e.g.\ Theorems~\ref{tw2wstep} %and~\ref{tw3wstep} below.}
measure space $\xbm$ and an invertible measure-preserving transformation $T:X\to X$, define the unitary operator $U_T$ on $L^2\xbm$ by the formula $U_T(f)(x)=f(Tx)$. The transformation $T$ is called {\em weakly mixing} if $U_T$ is a weakly mixing operator on the space $\ch=L^2_0\xbm:=\{f\in L^2\xbm:\: \int f\,d\mu=0\}$.

It follows from~(\ref{wstep2}) that a unitary operator $U\in\cu(\ch)$ is weakly mixing if and only if for any $f,g\in\ch$ one has $\lim_{N-M\to \infty}\frac1{N-M}\sum_{n=M}^{N-1}|\langle U^nf,g\rangle|=0$. This, in turn, implies that $U$ is weakly mixing if and only if for any $f\in\ch$ there exists a set $E\subset\N:=\{1,2,\ldots\}$ satisfying
$d(E):=\lim_{N-M\to\infty}\frac{|E\cap[M,\ldots,N-1]|}{N-M}=0$ such that $U^nf\to 0$ weakly, when $n\to\infty,n\notin E$~\footnote{If $\ch$ is separable, one can actually show that $U\in\cu(\ch)$ is weakly mixing if and only if there exists $E\subset\N$ with $d(E)=0$ such that, for any $f\in \ch$, $U^nf\to0$ as $n\to\infty$, $n\notin E$.}. It is the presence of the exceptional set $E$ which  distinguishes between the notion of weak mixing and that of strong mixing (which  is defined by the condition $U^nf\to 0$ weakly, when $n\to \infty$). One can show that a ``generic'' unitary operator is weakly but not strongly mixing (see for example~\cite{Ha}, \cite{Na}). Moreover, the generic unitary operator is simultaneously weakly mixing and {\em rigid} (see, e.g.\ \cite{Be-Ju-Le-Ro} or \cite{Na}), meaning that, there exists a
sequence $n_k\to\infty$ such that for every $f\in\ch$, $U^{n_k}f\to f$ (in $\ch$), when $k\to\infty$. Since the exceptional set, along which a weakly mixing operator $U$ is rigid, is of zero density, it does not affect the value of the Ces\`aro limits
$\lim_{N-M\to\infty}\frac1{N-M}\sum_{n=M}^{N-1}|\langle U^nf,g\rangle|$. Consequently,  if one is  interested in distinguishing between various classes of weakly mixing operators based on their rigidity properties, the Ces\`aro averages may not be so effective a tool and one may want to look for some alternative notions of convergence. We will see below that the notion of convergence along the idempotent ultrafilters provides a  satisfactory alternative to Ces\`aro limits and seems to be especially useful for the study of behaviour of unitary operators along polynomials.

To provide an instructive glimpse into the effectiveness of idempotent ultrafilters, let us briefly discuss a polynomial generalization of the classical Khintchine's recurrence theorem which leads to interesting combinatorial applications. The celebrated Poincar\'e's recurrence theorem states that for any measure-preserving transformation $T$ on a probability space $\xbm$ and any $A\in\cb$ with $\mu(A)>0$, there exists $n\in\N$ such that $\mu(A\cap T^{-n}A)>0$. Khintchine's refinement \cite{Kh} of Poincar\'e's  theorem can be formulated as follows.  Recall that a subset of $\Z$ (or of $\N$) is called {\em syndetic}, if it has bounded gaps.

\begin{Th}\label{tw2wstep} Let $\xbm$ be a  probability space and $T:X\to X$ an invertible measure-preserving transformation. Assume that $A\in\cb$, $\mu(A)>0$. Then for any $\vep>0$ the set $\{n\in\Z:\:\mu(A\cap T^{-n}A)>\mu(A)^2-\vep\}$ is syndetic.
\end{Th}

To prove Theorem~\ref{tw2wstep} one can use von Neumann's ergodic theorem. Let $f=\un_A$ and let $$\lim_{N-M\to\infty}\frac1{N-M}\sum_{n=M}^{N-1}U_T^nf=f^\ast,$$ where $f^\ast={\rm proj}_{\ch_{inv}}f$ is the orthogonal projection of $f$ on the space of $U_T$-invariant functions. Then we have
$$
\lim_{N-M\to\infty}\frac1{N-M}\sum_{n=M}^{N-1}\mu(A\cap T^{-n}A)=\lim_{N-M\to\infty}\frac1{N-M}\sum_{n=M}^{N-1}\int f\cdot U_T^n(f)\,d\mu$$
$$
=\int f\cdot {\rm proj}_{\ch_{inv}}f\,d\mu=\int {\rm proj}_{\ch_{inv}}f\cdot {\rm proj}_{\ch_{inv}}f\,d\mu\cdot\int \un\cdot \un\,d\mu$$$$\geq \left(\int {\rm proj}_{\ch_{inv}}f\cdot \un\,d\mu\right)^2=\left(\int f\,d\mu\right)^2=\mu(A)^2.$$
(We used the fact that ${\rm proj}_{\ch_{inv}}$ is a self-adjoint operator and the Cauchy-Schwarz inequality)~\footnote{This proof is essentially due to Hopf \cite{Ho}. For a different proof of combinatorial nature, see \cite{Be}, Section~5.}. Consider now the following polynomial extension of Poincar\'e's theorem, obtained by Furstenberg \cite{Fu}.

\begin{Th}\label{tw3wstep} Let $\xbm$ be a  probability space and $T:X\to X$ an invertible measure-preserving transformation. Assume that $A\in\cb$, $\mu(A)>0$. Then for any $\vep>0$ and any polynomial $P\in\Z[x]$ with $P(0)=0$, there  are arbitrarily large $n$ such that  $\mu(A\cap T^{P(n)}A)>\mu(A)^2-\vep$~\footnote{Theorem~\ref{tw3wstep} implies S\'ark\"ozy's theorem stating that if $E\subset \N$ is a subset of positive upper Banach density: $d^\ast(E):=\limsup_{N-M\to\infty}\frac{|E\cap [M,N-1]|}{N-M}>0$, and $P\in\Z[x]$ satisfies $P(0)=0$, then one can find infinitely many $n\in\Z$ such that, for some $x,y\in E$, $x-y= P(n)$.}.
\end{Th}

To prove Theorem~\ref{tw3wstep},  Furstenberg invokes the spectral theorem\footnote{For  a ``geometric'' proof avoiding the spectral theorem, see \cite{Be87}, \cite{Be}.} and some classical results on uniform distribution. The polynomial recurrence theorem in question follows from the fact that for any $f\in\ch$ and any unitary operator $U\in\cu(\ch)$ the strong limit
\beq\label{wstep4}\lim_{N-M\to\infty}\frac1{N-M}\sum_{n=M}^{N-1} U^{P(n)}f\;\;\mbox{exists and, in addition,}\eeq
\beq\label{wstep5}
\lim_{N-M\to\infty}\frac1{N-M}\sum_{n=M}^{N-1}\mu(A\cap T^{P(n)}A)=:c_A>0.\eeq
While it is not known what is the ``optimal'' value of the constant $c_A$, one can provide examples where $c_A$ is strictly smaller than $\mu(A)^2$. There is one more distinction between the polynomial result~(\ref{wstep4}) and the von Neumann's ergodic theorem. Namely, while the limit of the ``linear'' Ces\`aro averages in von Neumann's theorem is an orthogonal projection, this is no longer the case for the polynomial averages in~(\ref{wstep4}). To ``fix'' this situation, we will introduce the ultrafilter analogues of the limits~(\ref{wstep4}) and (\ref{wstep5}) leading to a polynomial version of Khintchine's recurrence theorem which even in the case of linear polynomials gives more than Theorem~\ref{tw2wstep}.

We first briefly summarize some facts concerning the idempotent ultrafilters on $\N$. The reader will find more details in Section~\ref{section:p-limits} and the references indicated there.

An ultrafilter $p$ on $\N$ is a family of subsets of $\N$ satisfying (i) $\emptyset\notin p$, (ii) $\N\in p$, (iii) $A\in p$ and $A\subset B\subset\N$ implies $B\in p$, (iv) $A,B\in p$ implies $A\cap B\in p$ and (v) if $r\in\N$ and $\N=A_1\cup\ldots\cup A_r$, then some $A_i\in p$. (In other words, an ultrafilter is a maximal filter.) The space of ultrafilters on $\N$ is denoted by $\beta\N$ (and is identified with the Stone-$\check{\mbox{C}}$ech compactification of $\N$). Any element $n\in\N$ can be identified with the ultrafilter $\{A\subset\N:\:n\in A\}$. Given $A\subset\N$, let $\ov{A}:=\{p\in\beta\N:\:A\in p\}$. The family $\{\ov{A}:\:A\subset\N\}$ forms a basis for the open sets (and a basis for the closed sets as well) of $\beta\N$. The operation of addition on $\N$ can be extended to $\beta\N$ as follows. Given $p,q\in\beta\N$ and $A\subset\N$,
\beq\label{sty1}A\in p+q \Leftrightarrow\{n\in\N:\: A-n\in p\}\in q~\footnote{For $A\subset\N$ and $n\in\N$, the set $A-n$ is defined as $\{y\in\N:\:y+n\in A\}$.}.\eeq
 Formula (\ref{sty1}) makes $(\beta\N,+)$ a compact left topological semigroup~\footnote{\label{stopkaSTYCZ}Making $(\beta\N,+)$ a left topological semigroup means that for each $p\in\beta\N$ the function $\lambda_p(q)=p+q$ is continuous.}. By Ellis' lemma \cite{El}, any compact left topological semigroup has an idempotent.  Note that whenever $p\in\beta\N$ is an idempotent, by~(\ref{sty1}), we have
\beq\label{p2=p}
A\in p\Leftrightarrow A\in p+p\Leftrightarrow \{n\in\N:\:A-n\in p\}\in p.\eeq

Let now $X$ be a topological space. Given a sequence $(x_n)\subset X$ and an ultrafilter $p\in\beta\N$, we will write $p\,\text{-}\!\lim_{n\in\N}x_n=x$ if for any neighbourhood $U\ni x$, the set $\{n\in\N:\: x_n\in U\}\in p$.  Then, whenever $X$ is a compact Hausdorff space,
$p\,\text{-}\!\lim_{n\in\N}x_n$ exists and is unique. Moreover, if $p=p+p$, then~(\ref{p2=p}) implies
\beq\label{wstep6}
p\,\text{-}\!\lim_{n\in\N}x_n= p\,\text{-}\!\lim_{n\in\N}\left(p\,\text{-}\!
\lim_{m\in\N}x_{n+m}\right).\eeq

An immediate application of the introduced concepts gives an ultrafilter analogue of the von Neumann's ergodic theorem (and also a natural analogue of the splitting~(\ref{wstep1}) which we encountered in Theorem~\ref{tw1wstep}). Assume that $U\in\cu(\ch)$, let $r>0$,  and let $f\in\ch$ with $\|f\|=r$. The $r$-ball $X:=\{g\in\ch:\:\|g\|\leq r\}$, equipped with a metric $d$ induced by the weak topology, is a $U$-invariant compact Hausdorff space. Let $p\in \beta\N$, $p+p=p$ and set $f^\ast:=p\,\text{-}\!\lim_{n\in\N} U^nf$. Utilizing the formula~(\ref{wstep6}), we have
$$
f^\ast=p\,\text{-}\!\lim_{n\in\N} U^nf= p\,\text{-}\!\lim_{n\in\N}\left( p\,\text{-}\!\lim_{m\in\N} U^{n+m}f\right)$$
$$
= p\,\text{-}\!\lim_{n\in\N} U^n\left( p\,\text{-}\!\lim_{m\in\N} U^mf\right)= p\,\text{-}\!\lim_{n\in\N} U^nf^\ast.$$
It follows that for any $f\in\ch$ and any idempotent $p\in\beta\N$, $f^\ast= p\,\text{-}\!\lim_{n\in\N} U^nf$ is a rigid vector\footnote{\label{raa}A vector $f\in {\cal H}$ is called a {\em rigid vector} for $U$ if
for some increasing sequence $(n_k)\subset\N$ we have $U^{n_k}f\to
f$ strongly. Note that in this particular case, the strong convergence is equivalent to $U^{n_k}f\to f$ weakly. An operator $U$ is called {\em rigid along} $(n_k)$ if
$U^{n_k}\to Id$ strongly.}.  Indeed,  notice that, for each $\vep>0$, the set $\{n\in\N:\:d(U^nf,f)<\vep\}$ is a member of $p$ and hence is not empty. Therefore, we can find an increasing subsequence $(n_i)$ such that $d(U^{n_i}f,f)\to0$ which is equivalent to $U^{n_i}f\to f$ in $\ch$.
We remark in passing that it is not hard to show that any rigid vector is of the form $p\,\text{-}\!\lim_{n\in\N} U^nf$ for some idempotent $p\in \beta\N$ and $f\in \ch$.

Let now $Y$ be the unit ball in the space of bounded operators on $\ch$. Then $Y$, equipped with a metric induced by  the weak operator topology, becomes a compact  semitopological (i.e.\ left- and right topological) semigroup. Assume that $p\in\beta\N$, $p+p=p$ and let $W:=p\,\text{-}\!\lim_{n\in\N} U^n$. It is easy to check that $W$ is a self-adjoint idempotent and hence an orthogonal projection on the subspace of $p$-rigid vectors. We summarize this discussion in the following theorem.

\begin{Th}\label{tw4wstep} Assume that $U\in\cu(\ch)$ and let $p\in\beta\N$, $p+p=p$. Then $\ch=\ch_r\oplus\ch_m$, where $\ch_r=\{f\in \ch:\:p\,\text{-}\!\lim_{n\in\N}  U^nf=f\}$\footnote{Note that $\ch_{comp}\subset\ch_r$ because for each $\la\in\C$, $|\la|=1$, and each idempotent $p\in\beta\N$, we have $p\,\text{-}\!\lim_{n\in\N}\lambda^n=1$, so if $Uf=\la f$, $p\,\text{-}\!\lim_{n\in\N}U^nf=f$.} and $\ch_m=\{f\in\ch:\: p\,\text{-}\!\lim_{n\in\N} U^nf=0\}$.\end{Th}

\begin{Remark} When dealing with $p\,$-limits, we usually use weak convergence. It is however worth noticing  that
 the relation $p\,\text{-}\!\lim_{n\in\N}U^nf=f$ on the subspace $\ch_r$ holds in the weak topology if and only if it holds in the strong topology, cf.\ footnote~\ref{raa}.
\end{Remark}

Denote by $\cp$ the group of polynomials $P\in\Z[x]$ satisfying $P(0)=0$. The following result  represents a polynomial extension of Theorem~\ref{tw4wstep}.

\begin{Th}[\cite{Be}]\label{tw5wstep} For each
unitary operator $U\in\cu(\ch)$, each $p\in\beta\N$, $p+p=p$,
and each polynomial $P\in\cp$,
$$
p\,\text{-}\!\lim_{n\in\N} U^{P(n)}={\rm proj}_{\cal F},$$ where $\cal F$
is a closed, $U$-invariant subspace of $\cal H$.
\end{Th}

Theorem~\ref{tw5wstep} allows one to derive a polynomial Khintchine-like theorem.

\begin{Cor}\label{wn1wstep} Let $\xbm$ be a  probability space and $T:X\to X$ an invertible measure-preserving transformation. Given $P\in{\cal P}$, for
each $\vep>0$ and $A\in\cb$, the set
$$
R_{\vep}(A;P):=\{n\in\Z:\: \mu(A\cap T^{P(n)}A)\geq
(\mu(A))^2-\vep\}$$ is an ${\rm IP}^\ast$-set~\footnote{\label{stopka8} A set $S\subset\N$ is ${\rm IP}^\ast$ if $S\in p$ for any idempotent $p\in\beta\N$. The notation reflects the fact that $S$ is ${\rm IP}^\ast$ if and only if $S$ has a nontrivial intersection with any ${\rm IP}$-set
(see~(\ref{fs}) in the next section for the definition of ${\rm IP}$-set). The proof of Corollary~\ref{wn1wstep} then goes as
follows (cf.\ \cite{Be2}). We have $$ a:=p\,\text{-}\!\lim_{n\in\N}\mu(A\cap
T^{-P(n)}A)=\langle {\rm proj}_{\cal
F}\jed_A,\jed_A\rangle\geq\left(\langle {\rm proj}_{\cal
F}\jed_A,\jed\rangle\right)^2=(\mu(A))^2.$$ Therefore, for each
idempotent $p\in \beta\N$, $R_{\vep}(A;P)\cap\N\supset\{n\in\N:\:\left|\mu(A\cap T^{-P(n)}A)-a\right|<\vep\}$, whence $R_{\vep}(A;P)\cap\N\in p$. Thus $R_{\vep}(A;P)$ is ${\rm IP}^\ast$.

It is not hard to show that any ${\rm IP}^\ast$-set is syndetic. On the other hand, not every syndetic set
is ${\rm IP}^\ast$. (For example, $2\N+1$  is syndetic but not ${\rm IP}^\ast$.)
So, Corollary~\ref{wn1wstep} forms a non-trivial extension of Theorem~\ref{tw2wstep} in more than one respect.}.\end{Cor}

For a more general form of the polynomial Khintchine theorem and some combinatorial applications, see~\cite{Be-Fu-M}.

Theorem~\ref{tw5wstep} and Corollary~\ref{wn1wstep} indicate that ergodic theorems along idempotent ultrafilters
can be useful for ergodic-theoretical and combinatorial applications. But, as a matter of fact, studying the limits of
the form $p\,\text{-}\!\lim_{n\in\N} U^{P(n)}$ can also allow one to better
understand the intricate properties of the unitary operators
acting on a Hilbert space $\cal H$. This is of a special interest in  case
when $U$ has continuous spectrum. To continue the line of
juxtaposition of Ces\`aro limits with limits along ultrafilters, notice that if
$U\in{\cal U}({\cal H})$ is weakly mixing then for any
non-constant polynomial $P\in\Z[x]$ one has \cite{Fu} the strong limit
$$\lim_{N-M\to\infty}\frac1{N-M}\sum_{n=M}^{N-1}U^{P(n)}=0.$$
Moreover, as shown in \cite{Be87}, once $\ch=\ch_{wm}$, we also have
\beq\label{eq111}\lim_{N-M\to\infty}\frac1{N-M}\sum_{n=M}^{N-1}|\langle U^{P(n)}f,g\rangle|=0\eeq
for each $f,g\in\ch$ and each non-zero degree polynomial $P\in\Z[x]$. Note also that if~(\ref{eq111}) holds for some non-zero degree polynomial $Q\in\Z[x]$ then $U$ must be weakly mixing and hence~(\ref{eq111}) holds for all non-zero degree polynomials $P\in\Z[x]$.

While the fact expressed by the formula~(\ref{eq111}) forms an important ingredient in the proofs of various polynomial recurrence theorems such as Theorem~\ref{tw3wstep} and its far reaching extension, the polynomial Szemeredi theorem proved in \cite{BL}, the Ces\`aro limits fail to discern the more subtle behaviour of weakly mixing operators along the idempotent ultrafilters~\footnote{As it was mentioned in footnote~\ref{stopka8} (and will be stressed many more times in the sequel), $p\,-$limits of various ergodic expressions are intrinsically connected with the behaviour of these expressions along IP-sets, which, in turn, are connected with important applications of ergodic theory to combinatorics (see \cite{Be2}, \cite{Be-Fu-M}, \cite{BM}, \cite{FK}).}.

It is not hard to see that $U\in\cu(\ch)$ is weakly mixing if and only if for some idempotent $p\in\beta\N$, $p\,\text{-}\!\lim_{n\in\N}  U^n=0$. However, unlike the situation with Ces\`aro limits described by formula~(\ref{eq111}), the relation  $p\,\text{-}\!\lim_{n\in\N}  U^n=0$ does not imply, in general, neither  $p\,\text{-}\!\lim_{n\in\N} U^{2n}=0$, nor, say,  $p\,\text{-}\!\lim_{n\in\N}  U^{n^2}=0$. On the other hand, one can show that  $p\,\text{-}\!\lim_{n\in\N}  U^n=Id$ implies  $p\,\text{-}\!\lim_{n\in\N}  U^{kn}=Id$, for any $k\in\N$, and this is consistent (depending on a choice of $U$) with both  $p\,\text{-}\!\lim_{n\in\N}  U^{n^2}=0$ and
$p\,\text{-}\!\lim_{n\in\N}  U^{n^2}=Id$. It turns out (see Corollary~F below) that for any independent family $\{P_1,\ldots,P_m\}\subset \cp$~\footnote{Independence here means that if for some $n_1,\ldots,n_m\in\Z$, $\sum_{i=1}^mn_iP_i=0$, then $n_i=0$ for each $i=1,\ldots,m$.} and any choice of $E_i\in\{0,Id\}$, $i=1,\ldots, m$, there exist $U\in\cu(\ch)$ and an idempotent $p\in\beta\N$ such that $p\,\text{-}\!\lim_{n\in\N} U^{P_i(n)}=E_i$ for each $i=1,\ldots,m$.

Let us fix $U\in\cu(\ch)$ and an idempotent $p\in\beta\N$. It is not hard to show that $$\mbox{$\{P\in\cp:\: p\,\text{-}\!\lim_{n\in\N} U^{P(n)}=Id\}$ is an additive subgroup of $\cp$.}$$ One of the main results of this paper (Theorem~E below) gives a complete characterization of this kind of groups\footnote{This characterization problem is interesting only for unitary operators which are weakly mixing. Indeed, if $U\in{\cal U}({\cal H})$ has discrete
spectrum (that is, the space $\ch$ is spanned by the eigenvectors of $U$), we have $p\,\text{-}\!\lim_{n\in \N}U^{P(n)}=Id$ for each $P\in
{\cal P}$ and each $p\in \beta\N$, $p=p+p$, see
Proposition~\ref{choj9}.}, hereby contributing to the recently revived studies of the phenomenon of rigidity for weakly mixing  operators
(see  \cite{Aa-Ho-Le}, \cite{Ad}, \cite{Be-Ju-Le-Ro}, \cite {Gr},  \cite{Gr-Ei}).

Not all idempotents and not all weakly mixing unitary operators are interesting when we study such $p\,$-limits. Indeed, if $p\in \beta\N$ is a minimal
idempotent~\footnote{An idempotent $p\in \beta\N$ is said to be
{\em minimal} if it belongs to a minimal right ideal of
$(\beta\N,+)$. $U$ is weakly mixing if and only if $p\,\text{-}\!\lim_{n\in\N}  U^n=0$ for each minimal idempotent $p\in\beta\N$ \cite{Be1}. See \cite{Be1}, \cite{Be2} for the discussion of minimal idempotents and their applications to dynamics and combinatorics.} and $U$ is weakly mixing then
$p\,\text{-}\!\lim_{n\in\N} U^{kn}=0$ for each $k\geq1$ \cite{Be1}. Moreover, by Corollary~B below, we will obtain  $p\,\text{-}\!\lim_{n\in\N} U^{P(n)}=0$ for each positive
degree polynomial $P\in{\cal P}$. We should also notice that if
$U$ is {\em mildly mixing} \cite{Fu-We}, \cite{Sch},
i.e.\ when $U$ has no non-trivial rigid vectors, then
 $p\,\text{-}\!\lim_{n\in\N} U^{P(n)}=0$ for each positive
degree polynomial $P\in{\cal P}$ and every idempotent $p\in \beta\N$~\footnote{The latter is not surprising. As we have already noticed,
$p\,\text{-}\!\lim_{n\in\N}U^n={\rm proj}_{\ch_r}$.}.
Therefore the problem of calculating $p$-limits  along polynomial powers of $U$
becomes interesting if $U$ has non-trivial rigid vectors, in particular, if $U$ itself is rigid.

We now pass to a description of main results of the paper.

Given $n\in\N$, let ${\cal P}_{\leq N}$~\footnote{Similarly,
${\cal P}_{\geq N}$ denotes the set of  polynomials $P\in{\cal
P}$ satisfying deg$\,P\geq N$.} denote the (additive) group of all
polynomials $P\in\Z[x]$ with deg$\,P\leq N$ satisfying $P(0)=0$.
Our first result (proved in Section~\ref{poldec})  provides the following Hilbert space decomposition theorem for polynomial actions of unitary operators.

\vspace{2ex}

{\bf Theorem~A.} \ {\em For each
$N\geq1$, each idempotent $p\in \beta\N$ and each $U\in{\cal U}({\cal H})$ there exists a unique
decomposition \beq\label{dcm} {\cal H}=\bigoplus_{k\geq1} {\cal
H}^{(N)}_k\eeq into $U$-invariant closed subspaces such that for
each $P\in{\cal P}_{\leq N}$ and $k\geq1$ we have \beq\label{qw1}
p\,\text{-}\!\lim_{n\in\N}\left(U|_{{\cal
H}_k^{(N)}}\right)^{P(n)}=0\;\mbox{or}\;Id\eeq and, moreover,
\beq\label{qw2}\begin{array}{l} \mbox{whenever $k\neq l$,
there exists $Q\in{\cal P}_{\leq N}$}\\
\mbox{such that}\;\;p\,\text{-}\!\lim_{n\in\N}\left(U|_{{\cal
H}_k^{(N)}}\right)^{Q(n)} \neq p\,\text{-}\!\lim_{n\in\N}\left(U|_{{\cal
H}_l^{(N)}}\right)^{Q(n)}.\end{array} \eeq

Furthermore, the decomposition~(\ref{dcm}) has the following
property:
\beq\label{qw3}\begin{array}{l}
\mbox{For any $k\geq1$, if
$Q\in{\cal P}_{\leq N}$ is such that}\\
\mbox{$p\,\text{-}\!\lim_{n\in\N}\left(U|_{{\cal
H}_k^{(N)}}\right)^{sQ(n)}=0$
for each $s\in\N$}, then\\
\mbox{$p\,\text{-}\!\lim_{n\in\N}\left(U|_{{\cal
H}_k^{(N)}}\right)^{R(n)}=0$ for each $R\in{\cal P}_{\geq {\rm deg}\,Q
}$.}\end{array} \eeq}

\vspace{2ex}

\noindent
An important consequence of Theorem~A is the following result.

\vspace{2ex}

{\bf Corollary~B.} {\em Assume that $P\in \Z[x]$, $P(0)=0$ and deg$\,P=N\geq1$. Assume moreover that $p\,\text{-}\!\lim_{n\in\N}
U^{lP(n)}=0\;\;\mbox{for all}\;\;l\geq1$. Then
$p\,\text{-}\!\lim_{n\in\N}U^{Q(n)}=0$ for all $Q\in{\cal P}_{\geq N}$~\footnote{Note that in view of Theorem~E below, in general, this assertion fails if we only assume that $p\,\text{-}\!\lim_{n\in\N} U^{lP(n)}=0$ for all $1\leq l\leq L_0$. Indeed, making use of Lemma~\ref{sta4} below,  we can extend the cyclic group $H$ generated by $(L_0+1)P$ to a subgroup $G\subset\cp_{\leq N}$ (with $N={\rm deg}\,P$) of finite index in $\cp_{\leq N}$ so that $lP\notin G$ for $l=1,\ldots,L_0$.}.
}

\vspace{2ex}

Given a finite, positive Borel measure $\sigma$ on the circle $\bs^1$, we denote by $\widehat{\sigma}(n)$  its $n$-th Fourier coefficient: $\widehat{\sigma}(n):=\int_
{\bs^1}z^n\,d\sigma(z)$.
The decomposition result given in Theorem~A turns out to depend only on the maximal
spectral type of $U$. Therefore, it yields a  decomposition of any finite, positive Borel
measure on the circle:

\vspace{2ex}

{\bf Corollary~C.} \ {\em Assume that $\sigma$ is a probability
Borel measure on $\cir^1$. Let $N\in\N$ and $p\in \beta\N$, $p+p=p$.
Then there exists a unique decomposition
\beq\label{choj5}\sigma=\sum_{k\geq1}a_k\sigma^{(N)}_k\eeq such
that each $a_k>0$, $\sum_{k\geq1}a_k=1$, each $\sigma^{(N)}_k$ is also a probability Borel measure on
$\cir^1$, $\sigma_k^{(N)}\perp\sigma_l^{(N)}$ whenever $k\neq l$,
and, moreover, for
each $Q\in\Z[x]$ of degree at
most $N$ and $k\geq1$
\beq\label{choj6}\begin{array}{l}
p\,\text{-}\!\lim_{n\in\N}\widehat{\sigma}^{(N)}_k(Q(n))=0\;\;\mbox{or}\\
p\,\text{-}\!\lim_{n\in\N}\widehat{\sigma}^{(N)}_k(Q(n)-Q(0))=1.\end{array}\eeq
\beq\label{choj7}\begin{array}{l}\mbox{If $k\neq l$ then there
exists $Q\in\Z[x]$ of degree at most~$N$}\\ \mbox{such that}\;\;
p\,\text{-}\!\lim_{n\in\N}\widehat{\sigma}^{(N)}_k(Q(n))=0\;\;\mbox{and}\\
p\,\text{-}\!\lim_{n\in\N}\widehat{\sigma}^{(N)}_l(Q(n)-Q(0))=1\;\;\mbox{or
vice versa.}\end{array}\eeq}

\vspace{2ex}

Corollary~C is complemented by the following result (which can be viewed as another form of Corollary~B).

\vspace{2ex}

{\bf Corollary~D.} \
{\em Assume that $\sigma$ is a continuous probability Borel measure on
$\cir^1$ and let $p\in \beta\N$, $p+p=p$.\\
(i) If $p\,\text{-}\!\lim_{n\in\N}\widehat{\sigma}(ln+k)=0$ for each
$l\geq1$ and $k\in\Z$, then
$$p\,\text{-}\!\lim_{n\in\N}\widehat{\sigma}(Q(n))=0$$ for each positive
degree  polynomial $Q\in\Z[x]$.\\
(ii) If, for some $P\in{\cal P}_{\leq N}$, we have
$p\,\text{-}\!\lim_{n\in\N}\widehat{\sigma}(lP(n)+k)=0$ for each $l\geq1$
and $k\in\Z$, then $$p\,\text{-}\!\lim_{n\in\N}\widehat{\sigma}(Q(n))=0$$
for each $Q\in\Z[x]$ of degree not smaller than the degree of $P$.
}

\vspace{2ex}

Motivated by the decomposition result given by Theorem~A, we introduce the notion of
rigidity group. Let $N\in\N$. A subgroup $G\subset {\cal P}_{\leq N}$ is
called an $N$-{\em rigidity group} if there exist  $p\in
\beta\N$, $p+p=p$, and $U\in{\cal U}({\cal H})$  such that
$$
G=\{P\in{\cal P}_{\leq N}:\: p\,\text{-}\!\lim_{n\in\N} U^{P(n)}=Id\}$$
and $p\,\text{-}\!\lim_{n\in\N} U^{Q(n)}=0$ for each $Q\in{\cal P}_{\leq
N}\setminus G$. The second main goal of the paper is to prove the following result.

\vspace{2ex}

{\bf Theorem~E.} \ {\em
Assume that $G\subset {\cal P}_{\leq N}$ is a subgroup with
$\max\{{\rm deg}\,P:\: P\in G\}=N$. Then  $G$ is an $N$-rigidity group if and only
if $G$ has finite index in ${\cal P}_{\leq N}$.}

\vspace{2ex}

Our strategy to prove Theorem~E will be first to introduce in Section~\ref{rigiditygroups} the concept of $N$-{\em periodic} rigidity groups associated with  $G$ - these are subgroups of $\Z/k_1\Z\oplus\ldots\oplus\Z/k_N\Z$, where $k_j\geq1$ is the smallest natural number $k$ such that $kx^j\in G$, $j=1,\ldots,N$
($N$-rigidity groups will turn out to be the preimages of
$N$-periodic rigidity groups via the natural map $\Z^N\to
\Z/k_1\Z\oplus\ldots\oplus\Z/k_N\Z$)~\footnote{The numbers $k_1,\ldots,k_N$ are determined by the
fact that we consider $\{x,x^2,\ldots,x^N\}$ as the basis of
${\cal P}_{\leq N}$; if we choose a different basis in ${\cal P}_{\leq N}$,
say $Q_1,\ldots,Q_N$, and replace $x^j$ by $Q_j$ we obtain another sequence of
periods: $l_1,\ldots,l_N$, and a different $N$-periodic subgroup
$\ov{G}\subset \Z/l_1\Z\oplus\ldots\oplus\Z/l_N\Z$
satisfying the~$(\ast)$-property (see the definition of this property below), but $G$ will still be equal to
the preimage of $\ov{G}$ via the map $\Z^N\to
\Z/l_1\Z\oplus\ldots\oplus\Z/l_N\Z$, see
Section~\ref{rigiditygroups}.}. The fact that $k_j\geq1$ is well defined for
each $j=1,\ldots,N$ is not obvious and it is a consequence of the
decomposition theorem (Theorem~A)\footnote{The
decomposition theorem,  in turn, heavily depends on the $p$-limit
version of the classical van der Corput lemma, see
Lemma~\ref{vdC1} below.}. We then describe $N$-periodic rigidity
groups as duals of quotients of
$\Z/k_1\Z\times\ldots\times\Z/k_N\Z$ by the so called group couplings~\footnote{A subgroup $K\subset \Z/k_1\Z\times\ldots\times\Z/k_N\Z$
is called a {\em group coupling} if it has full projection on each
coordinate (cf.\ the notion of joining in ergodic theory).}. This will allow us to give a complete classification
of $N$-periodic rigidity groups as  groups
$\widetilde{G}\subset \Z/k_1\Z\oplus\ldots\oplus\Z/k_N\Z$ which
fulfill the following~$(\ast)$-property: For each $r=1,\ldots,N$
$$
(\ast)\;\;\;\;\;\;\;\left.\begin{array}{l}
(j_1,\ldots,j_{r-1},j_r,j_{r+1},\ldots,j_N)\in \widetilde{G}\\
(j_1,\ldots,j_{r-1},j'_r,j_{r+1},\ldots,j_N)\in
\widetilde{G}\end{array}\right\}\;\Longrightarrow\;
j_r=j'_r~\footnote{That is, $\widetilde{G}$ contains no nonzero
element of the form $(0,\ldots,0,j_r,0,\ldots,0)$.}.$$ In
particular, to prove that a group $\widetilde{G}$ satisfying
the~($\ast$)-property is an $N$-periodic rigidity group, we will
construct a {\em weighted unitary operator} $U=V^T_{\Theta\circ
\Phi}$ over an odometer $T$, where $\Phi$ is a cocycle taking values
``up to a limit distribution'' in a group coupling which will turn
out to be the dual of $\widetilde{G}$ and $\Theta$ is a character
of $\Z/k_1\Z\ot\ldots\ot\Z/k_N\Z$. For some $p\in \beta\N$, $p=p+p$, the
$N$-periodic rigidity group of $V_{\Theta\circ\Phi}^T$ will be
isomorphic to $\widetilde{G}$. The proof of Theorem~E will then
follow from some further algebraic considerations.

The following corollary of Theorem~E confirms the possibility of independent behavior of $p$-limits for
finite families of independent polynomials.

\vspace{2ex}

{\bf Corollary F.} \ {\em Assume that $P_1,\ldots,P_N\in {\cal P}_{\leq N}$ are independent. Then, for any
$s=1,\ldots,N$, there exists an $N$-rigidity subgroup $G$ containing
$P_1,\ldots,P_s$  such that $P_{s+1},\ldots,P_N\notin G$. That is, there exist $U\in\cu(\ch)$ and $p\in\beta\N$, $p+p=p$, such that $p\,\text{-}\!\lim_{n\in\N}U^{P_i(n)}=Id$ for $i=1,\ldots,s$ and $p\,\text{-}\!\lim_{n\in\N}U^{P_i(n)}=0$ for $i=s+1,\ldots,N$.}

\vspace{2ex}

Finally, we will consider groups of {\em
global rigidity}, that is, groups of the form $
G=\{P\in{\cal P}_{\leq N}:\: p\,\text{-}\!\lim_{n\in\N} U^{P(n)}=Id\}$
for some idempotent  $p\in \beta\N$ and some $U\in {\cal U}({\cal
H})$. The following result will be proved in Section~\ref{rigiditygroups}.

\vspace{2ex}

{\bf Theorem~G.} \ {\em Any subgroup $G$  of
${\cal P}_{\leq N}$ is a group of global rigidity, that
is, given a subgroup $G\subset{\cal P}_{\leq N}$,  there exist $U\in{\cal U}({\cal H})$ and $p\in \beta\N$, $p+p=p$,
such that
$$
G=\{P\in{\cal P}_{\leq N}:\: p\,\text{-}\!\lim_{n\in\N}
U^{P(n)}=Id\}.$$}

\vspace{1ex}

The structure of the paper is as follows. In Section~\ref{section:p-limits} we collect some necessary facts and provide additional details on $\beta\N$ and $p$-limits. In Section~\ref{poldec} we prove Theorem~A and derive  Corollaries~B,~C and~D. In Section~\ref{rigiditygroups} we characterize the class of groups $G\subset\cp_{\leq N}$ satisfying the~($\ast)$ property in terms of $N$-periodic rigidity, or, equivalently,  in terms of algebraic couplings (see Theorem~\ref{p-group-main}). The main part of Section~\ref{rigiditygroups} is devoted to the proof of Theorem~E. Section~\ref{rigiditygroups} will also contain the proofs of Corollary~F which is a consequence of Theorem~E, and of
Theorem~G which follows from Theorems~A and~E.

\section{Ultrafilters, $p$-convergence and
IP-limits}\label{section:p-limits} \subsection{Idempotent ultrafilters and Hindman's theorem} In the previous section we have already introduced $\beta\N$, the space of ultrafilters on $\N$.
In this section we will provide some additional discussion of results pertaining to ultrafilters which will be needed in the subsequent sections. The reader can find the missing details in \cite{Be}, \cite{Be1} and
\cite{Be2}.

As we have already mentioned, given $p\in\beta\N$, the map $q\mapsto p+q$ is continuous. In
other words, $(\beta\N,+)$ is a left topological compact
Hausdorff semigroup. We also remark that the right translations
are continuous only at principal ultrafilters.

In a general setting, when $S$ is a left topological compact
Hausdorff semigroup, the set of invertible
elements of $S$ is denoted by $I(S)$, and  the set of idempotents of $S$ is denoted by $E(S)$ (recall that $e\in
E(S)$ means that $e\cdot e=e$). According to Ellis' lemma
\cite{El}, under our assumptions on $S$, the set $E(S)$ is always
non-empty. If $I(S)\neq\emptyset$ then $I(S)\cap E(S)$ is a
singleton, namely the unit of $S$. When $S=\beta\N$, by $E(\beta\N)$ we will always mean
the set of idempotents of $(\beta\N,+)$.

Fix $k\in \N$ and $p\in E(\beta\N)$. Since
$\bigcup_{i=0}^{k-1}(k\N+i)=\N$ is a disjoint union, there is a unique $i$,
$0\leq i<k$, such that $k\N+i\in p$. Since $p+p=p$,
$$
B:=\{n\in\N:\:(k\N+i)-n\in p\}\in p.$$ It follows that
$(k\N+i)\cap B\neq\emptyset$. Take $n\in (k\N+i)\cap B$. Then for
some $r\in\N$, $n=kr+i$ and also $(k\N+i)-n\in p$. It follows
immediately that $k\N\in p$. This gives us the following useful fact: \beq\label{u1}
\mbox{If $p\in E(\beta\N)$ then $k\N\in p$ for each $k\in\N$.}\eeq

Let  $p\in E(\beta\N)$. By \cite{Hi}, Theorem 3.1 and \cite{Hia}, Theorem 3.3., each $A\in p$ must contain a set of the form
\beq\label{fs}
{\rm FS}((n_i)_{i\geq1}):=\{n_{i_1}+n_{i_2}+\ldots+n_{i_k}:\:
i_1<i_2<\ldots<i_k,k\geq1\}\subset A.~\footnote{In \cite{Hi}, Hindman proved that for any finite partition $\N=\bigcup_{i=1}^r C_i$, one of $C_i$ contains an ${\rm IP}$-set. The fact that any member of an idempotent ultrafilter has to contain an ${\rm IP}$-set (Hindman's theorem) follows from \cite{Hia}, Theorem 3.3.}\eeq In ergodic theory and
topological dynamics the sets of the form ${\rm FS}((n_i)_{i\geq1})$ are
called IP-{\em sets}. It follows that any $p$-large
set contains an IP-set.~\footnote{It is perhaps worth mentioning that the result~(\ref{u1})
immediately follows from Hindman's theorem. Indeed, among the sets $k\N+i$, $0\leq i\leq k-1$, only $k\N$ can contain an IP-set.}
On the other hand, given a sequence $(n_i)_{i\geq1}$ one can find
$p\in E(\beta\N)$ for which \beq\label{fs007}{\rm FS}(n_i)_{i\geq K}\in p\;\;\mbox{for each}\;\;K\geq1,\eeq see e.g.\ \cite{Be2}, Theorem~2.5.

\subsection{$p\,$- and IP-limits}
Assume that $X$ is a compact metric space with a metric $d$. Given
$p\in\beta\N$,  $(x_n)\subset X$ and $x\in X$ we define
$$p\,\text{-}\!\lim_{n\in\N}x_n=x$$
if for each open $U\ni x$ the set $\{n\in\N:\: x_n\in U\}$ is
$p$-large. Since $X$ is compact Hausdorff,  a $p$-limit exists and
 is unique\footnote{\label{uczech} By the universality property
of the Stone-$\check{\mbox{C}}$ech-compactification, if $X$
is a compact Hausdorff space and
$\underline{x}=(x_n)_{n\in\N}:\N\to X$, then there exists a unique
continuous extension $\beta \underline{x}:\beta\N\to X$ of
$\underline{x}$. The value of $\beta\underline{x}$ at $p\in\beta\N$ is given by the
limit of the ultrafilter $\{A\subset X:\: \underline{x}^{-1}(A)\in p\}$ of subsets of $X$.
Then $p\,\text{-}\!\lim_{n\in\N}
x_n=\beta\underline{x}(p)$.}.

Now, if $Y$ is another compact metric space and $f:X\to Y$ is
continuous then \beq\label{funkcjeip}p\,\text{-}\!\lim_{n\in\N}f(x_n)=f
\left(p\,\text{-}\!\lim_{n\in\N}x_n\right).\eeq

In view of~(\ref{u1}), whenever $p$ is an idempotent and $k\geq1$,
$$
\{n\in\N:\: x_n\in U\}\in p \Leftrightarrow \{n\in\N:\: x_n\in
U\}\cap k\N\in p
$$
$$
\Leftrightarrow \{n\in\N:\:k|n\;\mbox{and}\;\; x_n\in U\}\in p.$$
It follows that to check that $x\in X$ is the $p$-limit of a sequence $(x_n)$, where $p\in E(\beta\N)$, it is enough
to deal with numbers which are multiples of a fixed $k\geq1$. We
write this as \beq\label{u2}
p\,\text{-}\!\lim_{n\in\N}x_n=p\,\text{-}\!\lim_{k|n}x_n.\eeq

\vspace{2ex}

\begin{Remark}  We would like to stress that in general
$p\,\text{-}\!\lim_{n\in\N}x_n\neq p\,\text{-}\!\lim_{n\in\N} x_{kn}$. For example, we do not have equality for $k>1$ when  $x_n=U^n$, where
$U$ is a unitary operator on a Hilbert space  $\ch$ and $X$ is the unit ball of the space of linear bounded operators on $\ch$ equipped with the weak operator topology,
see Theorem~E (the group $G:=\{mkx:\:m\in\Z\}\subset\cp_1$ is of finite index in $\cp_1$, hence $G$ is a 1-rigidity group and $x\notin G$).
In view of~(\ref{u2}), it follows that it is not true, in general, that
$p\,\text{-}\!\lim_{k|n}x_n=p\,\text{-}\!\lim_{n\in\N} x_{kn}$.\end{Remark}

\vspace{2ex}

We  now introduce a related notion of convergence, namely that of
IP-convergence. The precise connection between these two notions is given
by Lemma~\ref{plimit} below.

Let $\cf$ denote the family of finite non-empty subsets of $\N$.
Given an increasing sequence $(n_i)$ of natural numbers,  for each
$\alpha\in\cf$ we set
$$
n_\alpha=\sum_{i\in\alpha}n_i.$$ Assume that $(x_n)\subset X$ and
$x\in X$. Assume moreover that for each $\vep>0$ there exists
$N\geq1$ such that for each $\alpha\in\cf$ satisfying
$\min\alpha\geq N$ we have
$$
d(x_{n_\alpha},x)<\vep.$$ Then we say that $x$ {\em is the
IP-limit  of}
$(x_{n_\alpha})_{\alpha\in\cf}$ and  write
$${\rm IP}\,\text{-}\,\lim x_{n_\alpha}=x.~\footnote{\label{ordinaryc} Note that a necessary (but, in general, not sufficient) condition
for the validity of ${\rm IP}\,\text{-}\,\lim x_{n_\alpha}=x$ is that $x_{n_i}\to x$ when $i\to\infty$.}$$

\begin{Lemma}\label{plimit}
(i) Assume that $(n_i)$ is an increasing sequence of natural
numbers. Then there exists $p\in E(\beta\N)$ such that
for each compact metric space $(X,d)$ and a sequence $(x_n)_{n\geq1}\subset X$ such that ${\rm IP}\,\text{-}\!\lim
x_{n_\alpha}=x$, we have $x=p\,\text{-}\!\lim_{n\in
\N}x_n$.

(ii) Let  $(X,d)$ be a compact metric space and $(x_n)_{n\geq1}\subset X$. Assume that $p\in E(\beta\N)$ and $p\,\text{-}\!\lim_{n\in\N} x_n=x$.
Then there exists an ${\rm IP}$-set ${\rm FS}((n_i)_{i\geq1})$ such that ${\rm
IP}\,\text{-}\!\lim x_{n_\alpha}=x$.
\end{Lemma}

\begin{proof} (i) \ In view of~(\ref{fs007}), there exists
$p\in E(\beta\N)$ such that  ${\rm FS}((n_i)_{i\geq K})\in p$ for each $K\geq1$. It follows that for each $\alpha_0\in\cf$, $
\{n_\alpha:\: \max\alpha_0\leq\min\alpha\}\in p$. Now, by
assumption, given $\vep>0$ there exists $\alpha_0\in\cf$ such that
$$
\{n\in\N:d(x_n,x)<\vep\}\supset \{n_\alpha:\:
\max\alpha_0\leq\min\alpha\}\in p$$ and therefore $\{n\in\N:\:
d(x_n,x)<\vep\}\in p$.

(ii) \ Denote $A_k=\{n\in\N:\:d(x_n,x)<\frac1k\}$. We have $A_k\in
p$ for all $k\geq1$. Since $p+p=p$ and each set $A_k$ is infinite (by~(\ref{fs})), we can choose $n_1<n_2<\ldots$ such that
$n_k\in B_k, B_k-n_k\in p$ for each $k\geq 1$, where
$$
B_1:=A_1\;\;\mbox{and}\;\; B_k:=A_k\cap(B_{k-1}-n_{k-1})\cap\ldots\cap(B_1-n_1)\;\;\mbox{for each}\;\;k\geq2.$$ It is now clear that for the IP-set $FS((n_i))$ we have ${\rm IP}\,\text{-}\!\lim x_{n_\alpha}=x$.
\end{proof}

We will also need the following fact.

\begin{Lemma}[see \cite{Be}, Theorem 3.8]\label{plimit-idemp}For each
$p,q\in\beta\N$ and $(x_n)\subset X$
$$
(p+q)\,\text{-}\!\lim_{n\in \N}x_n=q\,\text{-}\!\lim_{k\in\N}\left(p\,\text{-}\!
\lim_{l\in\N}x_{k+l}\right).$$ In particular, if $p\in
E(\beta\N)$, $p\,\text{-}\!\lim_{n\in \N}x_n=p\,\text{-}\!\lim_{k\in\N}
\left(p\,\text{-}\!\lim_{l\in\N}x_{k+l}\right)$.
\end{Lemma}

\subsection{$p$-limits in semitopological
compactifications of $\Z$}

Assume now that $S$ is a compact  metric semitopological
(i.e.\ left- and right topological\footnote{Recall
that for $(\beta\N,+)$ the addition~$+$ is left-continuous and is
not right-continuous. So, $(\beta\N,+)$ is a left topological semigroup, but not a semitopological semigroup.}) semigroup. In view of~(\ref{funkcjeip})
and the continuity of left- and right- translations, it follows
that whenever $(s_n)_{n\geq1}\subset S$ and $u\in S$
\beq\label{pmn101} p\,\text{-}\!\lim_{n\in\N}
(s_nu)=\left(p\,\text{-}\!\lim_{n\in\N}
s_n\right)u\;\;\mbox{and}\;\;p\,\text{-}\!\lim_{n\in\N}
(us_n)=u\left(p\,\text{-}\!\lim_{n\in\N} s_n\right).\eeq

Choose $s\in S$ and consider $s_n=s^n$,  $n\geq1$. By the
universality property of $\beta\N$ (see footnote~\ref{uczech}), it
follows that the  map $\beta\N\supset\N \ni n\mapsto s^n\in S$ has
a  unique extension to a surjective continuous semigroup
homomorphism of $(\beta\N,+)$ onto $\ov{\{s^n:\:n\in\N\}}$, the
extension being given by the formula $p\mapsto
p\,\text{-}\!\lim_{n\in\N}s^n$~\footnote{Indeed, it follows from
Lemma~\ref{plimit-idemp} (and~(\ref{pmn101})) that setting
$s_n=s^n$ we obtain that
$$
(p+q)\,\text{-}\!\lim_{n\in\N}s_n=q\,\text{-}\!\lim_{n\in\N}
\left(p\,\text{-}\!\lim_{m\in\N}s_{n+m}\right)=
q\,\text{-}\!\lim_{n\in\N}\left(p\,\text{-}\!\lim_{m\in\N}s_n\cdot
s_m\right)$$$$=q\,\text{-}\!\lim_{n\in\N}\left(s_n\cdot
\left(p\,\text{-}\!\lim_{m\in\N}s_m\right)\right)=
\left(q\,\text{-}\!\lim_{n\in\N}s_n\right)\cdot
\left(p\,\text{-}\!\lim_{n\in\N}s_n\right).$$}. Since it is a homomorphism, we have
the following.

\begin{Lemma}\label{plimit-proj} Assume that
$S$ is a compact metric semitopological semigroup. For each $p\in
E(\beta\N)$, $t:=p\,\text{-}\!\lim_{n\in\N}s^n$ is an idempotent of $S$,
i.e.\ $t\in E(S)$.
\end{Lemma}

From now on, we assume that \beq\label{mainass} \mbox{$S$ is a metrizable
semitopological compactification of $\Z$}.\eeq In other words, we
assume that $S$ is a compact metric Abelian semitopological
semigroup having a dense cyclic subgroup, that is, for some $u\in I(S)$,
$S=\ov{\{u^n:n\in\Z\}}$.
The unit of $S$ will be denoted by~$1$. The following statement
follows from Theorem~3.2 in \cite{La}.
\beq\label{add1}\mbox{Multiplication is jointly continuous at each
point $(i,s)\in I(S)\times S$.}\eeq

%\begin{Remark}\label{pproduktowa} Note that if $S$
%satisfies~(\ref{add1}) then for the product semitopological
%semigroups  $S^{\times N}$, $N\geq 2$, and $S\times
%S\times\ldots$,~(\ref{add1}) is also satisfied.
%\end{Remark}

\begin{Lemma}\label{plimit-ciaglosc} Assume that $S$ is a compact
metric Abelian semitopological semigroup satisfying~(\ref{add1}).
Let $p\in \beta\N$ and $p\,\text{-}\!\lim_{n\in\N} s_n=s$  and
$p\,\text{-}\!\lim_{n\in\N} t_n= t$, where $t\in I(S)$. Then
\beq\label{pmnozenie} p\,\text{-}\!\lim_{n\in\N} s_nt_n=st.\eeq If $p\in
E(\beta\N)$ and $u\in I(S)$ then $p\,\text{-}\!\lim_{n\in\N} u^n=1$ and
$p\,\text{-}\!\lim_{n\in\N} s_nu^n=s$.
\end{Lemma}
\begin{proof} According to~(\ref{add1}), given $\vep>0$ there exists $\delta>0$
such that if $d(s',s)<\delta$ and $d(t',t)<\delta$ then
$d(s't',ts)<\vep$. It follows that
$$\{n\in\N:\:d(s_nt_n,st)<\vep\}\supset\{n\in\N:\:d(t_n,t)<\delta\}\cap\{n\in\N:\:
d(s_n,s)<\delta\}.$$ As this intersection is still a $p$-large
set, so is the set $\{n\in\N:\:d(s_nt_n,st)<\vep\}$
and~(\ref{pmnozenie}) follows.

To complete the proof notice that if $p$ is additionally an
idempotent  then, in view of Lemma~\ref{plimit-proj},
$p\,\text{-}\!\lim_{n\in\N} u^n=t$ is an idempotent which, by assumption,
is in $I(S)$. Hence $t=1$. Finally, using~(\ref{pmnozenie}),
$p\,\text{-}\!\lim_{n\in\N} s_nu^n=s\cdot 1=s$.
\end{proof}

Now we consider sequences of the form $(s^{P(n)})$, where
$P\in{\cal P}$, and study their $p$-limits for $p\in E(\beta\N)$.

We will need the following elementary observation on integer coefficient
polynomials.

\begin{Lemma}\label{dop} Let $P\in{\cal P}$ be of degree $d\geq1$.
Then\beq\label{pdecomp1} P(x+y)-P(x)=Q(x,y)+P(y),\eeq where
$Q(x,y)\in \Z[x,y]$ is a polynomial divisible by $xy$ in the ring
$\Z[x,y]$. The $x$-degree\footnote{We assume that the degree
of~$0$ is equal to $0$.} and the $y$-degree of $Q(x,y)$ are equal
to $d-1$.~\footnote{The
result is immediate for monomials, hence it holds for all
polynomials $0\neq P\in{\cal P}$ (we recall that for any $P\in\cp$, we have $P(0)=0$,
so $\cp$ does not contain non-trivial constant polynomials).}
\end{Lemma}

\begin{Lemma}\label{idempotent} Assume that $S$ satisfies~(\ref{add1}) and
and let $s\in I(S)$. Let $P\in{\cal P}$ be a polynomial of degree
$d\ge 1$. Assume that for $j=1,\ldots,d-1$ there exists $r_j\ge 1$
such that $$p\,\text{-}\!\lim\limits_{n\in \N} s^{r_jn^j}=1.$$ Then
$p\,\text{-}\!\lim\limits_{n\in \N} s^{P(n)}\in E(S)$.
\end{Lemma}

\begin{proof} First, notice that in view of~(\ref{add1}) and~(\ref{pmnozenie}),
we have \beq\label{periodl}
 p\,\text{-}\!\lim\limits_{n\in \N}s^{lr_jn^j}=1
\eeq for each $l\ge 1$ and $j=1,\ldots,d-1$.

By Lemmata \ref{plimit-idemp} and \ref{dop}
$$
u:=p\,\text{-}\!\lim\limits_{n\in \N}s^{P(n)}=p\,\text{-}\!\lim\limits_{n\in
\N}(p\,\text{-}\!\lim\limits_{m\in \N}s^{P(n+m)})= p\,\text{-}\!\lim\limits_{n\in \N}
s^{P(n)}(p\,\text{-}\!\lim\limits_{m\in \N}(s^{(Q(n,m)+P(m)}),
$$
where $Q(x,y)\in\Z[x,y]$ is divisible by $xy$. Set
$r={\rm lcm}(r_1,\ldots,r_{d-1})$. Using (\ref{u2}), Lemma
\ref{plimit-ciaglosc} and (\ref{periodl}), we obtain that
$$
u=p\,\text{-}\!\lim\limits_{r|n} s^{P(n)}(p\,\text{-}\!\lim\limits_{m\in
\N}(s^{(Q(n,m)}\cdot s^{P(m)})=p\,\text{-}\!\lim\limits_{r|n} s^{P(n)}(1\cdot
u)=u^2.
$$
\end{proof}

Lemma \ref{idempotent} applied in the case $d=1$ shows  that
$p\,\text{-}\!\lim\limits_{n\in\N} s^{an}\in E(S)$ for any $a\in\Z$.

\subsection{From ordinary convergence to IP- and $p$-limits}

Let $S$ be a compact metric Abelian semitopological semigroup with
$I(S)\neq\emptyset$.

Assume that we have a countable collection of polynomials
$P_i\in{\cal P}$, $i\geq 1$. Assume moreover that for some
sequence $(q_n)\subset\N$ such that for each $i\geq1$ we have
$s^{P_i(q_n)}\to e_i$ in $S$, where $e_i\in E(S)$. Can we find
$p\in E(\beta\N)$ such that $p\,\text{-}\!\lim_{n\in\N}s^{P_i(n)}=e_i$
for each $i\geq1$? The proposition below gives a list of
conditions that will guarantee the positive answer to this
question.

\begin{Prop}\label{plimitpol} Fix $N\geq1$ and assume
that $0\neq P_i\in{\cal P}_{\leq N}$, $i\geq 1$. Assume that $S$
satisfies~(\ref{add1}) and let $s\in I(S)$. Let $(q_n)$ be an
increasing sequence of natural numbers such that for some
$r_j\geq1$, $j=1,\ldots,N-1$, \beq\label{plimitpol1}
s^{r_jq_n^j}\to 1.\eeq Denote $r={\rm lcm}(r_1,\ldots,r_{N-1})$
and assume in addition that \beq\label{plimitpol2}
r|q_n\;\;\mbox{for}\;\;n\geq n_0.\eeq Finally, assume that for
each $i\geq1$ \beq\label{plimitpol3}s^{P_i(q_n)}\to e_i\in
E(S).\eeq Then there exists $p\in E(\beta\N)$ such that for each
$i\geq1$ \beq\label{plimitpol4}
p\,\text{-}\!\lim_{n\in\N}s^{P_i(n)}=e_i.\eeq
\end{Prop}
\begin{proof} We divide the proof into two parts. In the first part we assume that $P_1=P_2=\ldots$ In the second part, using a diagonalization procedure and the first part of the proof, we complete the proof.

 {\bf Part 1.} \ We will first prove the result when the family
$\{P_i:\:i\geq1\}$ consists of one polynomial $P=P_1=P_2=\ldots$ and we set
$e_1=e\in E(S)$. Assume, without loss of generality, that ${\rm deg}\,P=N$ and write $P(x)=M_Nx^N+
\ldots+M_2x^2+M_1x\in\Z[x]$. First, choose $k_1>n_0$ so that
$$
d\left(s^{P(q_{k_1})},e\right)<\frac1{2^1},$$
$$
d\left(s^{P(q_{k_1})}e,e\right)<\frac1{2^1},$$ which is possible
by letting $k_1\to\infty$ in~(\ref{plimitpol3}) and using the
semicontinuity of the multiplication (and the fact that $e\cdot e=e$).

Suppose that for some $w\geq1$ the numbers $k_1<k_2<\ldots<k_w$
have already been selected so that for all choices
$i_1<\ldots<i_t$ with $i_1,\ldots,i_t\in\{1,\ldots,w\}$,
\beq\label{sq1} d\left(s^{P(q_{k_{i_1}}+\ldots+
q_{k_{i_t}})},e\right)<\frac1{2^{i_1}},\eeq \beq\label{sq2}
d\left(s^{P(q_{k_{i_1}}+\ldots+
q_{k_{i_t}})}e,e\right)<\frac1{2^{i_1}}.\eeq

We have now to select $k_{w+1}$. If we  show that this choice
depends only on the fact that $k_{w+1}$ is sufficiently large then
we are done because we deal with a bounded number of indices
$i_1<\ldots<i_t$.

Given $\alpha\subset\{1,\ldots,w\}$ ($\alpha$ may be empty), we
set $a_\alpha=\sum_{j\in\alpha}q_{k_j}$ ($a_\emptyset=0$). In view
of Lemma~\ref{dop}, we obtain
$$
s^{P(a_\alpha+q_{k_{w+1}})}=s^{P(a_\alpha)}
s^{Q(a_\alpha,q_{k_{w+1}})}s^{P(q_{k_{w+1}})}.$$ The polynomials
$R\in{\cal P}$, for which $s^{R(q_n)}\to 1$, form a subgroup of
$\cal P$. Since $r|a_\alpha$, $Q(x,y)$ is divisible by $x$ and has
$y$-degree $N-1$, it follows by~(\ref{plimitpol1})
and~(\ref{plimitpol3}) that
$$
s^{Q(a_\alpha,q_{k_{w+1}})}\to 1$$ and
$$
s^{Q(a_\alpha,q_{k_{w+1}})}s^{P(q_{k_{w+1}})}\to e,
$$
$$
s^{Q(a_\alpha,q_{k_{w+1}})}s^{P(q_{k_{w+1}})}e\to e$$ when
$k_{w+1}\to\infty$. In view of~(\ref{sq1}) and~(\ref{sq2}), there
exists $\vep>0$ such that
$$
d(s^{P(a_\alpha)},e)<\frac1{2^{\min\alpha}} -\vep,$$
$$
d(s^{P(a_\alpha)}e,e)<\frac1{2^{\min\alpha}} -\vep$$ for each
$\emptyset\neq\alpha\subset\{1,\ldots,w\}$. By the semicontinuity of
multiplication, we can find $\delta>0$ such that whenever
$d(s',s^{\prime\prime})<\delta$,
$$
d(s^{P(a_\alpha)}s',s^{P(a_\alpha)}s^{\prime\prime})<\vep.$$
Select $k_{w+1}$ large enough to satisfy
$$
d(s^{Q(a_\alpha,q_{k_{w+1}})}s^{P(q_{k_{w+1}})}, e)<\delta
$$
and
$$
d(s^{Q(a_\alpha,q_{k_{w+1}})}s^{P(q_{k_{w+1}})}e, e)<\delta$$ for
each $\emptyset\neq\alpha\subset\{1,\ldots,w\}$. Then
$$
d(s^{P(a_\alpha+q_{k_{w+1}})},e)\leq
d(s^{P(a_\alpha)}s^{Q(a_\alpha,q_{k_{w+1}})}s^{P(q_{k_{w+1}})},
s^{P(q_{k_{w+1}})}e)$$
$$
+d(s^{P(q_{k_{w+1}})}e,e)<\vep+\left(
\frac1{2^{\min\alpha}}-\vep\right)=\frac1{2^{\min\alpha}}$$ and
similarly
$$d(s^{P(a_\alpha+q_{k_{w+1}})}e,e)< \frac1{2^{\min\alpha}}.
$$
for each $\emptyset\neq\alpha\subset\{1,\ldots,w\}$. We can also
assume that $k_{w+1}$ yields
$d(s^{P(q_{k_{w+1}})},e)<\frac1{2^{w+1}}$ and
$d(s^{P(q_{k_{w+1}})}e,e)<\frac1{2^{w+1}}$ which covers the case
$\alpha=\emptyset$.

We have proved that our recurrence procedure can be continued. In
view of~(\ref{sq1}), $\mbox{IP}-\lim s^{P((q_{k_w})\alpha)}=e$. We
use now Lemma~\ref{plimit} to complete the proof in the case of
one-element family of polynomials.

{\bf Part 2.} By Part 1 of the proof, we can select a subsequence
$(q_{k_w^{(1)}})_{w\geq1}$ of $(q_n)_{n\geq1}$, so that $${\rm IP}\,\text{-}\!\lim
s^{P_1((q_{k_w^{(1)}})_\alpha)}=e_1$$ (that is, the
${\rm IP}$-convergence holds along the ${\rm IP}$-set
${\rm FS}((q_{k_w^{(1)}})_{w\geq1})$). At stage $j+1$ (using repeatedly
Part~1 of the proof), we select a subsequence
$(q_{k_w^{(j+1)}})_{w\geq1}$ of $(q_{k_w^{(j)}})_{w\geq1}$, so
that
$${\rm IP}\,\text{-}\!\lim
s^{P_{j+1}((q_{k_w^{(j+1)}})_\alpha)}=e_{j+1}$$ (that is, the
${\rm IP}$-convergence holds along the sub-${\rm IP}$-set
${\rm FS}((q_{k_w^{(j+1)}})_{w\geq1})$ of
${\rm FS}((q_{k_w^{(j)}})_{w\geq1})$).

Now, for each $i\geq 1$, we set $q_{m_i}=q_{k^{(i)}_i}$. It
follows that for each $j\geq1$, up to a finite number of terms,
$(q_{m_i})_{i\geq1}$ is a subsequence of
$(q_{k_w^{(j)}})_{w\geq1}$. Therefore, whenever we have an
${\rm IP}$-convergence along ${\rm FS}((q_{k_w^{(j)}})_{w\geq1})$, we also
have the ${\rm IP}$-convergence along ${\rm FS}((q_{m_i})_{i\geq1})$ (to the
same limit). We conclude that for each $j\geq1$,
$$
{\rm IP}\,\text{-}\!\lim s^{P_j((q_{m_i})_\alpha)}=e_j$$ and the result follows
from Lemma~\ref{plimit}~(i).
\end{proof}

\begin{Remark} We would like to emphasize that the
assumption~(\ref{plimitpol1}) in
Proposition~\ref{plimitpol} seems to be crucial. Indeed, in the proof of Proposition~\ref{plimitpol}, due to~(\ref{plimitpol1}), we are able to choose a subsequence $(q_{k_n})_{n\geq1}$ so that ${\rm IP}\,\text{-}\!\lim s^{P_i(q_{k_n})}=e_i$ for each $i\geq1$ and then we obtain a ``good'' $p\in E(\beta\N)$ using Lemma~\ref{plimit}.  However, in general, for a $(q_n)$ growing rapidly to infinity, it is even possible to have $U\in{\cal U}({\cal H})$ such that $U^{lq_n}\to
0$ for each $l\geq1$ and $U^{q_n^2}\to Id$. Then, for any subsequence $(q_{k_n})_{n\geq1}$ we have the same  convergence but we cannot have the IP-convergence as this, by Lemma~\ref{plimit}, yields $p\in
E(\beta\N)$ with $p\,\text{-}\!\lim_{n\in\N} U^{ln}=0$ for each $l\geq1$
and $p\,\text{-}\!\lim_{n\in\N} U^{n^2}=Id$, a contradiction with Corollary~B.
\end{Remark}

\section{$p$-polynomial decomposition of a Hilbert
space}\label{poldec}

\subsection{Some immediate consequences of Theorem~A}\label{poldec2013} Throughout this section we fix $p\in E(\beta\N)$. Before we start the proof of Theorem~A, let us make first some introductory
remarks and derive some simple consequences of it.

Suppose that we have a decomposition~(\ref{dcm}) satisfying
only~(\ref{qw1}). Then if for $k,l\in\N$ we have
$$
p\,\text{-}\!\lim_{n\in\N}\left(U|_{{\cal H}_k}\right)^{P(n)}
=p\,\text{-}\!\lim_{n\in\N}\left(U|_{{\cal H}_l}\right)^{P(n)}
$$
for each polynomial $P\in{\cal P}_{\leq N}$ (call such $k,l$ {\em
equivalent}), then we can replace both subspaces ${\cal
H}^{(N)}_k$ and ${\cal H}^{(N)}_l$ by one subspace ${\cal
H}^{(N)}_k\oplus {\cal H}^{(N)}_l$ and still have a decomposition
satisfying~(\ref{qw1}). By grouping up subspaces whose indices are
equivalent, we achieve a decomposition~(\ref{dcm}) in which
additionally to~(\ref{qw1}) we have~(\ref{qw2}).

Now, we claim that if a decomposition~(\ref{dcm})
satisfying~(\ref{qw1}) and~(\ref{qw2}) exists then it is unique.
Indeed, it is enough to show that
\beq\label{doda1}\begin{array}{l}\mbox{if $\cal F$ is a closed
$U$-invariant subspace such that~(\ref{qw1}) is satisfied on it,}\\
\mbox{then there is $k\geq1$ such that ${\cal F}\subset {\cal
H}^{(N)}_k$.}\end{array}\eeq
To prove this claim, suppose that for no $k\geq1$, $\cal F$ is included in ${\cal H}_k^{(N)}$. It follows that for some $x\in {\cal F}$,
$x=\sum_{i\geq1} x_i$ with $x_i\in {\cal H}^{(N)}_i$ ($i\geq1$)
there are $i_1\neq i_2$ such that $x_{i_1}\neq0\neq x_{i_2}$. By interchanging the roles of $x_{i_1}$
and $x_{i_2}$, if necessary, by~(\ref{qw2}), we can assume that there exists
$Q\in{\cal P}_{\leq N}$ such that
$$
p\,\text{-}\!\lim_{n\in\N}U^{Q(n)}x_{i_1}=0,\;
p\,\text{-}\!\lim_{n\in\N}U^{Q(n)}x_{i_2}=x_{i_2}.$$ On the other hand, we have
either
$$
p\,\text{-}\!\lim_{n\in\N}U^{Q(n)}x=x\;\;\mbox{or}\;\;
p\,\text{-}\!\lim_{n\in\N}U^{Q(n)}x=0.$$ In the first case, we obtain
$$
\sum_{i\geq 1}x_i=x=p\,\text{-}\!\lim_{n\in\N}U^{Q(n)}x= \sum_{i\geq1}
p\,\text{-}\!\lim_{n\in\N}U^{Q(n)}x_i,$$ where
$p\,\text{-}\!\lim_{n\in\N}U^{Q(n)}x_i\in{\cal H}^{(N)}_i$, $i\geq 1$. In particular, $p\,\text{-}\!\lim_{n\in\N}U^{Q(n)}x_{i_k}=x_{i_k}$ for $k=1,2$ and we obtain contradiction. Similarly, we obtain a contradiction in the second case and therefore~(\ref{doda1}) follows.

The following result is a direct consequence of~(\ref{doda1}).

\begin{Cor} \label{specsub} The subspaces ${\cal H}^{(N)}_k$ in the
decomposition~(\ref{dcm}) are {\em spectral subspaces}, i.e.\ if
$y\in{\cal H}$ and the spectral measure\footnote{The spectral
measure of $y\in {\cal H}$ is the unique finite, positive Borel
measure $\sigma_y$ on $\cir^1$ satisfying
$\widehat{\sigma}_y(n)=\langle U^ny,y\rangle$ for all $n\in\Z$;
see e.g.\ \cite{Gl}, Chapter~5.} of $y$ is absolutely continuous with
respect to the maximal spectral type of $U|_{{\cal H}^{(N)}_k}$
then $y\in {\cal H}^{(N)}_k$.\end{Cor}

\begin{Remark}\label{decmin} It also follows from~(\ref{doda1}) that the decomposition in
Theorem~A is minimal in the
sense that any decomposition satisfying~(\ref{qw1}) must be a
refinement of the one of that theorem.

By the minimality of the decomposition in Theorem~A, it follows that if $1\leq M<N$ and
$$
{\cal H}=\bigoplus_{l\geq1} {\cal H}^{(M)}_l=\bigoplus_{k\geq1}
{\cal H}^{(N)}_k$$ are decompositions given by
Theorem~A for $M$ and $N$ respectively then for each
$k\geq1$ there exists a unique $l_k\geq1$ such that ${\cal
H}^{(N)}_k\subset{\cal H}^{(M)}_{l_k}$.\end{Remark}

\vspace{2ex}

\begin{Cor}\label{misg} Assume that $G$ is  a finitely generated
subgroup of the group of all polynomials $P\in{\cal P}$. Let
$U\in{\cal U}({\cal H})$. Then there exists a unique decomposition
${\cal H}=\bigoplus_{k=1}^\infty {\cal H}^{(G)}_k$ into closed
$U$-invariant subspaces such that~(\ref{qw1}) and~(\ref{qw2}) hold
with ${\cal P}_{\leq N}$ replaced by $G$.
\end{Cor}
\begin{proof}Since $G$ is finitely generated, $G\subset
{\cal P}_{\leq N}$ for some $N\geq1$. Apply
Theorem~A to obtain a decomposition for
which~(\ref{qw1}) is satisfied when ${\cal P}_{\leq N}$ is
replaced by $G$ and then glue up subspaces which cannot be
distinguished by polynomials belonging do $G$ (similarly to the
proof of Theorem~A to obtain the required
decomposition).\end{proof}

\noindent {\bf Open problem:} We have been unable to decide whether the
decomposition described in Theorem~A exists for
arbitrary infinite family of polynomials in $\cal P$. We conjecture
that the answer is negative already for the family of all
monomials.

\subsection{Proof of Theorem~A}

Our proof is based on two results: Theorem~\ref{tw5wstep}  and the following lemma.

\begin{Lemma}[\cite{BM1}]\label{vdC1} Assume that
$(x_n)\subset {\cal H}$ is bounded. If
$$
p\,\text{-}\!\lim_{h\in\N}\left(p\,\text{-}\!\lim_{n\in\N}\langle
x_{n+h},x_n\rangle\right)=0$$ then $p\,\text{-}\!\lim_{n\in\N} x_n=0$.
\end{Lemma}

Let us start the proof with verifying the validity of Theorem~A for $N=1$.
Denote by ${\cal H}_1$ the subspace uniquely determined by
$p\,\text{-}\!\lim_{n\in\N} U^n={\rm proj}_{{\cal H}_1}$ (see
Theorem~\ref{tw5wstep}). Then ${\cal H}={\cal H}_0\oplus{\cal H}_1$
with both subspaces being $U$-invariant. Moreover,
$p\,\text{-}\!\lim_{n\in\N} \left(U|_{{\cal H}_0}\right)^n=0$ and
$p\,\text{-}\!\lim_{n\in\N} \left(U|_{{\cal H}_1}\right)^{n}=Id$. The latter equality implies
$$
p\,\text{-}\!\lim_{n\in\N} \left(U|_{{\cal
H}_1}\right)^{kn}=Id\;\;\mbox{for each}\;\;k\geq1.$$ We now apply the same argument
to $\left(U|_{{\cal H}_0}\right)^2$. We obtain a
decomposition ${\cal H}_0={\cal H}_{00}\oplus{\cal H}_{01}$ such
that $p\,\text{-}\!\lim_{n\in\N} \left(U|_{{\cal H}_{01}}\right)^{2n}=Id$
and $p\,\text{-}\!\lim_{n\in\N} \left(U|_{{\cal H}_{00}}\right)^{2n}=0$.
It follows that
$$
p\,\text{-}\!\lim_{n\in\N} \left(U|_{{\cal H}_{01}}\right)^{n}=0,
\;p\,\text{-}\!\lim_{n\in\N} \left(U|_{{\cal H}_{01}}\right)^{2n}=Id.$$
These two conditions determine the $p$-limits of
$\left(U^{kn}|_{{\cal H}_{01}}\right)_n$ for each $k\geq1$.
Indeed, if $k=2l$ then $p\,\text{-}\!\lim_{n\in\N} \left(U|_{{\cal
H}_{01}}\right)^{2ln}=Id$ and if $k=2l+1$ then by
Lemma~\ref{plimit-ciaglosc}
$$
p\,\text{-}\!\lim_{n\in\N} \left(U|_{{\cal
H}_{01}}\right)^{kn}=p\,\text{-}\!\lim_{n\in\N} \left(U|_{{\cal
H}_{01}}\right)^{2ln+n}=Id\circ 0=0.$$ On $H_{00}$ we have
$p\,\text{-}\!\lim_{n\in\N} \left(U|_{{\cal H}_{00}}\right)^{in}=0$ for
$i=1,2$. Applying the above argument  again  to
$\left(U|_{H_{00}}\right)^3$, we obtain  ${\cal H}_{00}={\cal
H}_{000}\oplus{\cal H}_{001}$ with
$$
p\,\text{-}\!\lim_{n\in\N} \left(U|_{{\cal H}_{001}}\right)^{3ln}=Id,\;
p\,\text{-}\!\lim_{n\in\N} \left(U|_{{\cal H}_{001}}\right)^{(3l+1)n}=0$$
and
$$p\,\text{-}\!\lim_{n\in\N} \left(U|_{{\cal H}_{001}}\right)^{(3l+2)n}=0$$ for
each $l\geq0$. On $H_{000}$ we have $p\,\text{-}\!\lim_{n\in\N}
\left(U|_{{\cal H}_{000}}\right)^{in}=0$ for $i=1,2,3$.  Continuing in this fashion, we will obtain a sequence
$$
{\cal H}_0\supset\ldots \supset{\cal H}_{0^k1}\supset {\cal
H}_{0^{k+1}1}\supset\ldots$$
of nested closed $U$-invariant subspaces, where
 $0^k1$ stands for
$\underbrace{0\ldots0}_k1$, $k\geq1$. Let ${\cal
H}_\infty=\left(\bigoplus_{k\geq0}{\cal H}_{0^k1}\right)^\perp$.
By construction, $U|_{{\cal H}_\infty}$ is totally
$p$-mixing\footnote{A unitary operator $V\in{\cal U}({\cal G})$ is
called {\em totally $p$-mixing} if $p\,\text{-}\!\lim_{n\in\N}V^{ln}=0$
for each $l\geq1$.} on ${\cal H}_\infty$ (indeed, if
$p\,\text{-}\!\lim_{n\in\N} \left(U|_{{\cal H}_\infty}\right)^n={\rm
proj}_{\cal F}$ then ${\cal F}\subset {\cal H}_1$ whence ${\cal
F}=\{0\}$; if $p\,\text{-}\!\lim_{n\in\N} \left(U|_{{\cal
H}_\infty}\right)^n=0$ and if $p\,\text{-}\!\lim_{n\in\N} \left(U|_{{\cal
H}_\infty}\right)^{2n}={\rm proj}_{\cal F}$ then ${\cal
F}\subset{\cal H}_{01}$ whence ${\cal F}=\{0\}$, etc.). We have
proved the following.
\begin{Lemma}\label{pd1} If $U\in{\cal U}({\cal H})$ and $p\in
E(\beta\N)$ then there exists a  decomposition
$$
{\cal H}={\cal H}_\infty\oplus\bigoplus_{k=0}^\infty {\cal
H}_{0^k1}$$ into $U$-invariant subspaces such that $U$ is totally
$p$-mixing on ${\cal H}_\infty$ and
$$
p\,\text{-}\!\lim_{n\in\N} \left(U|_{{\cal H}_{0^k1}}\right)^{rn}=
\left\{\begin{array}{lll} Id & \mbox{if} & (k+1)|r\\
0&\mbox{ \ } &\mbox{otherwise.}\end{array}\right.$$
\end{Lemma}

The following result is needed to complete the inductive  proof of the case $N=1$.

\begin{Lemma}\label{vitaly} Let $p\in E(\beta\N)$. Assume
that $U\in{\cal U}({\cal H})$ and $k\geq1$. Assume moreover that
for some $r_j\in\N$ \beq\label{uu1} p\,\text{-}\!\lim_{n\in\N}
U^{r_jn^j}=Id\eeq for
%$s=1,\ldots,r_j-1$,
$j=1,\ldots,k-1$ and
\beq\label{uu2} p\,\text{-}\!\lim_{n\in\N} U^{ln^k}=0\;\;\mbox{for
all}\;\;l\geq1. \eeq Then \beq\label{uu3} p\,\text{-}\!\lim_{n\in\N}
U^{P(n)}=0\eeq for all polynomials $P\in{\cal P}_{\geq k}$.
\end{Lemma}
\begin{proof} Let $y\in {\cal H}$ and let
$P\in{\cal P}$ be a polynomial of degree~$k+1$.  By Lemma~\ref{dop}, we have
$$
\langle U^{P(n+h)}y, U^{P(n)}y\rangle =\langle
U^{Q(n,h)}y,U^{-P(h)}y \rangle,$$ where the $x$-degree of
$Q$ is $k$ and $xy$ divides $Q(x,y)$. Therefore, in view
of~(\ref{uu1}) and (\ref{uu2}), whenever we fix $h\in\N$ such that
${\rm lcm}(r_1,\ldots,r_{k-1})|h$ then
$$
p\,\text{-}\!\lim_{n\in\N}U^{Q(n,h)}y=0.$$
 Hence, using~(\ref{u2}), we obtain
$$
p\,\text{-}\!\lim_{h\in\N}\left(p\,\text{-}\!\lim_{n\in\N}\langle
U^{P(n+h)}y,U^{P(n)}y\rangle\right)=
p\,\text{-}\!\lim_{h\in\N}\left(p\,\text{-}\!\lim_{n\in\N} \langle
U^{Q(n,h)}y,U^{-P(h)}y \rangle \right)=$$$$
p\,\text{-}\!\lim_{{\rm lcm}(r_1,\ldots,r_{k-1})|h}\left(p\,\text{-}\!\lim_{n\in\N}
\langle U^{Q(n,h)}y,U^{-P(h)}y \rangle \right)=0.
$$
 By Lemma~\ref{vdC1},
$p\,\text{-}\!\lim_{n\in\N} U^{P(n)}y=0$. Since $y\in\ch$ was arbitrary,
$p\,\text{-}\!\lim_{n\in\N} U^{P(n)}=0$ for all polynomials $P\in\cp$ of degree~$k+1$.

To obtain the result for all polynomials $P\in{\cal P}_{\geq
k+1}$, we use induction and Lemma~\ref{vdC1}. Indeed, if $P$ has
degree~$d\geq k+2$ then
$$
p\,\text{-}\!\lim_{h\in\N}\left(p\,\text{-}\!\lim_{n\in\N} \langle
U^{P(n+h)-P(n)}y,y\rangle\right)=$$$$
p\,\text{-}\!\lim_{h\in\N}\left(p\,\text{-}\!\lim_{n\in\N} \langle
U^{Q(n,h)}y,U^{-P(h)}y\rangle\right)=0$$ because the $x$-degree of
$Q(x,h)$ is $d-1$ and by the induction assumption
$$p\,\text{-}\!\lim_{n\in\N} U^{Q(n,h)}=0.$$ We have proved~(\ref{uu3})
for all $P\in{\cal P}_{\geq k+1}$.

To complete the proof, it remains to prove~(\ref{uu3}) for polynomials of degree~$k$. Suppose  that for some $Q(x)=s_1x+s_2x^2+\ldots +s_kx^k\in{\cal
P}$ with $s_k\geq 1$ we have $p\,\text{-}\!\lim_{n\in\N} U^{Q(n)}={\rm
proj}_{\cal F}$ with ${\cal F}\neq\{0\}$. Since~(\ref{uu1})
and~(\ref{uu2}) hold for $U|_{\cal F}$, by replacing $\cal H$ by
$\cal F$ if necessary, we can assume that
\beq\label{uuuu}p\,\text{-}\!\lim_{n\in\N} U^{Q(n)}=Id.\eeq Then
$p\,\text{-}\!\lim_{n\in\N} U^{rQ(n)}=Id$, where $r={\rm lcm}(r_1,\ldots,
r_{k-1})$. By putting together~(\ref{uuuu}) and~(\ref{uu1}) we
obtain that $p\,\text{-}\!\lim_{n\in\N} U^{rs_kn^k}=Id$, contrary
to~(\ref{uu2}).
\end{proof}

\vspace{2ex} The proof
of Theorem~A for $N=1$ is now completed  by invoking Lemmas~\ref{pd1} and~\ref{vitaly}.

Assume now that Theorem~A holds for some $N\geq1$. We
will prove it for $N+1$. To this end pick a subspace ${\cal
H}_k^{(N)}$. Two cases can arise.

\vspace{2ex}

 {\bf Case 1.} Assume there exists $Q\in{\cal P}_{\leq
N}$ such that $p\,\text{-}\!\lim_{n\in\N}\left(U|_{{\cal
H}^{(N)}_k}\right)^{sQ(n)}=0$ for each $s\geq1$. In this case we
apply the induction assumption (the property~(\ref{qw3})) to
conclude that ${\cal H}_{k}^{(N)}$ will also be one of the
subspaces  in the decomposition~(\ref{dcm}) with $N$ replaced by
$N+1$ (indeed, $p\,\text{-}\!\lim_{n\in\N}\left(U|_{{\cal
H}^{(N)}_k}\right)^{R(n)}=0$ for all $R\in{\cal P}_{\geq N+1}$).

\vspace{2ex}

 {\bf Case 2.} There exist $a_1,\ldots,a_N\geq1$ for
which we have \beq\label{dz1}\begin{array}{l} p\,\text{-}\!\lim_{n\in\N}
\left(U^{in^s}|_{{\cal H}^{(N)}_{k}}\right)=0\;\;\mbox{for
each}\;\;i=1,\ldots, a_s-1,\; \mbox{and}\\p\,\text{-}\!\lim_{n\in\N}
\left(U^{a_sn^s}|_{{\cal H}^{(N)}_{k}}\right)=Id\end{array}\eeq
for $s=1,\ldots, N$.

We now repeat verbatim the proof of Lemma~\ref{pd1} for
$U=U|_{{\cal H}_k^{(N)}}$ and for the family $\{nx^{N+1}:\:n\geq1\}$
instead of the family of linear polynomials.  This yields the
decomposition \beq\label{dz2} {\cal H}_k^{(N)}={\cal
F}_\infty\oplus\bigoplus_{l=0}^\infty {\cal F}_{0^l1},\eeq where
\beq\label{dz3}
 p\,\text{-}\!\lim_{n\in\N}
\left(U^{in^{N+1}}|_{{\cal F}_\infty}\right)=0\;\;\mbox{for
each}\;\;i\geq1\eeq and \beq\label{dz4}\begin{array}{l}
 p\,\text{-}\!\lim_{n\in\N}
\left(U^{jn^{N+1}}|_{{\cal F}_{0^l1}}\right)=0\;\;\mbox{for
each}\;\;j=1,\ldots,l,\\ p\,\text{-}\!\lim_{n\in\N}
\left(U^{(l+1)n^{N+1}}|_{{\cal
F}_{0^l1}}\right)=Id.\end{array}\eeq

In view of Lemma~\ref{vitaly}, the space ${\cal F}_\infty$ will be
one of the subspaces ${\cal H}^{(N+1)}_r$ in the
decomposition~(\ref{dcm}) for $N+1$. The subspaces ${\cal
F}_{0^l1}$ are (in general) not yet elements of the
decomposition~(\ref{dcm}) for $N+1$ because we know only the
$p$-limits for all $P\in{\cal P}_{\leq N}$ and all monomials
$jn^{N+1}$ for $j\geq1$ (the latter follows by the same argument
as in the proof of Lemma~\ref{pd1}).  Take
$Q(x)=b_1x+\ldots+b_Nx^N+b_{N+1}x^{N+1}$, where
\beq\label{dz5}\mbox{$0\leq b_i<a_i$ for $i=1,\ldots,N$ and $1\leq
b_{N+1}< l+1$.}\eeq Using Lemma~\ref{pd1}, we obtain a
decomposition ${\cal F}_{0^l1}={\cal G}_0\oplus {\cal G}_1$ such
that \beq\label{dz6}
 p\,\text{-}\!\lim_{n\in\N}
\left(U^{Q(n)}|_{{\cal G}_{0}}\right)=0, \;  p\,\text{-}\!\lim_{n\in\N}
\left(U^{Q(n)}|_{{\cal G}_{1}}\right)=Id.\eeq Then we take another
$b_i$'s satisfying~(\ref{dz5}) and apply the same reasoning to
${\cal G}_0$ and ${\cal G}_1$ separately. Since the
condition~(\ref{dz5}) defines only a finite family of polynomials,
we end up with a (finite) decomposition
$${\cal F}_{0^l1}={\cal J}_1\oplus \ldots\oplus {\cal J}_K$$
having the property that for each $Q\in{\cal P}_{\leq N+1}$ whose
coefficients satisfy~(\ref{dz5}), we have $p\,\text{-}\!\lim_{n\in\N}
\left(U^{Q(n)}|_{{\cal J}_{q}}\right)=0$ or $Id$, $q=1,\ldots,K$.
But, using~(\ref{dz1}) and~(\ref{dz4}), the same conclusion holds
for all $Q\in{\cal P}_{\leq N+1}$, and we take the subspaces
${\cal J}_q$ as elements of the decomposition~(\ref{dcm})  for
$N+1$.

\vspace{2ex}

To complete the proof of Theorem~A (the
property~(\ref{qw3}) has still to be shown) we need to show that on
subspaces ${\cal J}_q$ for no polynomial $Q\in{\cal P}_{\leq N+1}$
we  have $p\,\text{-}\!\lim_{n\in\N} \left(U^{mQ(n)}|_{{\cal
J}_{q}}\right)=0$ for all $m\geq1$. This is however clear, if
$L:={\rm lcm}(a_1,\ldots,a_N,l+1)$ then $p\,\text{-}\!\lim_{n\in\N}
\left(U^{LQ(n)}|_{{\cal J}_{q}}\right)=Id$.

In this way we have completed the proof of Theorem~A.
As an immediate consequence of it (see~(\ref{qw3})), we obtain
Corollary~B.

\begin{Remark}
Our proof of Theorem~A was based on Theorem~\ref{tw5wstep}.
However it could be organized in another way.  Lemma~\ref{idempotent}  is sufficient in order to perform the first stage of the
construction (see Lemma~\ref{pd1}). Then an application of Lemma~\ref{idempotent} to the restrictions of the operator $U$ to any of
the spaces ${\cal H}_{0^k1}$ shows that
$p\,\text{-}\!\lim_{n\in\N}U^{P(n)}|_{{\cal H}_{0^k1}}$ is an orthogonal
projection for any polynomial $P\in{\cal P}$ of degree~2.
Continuing  this way, we will be able to prove Theorem~A
using only Lemma~\ref{idempotent}. Note that  Theorem~\ref{tw5wstep}
follows from Theorem~A.
\end{Remark}

\subsection{$p$-polynomial properties of Fourier transforms of
measures} We have already noted that the subspaces that appear
in~(\ref{dcm}) in Theorem~A are spectral (see Corollary~\ref{specsub}). This
suggests that the assertion of that theorem depends only on the
maximal spectral type of $U\in{\cal U}({\cal H})$. Therefore from
this theorem one can derive some harmonic analysis consequences
concerning measures on the circle. To this end, we consider a
special case of Theorem~A in which ${\cal
H}=L^2(\cir^1,\sigma)$ and $U=V_\sigma$, $V_\sigma(f)(z)=zf(z)$
with $\sigma$ a probability Borel measure on $\cir^1$.

\begin{Lemma}\label{choj1} Assume that $P\in{\cal P}_{\leq N}$.
Under the above notation, we have:\\
(i) $p\,\text{-}\!\lim_{n\in\N} V_\sigma^{P(n)}=Id$ if and only if
$p\,\text{-}\!\lim_{n\in\N}\widehat{\sigma}(P(n))=1$.\\
(ii) $p\,\text{-}\!\lim_{n\in\N} V_\sigma^{P(n)}=0$ if and only if
$p\,\text{-}\!\lim_{n\in\N}\widehat{\sigma}(P(n)+k)=0$ for each $k\in\Z$.
\end{Lemma}
\begin{proof}
For (i) just notice that  by a convexity argument\footnote{\label{stopkaW} If $g:\bs^1\to\bs^1$ is measurable and $|\int_{\bs^1}g\,d\sigma -1|<\vep$ then we have
$$
\left|\sigma(A)\cdot\left(\frac1{\sigma(A)}\int_Ag\,d\sigma-1\right)+
\sigma(A^c)\cdot\left(\frac1{\sigma(A^c)}\int_{A^c}g\,d\sigma-1\right)\right|<\vep,$$
where $A:=\{z\in\bs^1:\:|g(z)-1|\geq\delta\}=\{z\in\bs^1:\:{\rm Re}(g(z))\leq 1-\delta^2/2\}$.  Now, the real parts of both numbers $\frac1{\sigma(A^c)}\int_{A^c}g\,d\sigma
-1$ and $\frac1{\sigma(A)}\int_{A}g\,d\sigma
-1$ are non-positive, whence $\left|\sigma(A)\cdot\left(\frac1{\sigma(A)}\int_Ag\,d\sigma-1\right)\right|<\vep$ and
since ${\rm Re}(\frac1{\sigma(A)}\int_Ag\,d\sigma)<1-\delta^2/2$, we have $\sigma(A)\leq \vep/(\delta^2/2)$.}
$$p\,\text{-}\!\lim_{n\in\N}
\int_{\cir^1}z^{P(n)}\,d\sigma(z)=1\;\;\mbox{if and only if}\;\;
p\,\text{-}\!\lim_{n\in\N} z^{P(n)}=1$$ in $L^2(\cir^1,\sigma)$. For
(ii), use the elementary fact that $p\,\text{-}\!\lim_{n\in\N}
V_\sigma^{P(n)}=0$ if and only if for each $r,s\in\Z$
$$
p\,\text{-}\!\lim_{n\in\N}\left\langle V^{P(n)}_\sigma
z^r,z^s\right\rangle=0.$$\end{proof}

\begin{Remark}\label{choj8} Note that the condition
$p\,\text{-}\!\lim_{n\in\N}\widehat{\sigma}(P(n)+k)=0$ for each $k\in\Z$ and each $P\in {\cal P}_{\leq N}$
is equivalent to saying that
$p\,\text{-}\!\lim_{n\in\N}\widehat{\sigma}(Q(n))=0$ for each $Q\in\Z[x]$
of degree at most~$N$.\end{Remark}

{\em Proof of Corollary~C.} \ In view of Wiener's lemma,\footnote{Wiener's lemma says that whenever $\cf\subset L^2(\bs^1,\sigma)$ is a closed $V_\sigma$-invariant subspace then there exists a Borel subset $A\subset\bs^1$ such that $\cf=\jed_AL^2(\bs^1,\sigma)$ \cite{Pa}.} each of subspaces ${\cal H}^{(N)}_k$
in Theorem~A is of the form ${\cal
H}^{(N)}_k=\jed_{A^{(N)}_k}L^2(\cir^1,\sigma)$, where
$A_k^{(N)}\cap A_{l}^{(N)}=\emptyset$ (modulo~$\sigma$) whenever
$k\neq l$. In other words, we have proved the existence of
decomposition of $\sigma$ into the following mutually singular
terms: $$\sigma=\sum_{k\geq1} \sigma|_{A^{(N)}_k}.$$

We  set
$\sigma^{(N)}_k=\frac1{\sigma(A_k^{(N)})}\sigma|_{A^{(N)}_k}$ and
$a_k=\sigma(A_k^{(N)})$. To complete the proof of Corollary~C it is enough to apply Theorem~A and
Lemma~\ref{choj1} (see also Remark~\ref{choj8}).\koniec

The decomposition result given by Corollary~C does not depend on the fact that $\sigma$ is
continuous or not. In case of discrete measures however the
following results holds.

\begin{Prop}\label{choj9} Assume that $\sigma$ is a probability
Borel measure on $\cir^1$. Then the following conditions are
equivalent:\\
(i) $\sigma$ is atomic.\\
(ii) $p\,\text{-}\!\lim_{n\in\N}\widehat{\sigma}(n)=1$ for each $p\in
E(\beta\N)$.\\
(iii) $p\,\text{-}\!\lim_{n\in\N}\widehat{\sigma}(P(n))=1$ for each $P\in
{\cal P}$ and $p\in E(\beta\N)$.
\end{Prop}\begin{proof}
To obtain the proof that (i) implies (ii), first notice that if
$p\,\text{-}\!\lim_{n\in\N}U^{P(n)}=0$ (for a non-zero $P\in{\cal P}$)
for some $U\in{\cal U}({\cal H})$ then $U$ has continuous
spectrum; indeed, in view of Lemma~\ref{plimit}~(ii) and footnote~\ref{ordinaryc}, $ p\,\text{-}\!\lim_{n\in\N}\langle
U^{P(n)}x,x\rangle=0$ implies $\lim_{s\to\infty}\langle
U^{k_s}x,x\rangle=0$ for an ordinary subsequence
$(k_s)_{s\geq1}\subset\N$ which is sufficient to conclude that the
spectral measure $\sigma_x$ is continuous.

Let us prove now that (ii) implies (i). Suppose that $U$ has
partly continuous spectrum. Then ${\cal H}={\cal H}_d\oplus {\cal
H}_c$ where ${\cal H}_c\neq\{0\}$ and $U$ has continuous spectrum
on $\ch_c$ and discrete spectrum on $\ch_d$. It follows that there exists an increasing sequence
$(m_i)\subset\N$ of density~$1$ such that $\left(U|_{{\cal
H}_c}\right)^{m_i}\to 0$ weakly. Since $(m_i)$ is of density one,
we can find an increasing subsequence $(n_k)$ of it such that
${\rm FS}((n_k))$ is contained in $\{m_i:\:i\geq1\}$\footnote{This fact
is well known. As a matter of  fact, each subset $A$ of $\N$ containing
intervals of arbitrary lengths includes an IP-set: if
$n_1,\ldots,n_k$ have already been selected,  choose $n_{k+1}\in A$ so that
all sums $n_{i_1}+\ldots+n_{i_s}+n_{k+1}$ are in $A$ for each
$1\leq i_1<\ldots<i_s\leq k$.}. Now,
$${\rm IP}\,\text{-}\!\lim\left(U|_{{\cal H}_c}\right)^{n_\alpha}=0.$$ Hence, by
Lemma~\ref{plimit}, there exists $p\in E(\beta\N)$ such that
$p\,\text{-}\!\lim_{n\in\N}\left(U|_{{\cal H}_c}\right)^{n}=0$ and
therefore we cannot have $p\,\text{-}\!\lim_{n\in\N}U^{n}=Id$, a
contradiction.
\end{proof}

{\em Proof of Corollary~D.} \  Let us prove (i) (the proof of (ii) is similar). By Lemma~\ref{choj1}, $p\,\text{-}\!\lim_{n\in\N}V_\sigma^{ln}=0$ for each $l\geq1$. In view of Corollary~B, $p\,\text{-}\!\lim_{n\in\N}V_\sigma^{P(n)}=0$ for each $0\neq P\in\cp$. The result now follows directly from Lemma~\ref{choj1}.\koniec

\section{Classification of $N$-rigidity
groups}\label{rigiditygroups}
The main goal of this section  is to prove Theorem~E.
\subsection{The notion of $N$-rigidity group} Motivated by the properties of the
decomposition~(\ref{dcm}) in Theorem~A, we will now
introduce and study $N$-rigidity groups.

We fix $N\geq1$.  Assume that
$U\in{\cal U}({\cal H})$ and $p\in E(\beta\N)$. Let $\{0\}\neq{\cal
J}\subset {\cal H}$ be a closed $U$-invariant subspace such
that~(\ref{qw1}) holds for it. Let us now consider
\beq\label{defrg} G(p,N,U,{\cal J}):=\{P\in{\cal P}_{\leq N} :\:
p\,\text{-}\!\lim_{n\in\N} \left(U|_{\cal J}\right)^{P(n)}=Id\}.\eeq It
is not hard to see that $G(p,N,U,{\cal J})$ is a subgroup of
${\cal P}_{\leq N}$. Let $N'$ be the maximum degree of elements
in $G(p,N,U,{\cal J})$. We claim  that on ${\cal J}$, for each
$r=1,\ldots,N'$, there exists (a unique) integer $k_r\geq1$ such that
\beq\label{cla1}\begin{array}{l} p\,\text{-}\!\lim_{n\in\N}
\left(U|_{\cal J}\right)^{jn^r}=0
\;\;\mbox{for}\;\;j=1,\ldots,k_r-1\\
\mbox{and}\;\;p\,\text{-}\!\lim_{n\in\N}  \left(U|_{\cal
J}\right)^{k_rn^r}=Id.\end{array}\eeq Indeed, otherwise for some
$1\leq j\leq N'$, $p\,\text{-}\!\lim_{n\in\N} \left(U|_{\cal
J}\right)^{ln^j}=0$ for all $l\geq1$ and, by Corollary~B,
we have $p\,\text{-}\!\lim_{n\in\N} \left(U|_{\cal J}\right)^{P(n)}=0$
for all $P\in{\cal P}_{\geq j}$. In
other words
$$
N'=\max\left\{1\leq m\leq N:\:\begin{array}{l}\mbox{for each $r=1,\ldots,m$
there exists
$k_r\geq1$ such that}\\
p\,\text{-}\!\lim_{n\in\N} \left(U|_{\cal J}\right)^{jn^r}=0
\;\;\mbox{for}\;\;j=1,\ldots,k_r-1\\
\mbox{and}\;\;p\,\text{-}\!\lim_{n\in\N}  \left(U|_{\cal
J}\right)^{k_rn^r}=Id.\end{array}\right\}$$ Now, fix $1\leq m\leq
N'$ and let $\{Q_1,\ldots,Q_{m}\}$ be an arbitrary set of generators
for ${\cal P}_{\leq m}$. Since $Q_j(x)=\sum_{s=1}^{m}a_{j,s}x^s$, $j=1,\ldots,m$,
$$
p\,\text{-}\!\lim_{n\in\N}  \left(U|_{\cal J}\right)^{bQ_j(n)}=Id,$$
where $b={\rm lcm}(k_1,\ldots,k_m)$. Hence there are (unique) integers
$l_1,\ldots,l_{m}\geq 1$ such that $p\,\text{-}\!\lim_{n\in\N}
\left(U|_{\cal J}\right)^{jQ_r(n)}=0$ for $j=1,\ldots,l_r-1$ and
$p\,\text{-}\!\lim_{n\in\N}  \left(U|_{\cal J}\right)^{l_rQ_r(n)}=Id$ for
$r=1,\ldots,m$. It follows that \beq\label{misg1}
N'=\max\left\{1\leq m\leq N:\:\begin{array}{l}\mbox{for each $r=1,\ldots,m$}\\\mbox{there exists
$l_r\geq1$ such that}\\
p\,\text{-}\!\lim_{n\in\N} \left(U|_{\cal J}\right)^{jQ_r}=0
\\\mbox{for}\;\;j=1,\ldots,l_r-1\\
\mbox{and}\;\;p\,\text{-}\!\lim_{n\in\N}  \left(U|_{\cal
J}\right)^{l_rQ_r}=Id.\end{array}\right\} \eeq

Clearly, \beq\label{ppp1} G(p,N,U,{\cal J})=G(p,N,U,{\cal
H}^{(N)}_t),\eeq where $t\geq1$ is unique so that ${\cal
J}\subset{\cal H}^{(N)}_t$ (see Theorem~A, formula~(\ref{dcm}) and Section~\ref{poldec2013}).

In view of~(\ref{misg1}),  we will assume in what follows that
$N'=N$ (if $N'<N$, the corresponding group has already been
introduced above as $G(p,N',U,{\cal J})$). Let $P(x)=\sum_{r=1}^N
a_rx^r\in\Z[x]$. Then there exist integers $0\leq j_r<k_r$ (see~(\ref{cla1})
for the definition of $k_r$), $m_r\in\Z$ such that
$a_r=j_r+m_rk_r$ for $r=1,\ldots,N$. By~(\ref{cla1}) and
Lemma~\ref{plimit-ciaglosc}
$$ p\,\text{-}\!\lim_{n\in\N} \left(U|_{\cal
J}\right)^{P(n)}=p\,\text{-}\!\lim_{n\in\N} \left(U|_{\cal
J}\right)^{j_1n+j_2n^2+\ldots+j_Nn^N}.$$

For $k\geq1$,  denote $\Z_k=\Z/k\Z$ and let $\pi_k$ stand for the
natural homomorphism  $\Z\to\Z_k$. Set
\beq\label{cla2} \widetilde{G}(p,N,U,{\cal J})
:=\left\{\begin{array}{l} \pi_{k_1}\times\ldots\times\pi_{k_N}
(j_1,\ldots,j_N)\in\Z_{k_1}\oplus\ldots\oplus\Z_{k_N}:\\
0\leq j_s<k_s,\; s=1,\ldots,N\;\mbox{and}\\p\,\text{-}\!\lim_{n\in\N}
\left(U|_{\cal
J}\right)^{j_1n+\ldots+j_Nn^N}=Id\end{array}\right\}.\eeq
It follows from Lemma~\ref{plimit-ciaglosc} that
$\widetilde{G}(p,N,U,{\cal J})$ is a
subgroup of
$\Z_{k_1}\oplus\ldots\oplus\Z_{k_N}$. We also have
\beq\label{ppp2} G(p,N,U,{\cal
J})=\left(\pi_{k_1}\times\ldots\times\pi_{k_N}\right)^{-1}
(\widetilde{G}(p,N,U,{\cal J}))~\footnote{If ${\cal J}={\cal H}$
we will simply write $G(p,N,U)$ and $\widetilde{G}(p,N,U)$.}.\eeq

Assume that $G\subset {\cal P}_{\leq N}$ satisfies $$ \max\{{\rm
deg}\,P:\: P\in G\}=N.$$ Then $G$ is called an $N$-{\em rigidity
group} (or, sometimes, {\em rigidity group} if no confusion
arises) if there are $p\in E(\beta\N)$ and $U\in{\cal U}({\cal
H})$ and $\{0\}\neq{\cal J}\subset {\cal H}$ such that
$G=G(p,N,U,{\cal J})$. Given  $p\in E(\beta\N)$, the groups of the form
$G=G(p,N,U,{\cal J})$ are called    $(p,N)$-{\em rigidity groups}.

Then $U|_{\cal J}$ satisfies~(\ref{cla1}) and the vector
$(k_1,\ldots,k_N)\in\N^N$ is called a {\em period} of $G$. Other
periods of $G$ will be obtained in the same way by choosing a
different basis in ${\cal P}_{\leq N}$, see Remark~\ref{rozj11} below. By~(\ref{cla1}), the group
$\widetilde{G}=\widetilde{G}(p,N,U,{\cal J})$ satisfies the $(\ast)$-property (see Introduction): For each $r=1,\ldots,N$
$$
(\ast)\;\;\;\;\;\;\;\begin{array}{l}
(j_1,\ldots,j_{r-1},j_r,j_{r+1},\ldots,j_N)\in \widetilde{G}\\
(j_1,\ldots,j_{r-1},j'_r,j_{r+1},\ldots,j_N)\in
\widetilde{G}\end{array}\;\Longrightarrow\;j_r=j'_r.$$
Note that $\widetilde{G}$ satisfies~$(\ast)$ if and only if
$(0,\ldots,0,j_r,0,\ldots,0)\in \widetilde{G}$ implies
$j_r=0$.

\begin{Remark}\label{rozj11} The
definition of $\widetilde{G}=\widetilde{G}(p,N,U,{\cal J})$ depends on the choice $x,x^2,\ldots,x^N$ as generators
in ${\cal P}_{\leq N}$. If we identify ${\cal P}_{\leq N}$ with
$\Z^N$ then~(\ref{ppp2}) can be written in a more suggestive form
$$
 G(p,N,U,{\cal
J})=\left(\pi_{k_1}\times\ldots\times\pi_{k_N}\right)^{-1}
(\pi_{k_1}\times\ldots\times\pi_{k_N}(G(p,N,U,{\cal J}))).$$ Note
however that if we take a different set of
generators of ${\cal P}_{\leq N}$, say $Q_1,\ldots,Q_N$, then we
obtain integers $l_1,\ldots,l_N\geq1$ as in~(\ref{misg1}) and, for $G=G(p,N,U,{\cal
J})$, we can define the $\widetilde{G}$ using $\pi_{l_j}$, where
the identification of ${\cal P}_{\leq N}$ with $\Z^N$ is given by
$$Q_j\mapsto (\underbrace{0,\ldots,0}_{j-1},1,
\underbrace{0,\ldots,0}_{N-j}).$$ Then~(\ref{ppp2}) is also true. The vector $(l_1,\ldots,l_N)\in\N^N$
is said to be a {\em period}  of $G$. Note however that since~(\ref{misg1}) holds for any period, the crucial
$(\ast)$-property  does not depend on the choice of
isomorphism between ${\cal P}_{\leq N}$ and $\Z^N$.
\end{Remark}

In view of~(\ref{ppp2}) and of Remark~\ref{rozj11}, $G(p,N,U,{\cal J})$ is entirely
determined by a choice of a period and by the corresponding to it group
$\widetilde{G}(p,N,U,{\cal J})$. The latter group will be called
an $N$-{\em periodic rigidity group} of $U$. Periods and
$N$-periodic rigidity groups depend on a choice of generators in
${\cal P}_{\leq N}$. However, as it was explained in Remark~\ref{rozj11},  all $N$-periodic rigidity groups obtained from  an $N$-rigidity group $G=G(p,N,U,{\cal J})$ satisfy the $(\ast)$-property.

\begin{Remark}\label{rozjasnienie} Any natural number, including~$1$, can appear as one of the entries of a period of an  $N$-periodic
rigidity group. Moreover, a $(p,N)$-rigidity group $G$ equals $\cp_{\leq N}$ if and only if it has a period equal to $(1,\ldots,1)$ (and then, there is only one period for it).
\end{Remark}

\subsection{Algebraic characterization of groups satisfying the
$(\ast)$-property}

A subgroup $K\subset\Z_{k_1}\times\ldots\times\Z_{k_N}$ is called
an {\em algebraic coupling}\footnote{The concept of algebraic coupling is analogous to the notion of joining in ergodic theory \cite{Gl}.} if it has the full projection on each
coordinate. Our aim is to show that a group $\widetilde{G}\subset
\Z_{k_1}\oplus\ldots\oplus\Z_{k_N}$  satisfies the
$(\ast)$-property if and only if it annihilates\footnote{To speak about annihilators of subgroups of a locally compact Abelian group $V$, we need a bilinear form defined on $\widehat{V}\times V$, where $\widehat{V}$, the Pontriagin's {\em dual} of $V$, denotes the group of continuous group homomorphisms from $V$ to $\bs^1$,  see~(\ref{sk1}) below.} an algebraic coupling contained in $\Z_{k_1}\times\ldots\times\Z_{k_N}$, see Theorem~\ref{p-group-main} below.

 To simplify the  notation we
will write $G$ instead of $\widetilde{G}$.
%One more piece of
%notation is needed.
%Given an element $g$ of a group $G$ we denote
%by $\langle g\rangle$ the subgroup of $G$ generated by $g$.
%The cardinality of $G$ is denoted by $|G|$;  $|g|$ denotes the
%order of $g\in G$.
We have already identified $\Z_k$ with $\Z/k\Z$
and elements $z+k\Z$ of $\Z/k\Z$ will be denoted by
$\ov{z}\in\Z_k$.

Let $V=\Z_{k_1}\oplus\ldots\oplus\Z_{k_N}$ and let
$\overline{\pi}_j:V\rightarrow \Z_{k_j}$ denote the projection
onto the $j$th coordinate. Then, up to some identification,
$\widehat{V}=\Z_{k_1}\times\ldots\times \Z_{k_N}$.

Assume that we are given injective characters
$\chi_j:\Z_{k_j}\rightarrow \bs^1$, $j=1,\ldots,N$. Let $
\xi:\widehat{V}\times V\rightarrow \bs^1 $ be a $\Z$-bilinear map
(the Abelian group $\bs^1$ is a $\Z$-module) defined by the
formula \beq\label{sk1}
\xi((c_1,\ldots,c_N),(d_1,\ldots,d_N))=\chi_1(c_1d_1)\cdot\ldots
\cdot \chi_N(c_Nd_N)\eeq for any $(c_1,\ldots,c_N)\in\widehat{V}$,
$(d_1,\ldots,d_N)\in V$, where we write $c_jd_j$ for the multiplication in the ring $\Z_{k_j}$.

Given a subset $A\subset V$,  the {\em annihilator}
$A^\perp$ of $A$ (with respect to $\xi$) is defined by
\beq\label{da}
A^{\perp}:=\{c\in \widehat{V}:\xi(c,A)=\{1\}\}~\footnote{The annihilator of a subset $B\subset \widehat{V}$ is defined similarly.}.
\eeq
The following lemma is classical.

\begin{Lemma}\label{sk2} The set $A^{\perp}$
is a subgroup of $\widehat{V}$ and $A\subset (A^{\perp})^{\perp}$
for any $A\subset V$.\end{Lemma}

Our present goal  is to prove the following result.

\begin{Prop}\label{sk3} Assume we are given injective characters
$\chi_j:\Z_{k_j}\rightarrow \bs^1$, $j=1,\ldots,N$ and
$\xi:\widehat{V}\times V\rightarrow \bs^1$ is defined by the
formula~(\ref{sk1}), where $V=\Z_{k_1}\oplus\ldots\oplus\Z_{k_N}$.
Let $G\subset V$ be a subgroup of $V$. Then

{\rm (i)} $G=(G^{\perp})^{\perp}$.

{\rm (ii)} The group $G$ has the $(\ast)$-property if and only if
$K:=G^{\perp}$ is an algebraic coupling of
$\Z_{k_1},\ldots,\Z_{k_N}$.\end{Prop} \vspace{2ex}

In order to prove~(i) above, it will be convenient to introduce a more general
framework. Assume that $V$ is a locally compact Abelian group and
define a bilinear form by
$$
\langle\cdot,\cdot\rangle:\widehat{V}\times V\to \bs^1,\;\;\langle
\phi,v\rangle=\phi(v).
$$
The above form allows one to define the concept of annihilator of
a set both in $V$ as well as in
$\widehat{V}$~\footnote{Classically, if $G$ is a closed subgroup
of $V$, $G^\perp$ has a natural identification with
$(V/G)^{\widehat{}}$.}. We then have one more classical observation
(cf.\ \cite{Ru}).

\begin{Lemma}\label{kasjan1} For each closed subgroup $G\subset V$,
$G=(G^{\perp})^{\perp}$. \end{Lemma} \begin{proof} By definition,
$G\subset(G^\perp)^\perp$. On the other hand, if $v\notin G$ then
there exists $\phi\in\widehat{V}$ such that $\phi(G)=\{1\}$ and
$\phi(v)\neq1$. It follows that $\phi\in G^\perp$ and
$\phi(v)\neq1$ implies $v \notin(G^\perp)^\perp$.
\end{proof}

Now, suppose that $\alpha$ and $\beta$ are continuous
automorphisms of $V$ and $\widehat{V}$ respectively. Set
$$
\xi:\widehat{V}\times V\to \bs^1,\;\;\xi(\phi,v)=\langle
\beta(\phi),\alpha(v)\rangle.
$$
We obtain the following extension of Lemma~\ref{kasjan1}.

\begin{Lemma}\label{kasjan2} The assertion of Lemma~\ref{kasjan1} holds when  $\perp$ is defined, using~(\ref{da}), with respect to $\xi$. \end{Lemma} \begin{proof} The proof repeats the argument from
the proof of Lemma~\ref{kasjan1} (take $v\notin G$, choose
$\phi$ so that $\beta(\phi)(\alpha(G))=\{1\}$ and
$\beta(\phi)(\alpha(v))\neq1$).
\end{proof}

The claim~(i) of Proposition~\ref{sk3} now follows immediately
from~(\ref{sk1}).

{\em Proof}~{\em of~(ii) in Proposition~\ref{sk3}}. Assume that
$K=G^{\perp}$ is an algebraic coupling of
$\Z_{k_1},\ldots,\Z_{k_N}$ and suppose that $g=(g_1,0,\ldots,0)\in
G$. For any $c_1\in\Z_{k_1}$ there exists $(c_2,\ldots,c_N)\in
\Z_{k_2}\times\ldots\times\Z_{k_N}$ such that
$c:=(c_1,c_2,\ldots,c_N)\in K$. Then
$$
1=\xi(c,g)=\chi_1(c_1g_1).
$$
Since $c_1$ is arbitrary and $\chi_1$ is an injective character,
$g_1=0$. Repeating the argument for the remaining coordinates proves that $G$ has the $(\ast)$-property.

Conversely, assume that $K$ is not an algebraic coupling of
$\Z_{k_1},\ldots,\Z_{k_N}$.  Assume  that
$\ov{\pi}_1(K)\neq \Z_{k_1}$ and let $c_1$ be a generator of
$\ov{\pi}_1(K)$. There is an integer $z$ such that $zc_1=0$ in
$\Z_{k_1}$ and $\ov{z}\neq 0$ in $\Z_{k_1}$. Then $0\neq
g:=(\ov{z},0,\ldots,0)\in K^{\perp}=G$ (the latter equality
follows from~(i)), thus, $G$ does not have the $(\ast)$-property.
\koniec

\vspace{2ex}

Let us consider the special case $N=2$. It follows from the
$(\ast)$-property that $G\subset\Z_{k_1}\oplus\Z_{k_2}$ is the
graph of an isomorphism between subgroups
$\Z_s^{(1)}\subset\Z_{k_1}$ and $\Z_s^{(2)}\subset\Z_{k_2}$, where both
$\Z_s^{(1)}$ and  $\Z_s^{(2)}$
are isomorphic to $\Z_s$ (for some $s|k_i$, $i=1,2$).

In the general case, take $G\subset
\Z_{k_1}\oplus\ldots\oplus\Z_{k_N}$ satisfying the property~($\ast$). Denote by
$G_1$ the projection of $G$ on the first $N-1$ coordinates, i.e.\
on $\Z_{k_1}\oplus\ldots\oplus\Z_{k_{N-1}}$. By the property~($\ast$),
there exists a group homomorphism $w:G_1\to\Z_{k_N}$
such that \beq\label{opis1}
G=\{(g_1,w(g_1))\in\left(\Z_{k_1}\oplus\ldots
\oplus\Z_{k_{N-1}}\right)\oplus \Z_{k_N}:\: g_1\in G_1\}.\eeq Note
that if
$g_1=(0,\ldots,0)\in\Z_{k_1}\oplus\ldots\oplus\Z_{k_{N-1}}$ then
$w(g_1)=0$. It easily follows that $G$ satisfies the $(\ast)$-property
 if and only if the kernel of $w$ (which is a subgroup of
$G_1\subset\Z_{k_1}\oplus\ldots\oplus\Z_{k_{N-1}}$) also satisfies the property~($\ast$).
Then, we can repeat the same argument with $G$ replaced by ${\rm ker}\,w$ to go one step down.

\subsection{Constructions -- measure-theoretic preparations}\label{prepa} Proposition~\ref{sk3} provides  a full
algebraic description of groups satisfying the $(\ast)$-property as
groups which annihilate algebraic couplings of finite
cyclic groups.

Let us summarize the notational agreements from the previous section.  We will be writing
$V$ as $\Z_{k_1}\oplus\ldots\oplus \Z_{k_N}$. By $\chi_r$ we
denote the natural embedding character $\ov{1}\mapsto e^{2\pi
i/k_r}$ of $\Z_{k_r}$ in the group
$$\Sigma_{k_r}:=\{e^{2\pi iu/k_r}:\:u=0,1,\ldots,k_r-1\}.$$
Notice that $\Z_{k_r}$ has the natural structure of
$\Z_{k_r}$-module. Therefore, if
$(j_1,\ldots,j_N)\in\Z_{k_1}\oplus\ldots\oplus\Z_{k_N}$ then the
formula
$$
j_1\times\ldots\times j_N(c_1,\ldots,c_N)=(j_1c_1,\ldots,j_Nc_N),\;\;(c_1,\ldots,c_N)\in \Z_{k_1}\times\ldots\times \Z_{k_N},$$
defines  an endomorphism $j_1\times\ldots\times j_N$ of
$\Z_{k_1}\times\ldots\times\Z_{k_N}$ and the form $\xi$ from the
previous section can be written as
$$
\xi((c_1,\ldots,c_N),
(j_1,\ldots,j_N))=\chi_1\ot\ldots\ot\chi_N(j_1\times\ldots \times
j_N(c_1,\ldots,c_N)).$$

We have now the following result.

\begin{Th}\label{p-group-main} Assume that $G\subset\Z_{k_1}
\oplus\ldots\oplus\Z_{k_N}$ is a subgroup. Then the following
assertions are equivalent: \begin{enumerate} \item[(a)] $G$ is an
$N$-periodic rigidity group.\item[(b)] $G$ satisfies the
$(\ast)$-property.
\item[(c)] $G=K^\perp$ for some algebraic coupling $K$ of $\Z_{k_1},
\ldots,\Z_{k_N}$.\end{enumerate}
\end{Th}

We have already seen that (a) implies (b) and that (b) and (c) are
equivalent. It remains to show that (c)  implies (a) and for that
we need a relevant construction.

Let us note in passing the following immediate consequence of
Theorem~\ref{p-group-main}.

\begin{Cor}\label{trivialgroup} For any choice of
$(k_1,\ldots,k_N)\in\N^N$, the trivial group is an $N$-periodic
rigidity group (with respect to the period $(k_1,\ldots,k_N)$). In other words, for some $U\in\cu(\ch)$ and $p\in E(\beta\N)$, if
$P(x)=a_1x+\ldots+a_Nx^N$ with $a_i\in\Z$ ($i=1,\ldots,N)$ then $$
p\,\text{-}\!\lim_{n\in\N} U^{P(n)}=\left\{
\begin{array}{ccc} Id & \mbox{if} & a_i\equiv0\;\;\mbox{mod}\;\;
k_i,\;i=1,\ldots,N,\\
0&\mbox{otherwise.}& \end{array}\right.$$
\end{Cor}

We will now introduce the main probabilistic tools used in the
constructions which will be carried over in the next subsection.
These constructions will complete  the proof of
Theorem~\ref{p-group-main}, that is, given an algebraic coupling $K\subset \Z_{k_1}\times\ldots\times\Z_{k_N}$, we will find $U\in\cu(\ch)$ and $p\in E(\beta\N)$ such that $K^\perp=\widetilde{G}(p,N,U)$.

Assume  that $\xbm$ is a probability Borel space and let
$K\subset \Z_{k_1}\times\ldots\times \Z_{k_N}$ be an algebraic
coupling. Assume that
$Y=(Y_1,\ldots,Y_N):X\to\Z_{k_1}\times\ldots\times \Z_{k_N}$ is
measurable with $Y(X)=K$ and \beq\label{q6}
Y_\ast(\mu)=\la_K~\footnote{If Z is a random variable on
$(\Omega,{\cal F},P)$ taking values in a Borel space $(\Sigma, {\cal
A})$ then by $Z_\ast$ or $Z_\ast(P)$ we denote the {\em
distribution} of $Z$: $Z_\ast(A):=P(Z^{-1}(A))$ for $A\in \ca$.
Whenever $K$ is a compact group, $\lambda_K$ stands for its (normalized) Haar
measure.}.\eeq Given $(j_1,\ldots,j_N)\in
\Z_{k_1}\oplus\ldots\oplus \Z_{k_N}$, set
$$K_{j_1,\ldots,j_N}:=j_1\times\ldots\times j_N(K)=
\{(j_1c_1,\ldots,j_Nc_N):\:(c_1,\ldots,c_N)\in K\}.$$ Then
$K_{j_1,\ldots,j_N}\subset \Z_{k_1}\times\ldots\times \Z_{k_N}$ is
a subgroup and if we denote
$Y_{j_1,\ldots,j_N}=(j_1Y_1,\ldots,j_NY_N)$ then (in view
of~(\ref{q6}))
$$\left(Y_{j_1,\ldots,j_N}\right)_\ast(\mu)=\la_{K_{j_1,\ldots,j_N}}.$$
%A limit version (cf.\ Lemma~\ref{uproszczenie}) of the following simple lemma
%will provide a main tool of the forthcoming construction.

%\begin{Lemma}\label{q7} Under the above notation, the distribution
%of the random variable $(\chi_1\ot\ldots\ot\chi_N)\circ
%Y_{j_1,\ldots,j_N}$ is the Haar measure of a finite subgroup
%$F_{j_1,\ldots,j_N}$ of $\,\bs^1$. Moreover,
%$F_{j_1,\ldots,j_N}=\{1\}$ if and only if $(j_1,\ldots,j_N)$
%annihilates $K$.
%\end{Lemma}
%\begin{proof} Since $\chi_1\ot\ldots\ot\chi_N|_{K_{j_1,\ldots,j_N}}$ is a
%homomorphism, the distribution of $(\chi\ot\ldots\ot\chi_N)\circ
%Y_{j_1,\ldots,j_N}$ is the Haar measure of the finite subgroup
%$F_{j_1,\ldots,j_N}:=(\chi_1\ot\ldots\ot\chi_N)(K_{j_1,\ldots,j_N})$
%of $\,\bs^1$. To complete the proof it is enough to apply

%Proposition~\ref{sk3}.
%\end{proof}

%TEKST PRZENIESIONY Z 4.1
%-----------------------------------------------------

Assume that $X$ is a compact monothetic metric group
$\mu=\la_X$ the normalized Haar measure on the $\sigma$-algebra $\cb$ of Borel
subsets of $X$. Let $Tx=x+x_0$, where $x_0$ is such that $\{nx_0:\:n\in\Z\}$ is dense
in $X$. Assume that $\xi:X\to\cir^1$ is measurable.

Let $V_\xi^T\in{\cal U}\left(L^2\xbm\right)$, $V_\xi^T(f)=\xi\cdot
f\circ T$, thus  $(V_\xi^T)^n(f)=\xi^{(n)}\cdot
f\circ T^n$ for $n\in\Z$, where
$$\xi^{(n)}(x)=\left\{\begin{array}{lll}
\xi(x)\cdot\ldots\cdot\xi(T^{n-1}x)&\mbox{if}& n>0\\
1&\mbox{if}&n=0\\
(\xi(T^nx)\cdot\ldots\cdot\xi(T^{-1}x))^{-1}&\mbox{if}&
n<0.\end{array}\right.$$ Then, the cocycle
identity $\xi^{(m+n)}(x)=\xi^{(m)}(x)\cdot\xi^{(n)}(T^mx)$ holds. Often,
$\xi$ itself is called a cocycle. By the same token, we can define
cocycles taking values in general compact Abelian groups.

If we take $f\in L^2\xbm$ to be a character of $X$ then, for each $n\in\Z$, we have $$\left\langle
\left(V^T_\xi\right)^{n}f,f\right\rangle=
\int_X\xi^{(n)}(x)f(T^{n}x)\ov{f(x)}\,d\mu(x)=$$
$$
f(nx_0)\int_X\xi^{(n)}(x)|f(x)|^2\,d\mu(x)=
f(x_0)^{n}\int_X\xi^{(n)}(x)\,d\mu(x)=
f(x_0)^{n}\left\langle
\left(V^T_\xi\right)^{n}1,1\right\rangle.$$ It follows that if $(n_t)$ is a strictly increasing sequence of natural numbers then
\beq\label{gext5}\begin{array}{l}
\mbox{if}\;\;\int_X\xi^{(n_t)}(x)\,d\mu(x)\to0\;\;\mbox{then}\;\;
\left(V^T_\xi\right)^{n_t}\to 0\;\mbox{weakly}.\end{array}\eeq
If, moreover,  $n_tx_0\to 0$ in
$X$ then, by the same argument as above, we have
 \beq\label{gext6}
\xi^{(n_t)}\to1\;\mbox{in measure} \;\mbox{implies}\;\;
\left(V^T_{\xi}\right)^{n_t}\to Id\;\mbox{strongly}.\eeq
We recall that by  footnote~\ref{stopkaW}, $\xi^{(n_t)}\to1$ in measure if and only if $\int_X \xi^{(n_t)}\,d\mu\to1$.

Assume now that $\va:X\to  H$ is a cocycle, where $H$ is a compact
Abelian metrizable group. Assume also that $\chi\in\widehat{H}$.

\begin{Prop}\label{critWC}  Assume  that $(n_t)$ is a strictly increasing sequence of natural numbers and $n_tx_0\to 0$
(in $X$). Assume moreover that
$$
\left(\chi\circ\va^{(n_t)}\right)_\ast(\mu)\to\la_F$$ where
$F\subset\bs^1$ is a closed subgroup.~\footnote{That is, either $F=\bs^1$ or $F$ is finite.} Then
$\left(V^T_{\chi\circ\va}\right)^{n_t}\to0$ if $F\neq\{1\}$ and
$\left(V^T_{\chi\circ\va}\right)^{n_t}\to Id$ if $F=\{1\}$.
\end{Prop}\begin{proof} Since
$$
\int_X(\chi\circ\va)^{(n_t)}\,d\mu=
\int_X\chi\circ\va^{(n_t)}\,d\mu=
\int_{\bs^1}z\,d\left(\chi\circ\va^{(n_t)}\right)_\ast(\mu)\to
\int_{\bs^1}z\,d\la_F,$$ the result follows directly
from~(\ref{gext5}) and~(\ref{gext6}).
\end{proof}

Note that, in particular, whenever $L\subset H$ is a (non-trivial)
compact subgroup, \beq\label{gext7} \mbox{
$(\va^{(n_t)})_\ast\to\la_L$ implies
$\int_X\chi(\va^{(n_t)}(x))\,d\mu(x)\to0$} \eeq for each character
$\chi\in\widehat{H}$, $\chi(L)\neq\{1\}$. This, together
with~(\ref{gext5}), yields a criterion for a weak convergence to
zero of our special weighted operators
$(V^T_{\chi\circ\va})^{n_t}$ along any strictly increasing sequence $(n_t)$ of natural numbers.

Denote $\Theta=\chi_1\ot\ldots\ot\chi_N$ and suppose that we are
given a cocycle $\Phi:=(\varphi_1,\ldots,\va_N):X\to
\Z_{k_1}\times\ldots\times\Z_{k_N}$ (over the $(n_t)$-odometer
$T$, see Section~\ref{construction} below). We intend to study the weighted operator
$U=V^{T}_{\Theta\circ\Phi}$. We are interested in computing the weak
limits of $(U^{j_1n_t+j_2n_t^2+\ldots+j_Nn_t^N})_{t\geq1}$ for
$(j_1,\ldots,j_N)\in \Z_{k_1}\oplus\ldots\oplus\Z_{k_N}$. Using the cocycle identity, we have
$$
\int_X\Theta\left(\va_1^{(j_1n_t+j_2n_t^2+\ldots+j_Nn_t^N)},
\ldots,\va_N^{(j_1n_t+j_2n_t^2+\ldots+j_Nn_t^N)}\right)\,d\mu$$$$=
\int_X\Theta\left(\va_1^{(j_1n_t)}+\va_1^{(j_2n_t^2)}\circ
T^{j_1n_t}+\ldots+\va_1^{(j_Nn_t^N)}\circ
T^{j_1n_t+j_2n_t^2+\ldots+j_{N-1}n_t^{N-1}}\right.,$$$$\va_2^{(j_2n_t^2)}+
\va_2^{(j_1n_t)}\circ T^{j_2n_t^2}+\ldots+\va_2^{(j_Nn_t^N)}\circ
T^{j_1n_t+j_2n_t^2+\ldots+j_{N-1}n_t^{N-1}},\ldots,$$
$$\left.\va_N^{(j_Nn_t^N)}+
\va_N^{(j_1n_t)}\circ
T^{j_Nn_t^N}+\ldots+\va_N^{(j_{N-1}n_t^{N-1})}\circ
T^{j_1n_t+\ldots+j_{N-2}n_t^{N-2}+j_Nn_t^N}\right)\,d\mu.$$

Our basic assumption on the cocycles $\va_1,\ldots,\va_N$ is the
following: For each $r=1,\ldots,N$, whenever
$s\in\{1,\ldots,N\}\setminus\{r\}$, \beq\label{ba}
\va_r^{(n_t^s)}\to 0\;\;\mbox{in measure as}\;t\to\infty.\eeq It
follows that  whenever $r=1,\ldots,N$,
$s\in\{1,\ldots,N\}\setminus\{r\}$ and
$i=1,\ldots,k_s-1$,
$$
\va_r^{(in_t^s)}\to 0\;\;\mbox{in measure}$$ and hence
$$\va_r^{(in_t^s)}\circ T^{k}\to0\;\;
\mbox{in measure as}\;\;t\to\infty$$ for each $k\in\Z$. It follows
that under the assumption~(\ref{ba}) (and tacitly assuming that
one of the limits below exists), we have \beq\label{ba1}
\begin{array}{l}
\lim_{t\to\infty}\!\int_X\!\Theta\!\left(\va_1^{(j_1n_t+
j_2n_t^2+\ldots+j_Nn_t^N)},
\ldots,\va_N^{(j_1n_t+j_2n_t^2+\ldots+j_Nn_t^N)}\right)d\mu\!=\\
\lim_{t\to\infty}\int_X\Theta\left(\va_1^{(j_1n_t)},\va_2^{(j_2n_t^2)},
\ldots,\va_N^{(j_Nn_t^N)}\right)\,d\mu.\end{array}\eeq Therefore,
by Proposition~\ref{critWC}, the weak limit of
$(U^{j_1n_t+j_2n_t^2+\ldots+j_Nn_t^mN})_{t\geq1}$ (exists and) depends only on
$$
\lim_{t\to\infty}\int_X\Theta
\left(\va_1^{(j_1n_t)},\va_2^{(j_2n_t^2)},
\ldots,\va_N^{(j_Nn_t^N)}\right)\,d\mu=$$$$
\lim_{t\to\infty}\int_X\chi_1\left(\va_1^{(j_1n_t)}\right)\cdot
\chi_2\left(\va^{(j_2n_t^2)}\right)\cdot\ldots\cdot
\chi_N\left(\va_N^{(j_Nn_t^N)}\right)\,d\mu.$$ Another
characteristic feature of the forthcoming construction is that for
each $r=1,\ldots,N$ and $i=1,\ldots,k_r$ \beq\label{ba2}
\mu\left(\left\{x\in X:\:\va_r^{(in_t^r)}(x)\neq
i\va_r^{(n_t^r)}(x)\right\}\right)\to 0.\eeq This implies
\beq\label{ba3}
\begin{array}{l}
\lim_{t\to\infty}\!\int_X\!\Theta\!
\left(\va_1^{(j_1n_t+j_2n_t^2+\ldots+ j_Nn_t^N)},\!
\ldots,\!\va_N^{(j_1n_t+j_2n_t^2+\ldots+j_Nn_t^N)}\right)d\mu\\
=\lim_{t\to\infty}\int_X\chi_1\ot\ldots\ot\chi_N\left(j_1\times
\ldots\times j_N( \va_1^{(n_t)},
\ldots,\va_N^{(n_t^N)})\right)\,d\mu\end{array}\eeq and since
$\va_r$ takes values in $\Z_{k_r}$ (by taking $i=k_r$ in~(\ref{ba2})), we will have \beq\label{ba3prime}
\mu\left(\left\{x\in X:\:\va_r^{(k_rn_t^r)}(x)\neq
0\right\}\right)\to 0.\eeq It follows from~(\ref{ba3prime}) that
\beq\label{co2}
\int_X\chi_r\left(\va_r^{(k_rn_t^r)}\right)\,d\mu\to1. \eeq
 Moreover,
our construction of $\Phi=(\varphi_1,\ldots,\va_N)$ will ensure
that for each $(j_1,\ldots,j_N)\in \Z_{k_1}\oplus\ldots\oplus
\Z_{k_N}$, we will have \beq\label{co3}\begin{array}{l}
\lim_{t\to\infty}\int_X\chi_1\ot\ldots\ot\chi_N\left(j_1\times
\ldots\times j_N( \va_1^{(n_t)},
\ldots,\va_N^{(n_t^N)})\right)\,d\mu\\
=\left\{\begin{array}{lll}
1&\mbox{if}&(j_1,\ldots,j_N)\in G\\
0&\mbox{otherwise.}&\end{array}\right.\end{array}\eeq  Since for
each $i=1,\ldots,k_r-1$,
$(0,\ldots,0,\stackrel{r}{i},0,\ldots,0)\notin G$ (in view of the
$(\ast)$-property, we have the following:
 \beq\label{co1}
\int_X\chi_r\left(\va_r^{(in_t^r)}\right)\,d\mu\to0
\;\;\mbox{for}\;\; r=1,\ldots,N\;\mbox{and}\;i=1,\ldots,k_r-1.\eeq The key property of $\Phi$
will be that the assumption~(\ref{upr}) of the lemma below will be
satisfied.
\begin{Lemma}\label{uproszczenie} Assume that
$$\Phi=(\varphi_1,\ldots,\varphi_N):X\to
\Z_{k_1}\times\ldots\times\Z_{k_N}$$ satisfies \beq\label{upr}
\left(\varphi_1^{(n_t)},\ldots,\varphi_N^{(n_t^N)}\right)_\ast(\mu)
\to\la_K,\eeq
where $K=G^\perp$. Then (\ref{co3}) holds.\end{Lemma}
\begin{proof}We have
$$
\left(\chi_1\ot\ldots\ot\chi_N\circ (j_1\times\ldots\times j_N)
\left(\varphi_1^{(n_t)},\ldots,\varphi_N^{(n_t^N)}\right)\right)
_\ast(\mu)=$$
$$
\left(\chi_1\ot\ldots\ot\chi_N\circ (j_1\times\ldots\times
j_N)\right)_\ast\left(
\left(\varphi_1^{(n_t)},\ldots,\varphi_N^{(n_t^N)}\right)
_\ast(\mu)\right)\to$$$$\left(\chi_1\ot\ldots\ot\chi_N\circ
(j_1\times\ldots\times j_N)\right)_\ast(\la_K).
$$
The measure $\left(\chi_1\ot\ldots\ot\chi_N\circ
(j_1\times\ldots\times j_N)\right)_\ast(\la_K)$ is the Haar
measure of a finite group $F\subset\bs^1$. By
Proposition~\ref{sk3}, $F=\{1\}$ if and only if
$(j_1,\ldots,j_N)\in G$. The result follows from
Proposition~\ref{critWC} (cf.~(\ref{gext7})).
\end{proof}

 Now,
(\ref{ba3}),~(\ref{co1}) and~(\ref{co2}) imply that for
$r=1,\ldots,N$  \beq\label{co4}
U^{in^r_t}\to0\;\;\mbox{for}\;\;i=1,\ldots,k_r-1;\eeq
\beq\label{co44} U^{k_rn_t^r}\to Id.\eeq
Moreover,~(\ref{co3}) implies that \beq\label{co5}
U^{j_1n_t+\ldots+j_Nn_t^N}\to \left\{\begin{array}{lll}
Id&\mbox{if}&(j_1,\ldots,j_N)\in G\\
0&\mbox{otherwise.}&\end{array}\right.\eeq
We will always assume that
\beq\label{podzielnosc}{\rm lcm}(k_1,\ldots,k_N)|n_t\;\;\mbox{for all
$t\geq t_0$}\eeq (cf.~(\ref{plimitpol2})). Finally, apply
Proposition~\ref{plimitpol} to conclude
that~(\ref{co4}),~(\ref{co44}) and~(\ref{co5}) imply the existence
of $p\in E(\beta\N)$ such that for $r=1,\ldots,N$\beq\label{co6}
p\,\text{-}\!\lim_{n\in\N}
U^{in^r}=0\;\;\mbox{for}\;\;i=1,\ldots,k_r-1,\eeq \beq\label{co7}
p\,\text{-}\!\lim_{n\in\N} U^{k_rn^r}=Id\eeq and \beq\label{co8}
p\,\text{-}\!\lim_{n\in\N} U^{j_1n+\ldots+j_Nn^N}=
\left\{\begin{array}{lll}
Id&\mbox{if}&(j_1,\ldots,j_N)\in G\\
0&\mbox{otherwise.}&\end{array}\right.\eeq Therefore $G=G(p,N,U)$
and the proof of Theorem~\ref{p-group-main} is complete.

\subsection{Main construction}\label{construction}
%%%%%TO CO ZOSTA�O Z APPENDIX
Assume that $(n_t)_{t\geq1}$ is an increasing sequence of natural
numbers with $n_t|n_{t+1}$ for $t\geq1$. In other words
\beq\label{odom1}n_{t+1}=\rho_{t+1}n_t,
\eeq
where $n_0=1$ and the natural numbers $\rho_{t+1}$ satisfy $\rho_{t+1}\geq2$ for $t\geq0$.
If we denote
$\rho_0=1$ then for each $t\geq1$, $n_t=\Pi_{i=0}^t\rho_i$. Set $$
X=\Pi_{t=1}^\infty \Z_{\rho_t}.$$ We will view $X$ as a compact group (in the product topology) with the group law  given by the coordinate addition
with carrying the remainder to the right. Let $\mu$ denote the normalized Haar measure on $X$. Denote
$$
\ov1=(1,0,0,\ldots)$$ and notice that $X$ is a monothetic group since the set
$\{n\cdot\ov1:n\in\Z\}$ is dense in $X$.  We set $Tx=x+\ov1$ for $x\in X$ and call $T$ an {\em
odometer} or, more precisely, the $(n_t)$-{\em odometer}.

Let $D_0^t=\{x\in X:\:x_1=\ldots =x_{t}=0\}$. Then the sets
$T^{i}D^t_0=:D^t_i$ for $i=0,\ldots,n_t-1$ are pairwise disjoint, $\bigcup_{i=0}^{n_t-1}D^t_i=X$
and $T^{n_t}D^t_0=D^t_0$. In this
way we obtain a sequence $\cd^t:=(D^t_0,\ldots,D^t_{n_t-1})$,
$t\geq1$, of partitions of $X$.

Moreover, since $D^{t+1}_0\subset D^t_0$, the partition
$\cd^{t+1}$ refines $\cd^t$, see Figure~\ref{pic1prime}.

For $k=0,\ldots,\rho_{t+1}-1$, we denote $
C_{k}^{t+1}=D_{kn_t}^{{t+1}}\cup\ldots\cup D_{kn_t+n_t-1}^{{t+1}}$ -
the $k$th column of $ \cd^t$. \medskip

\begin{figure}[ht]
\centering
\includegraphics[width=250pt]{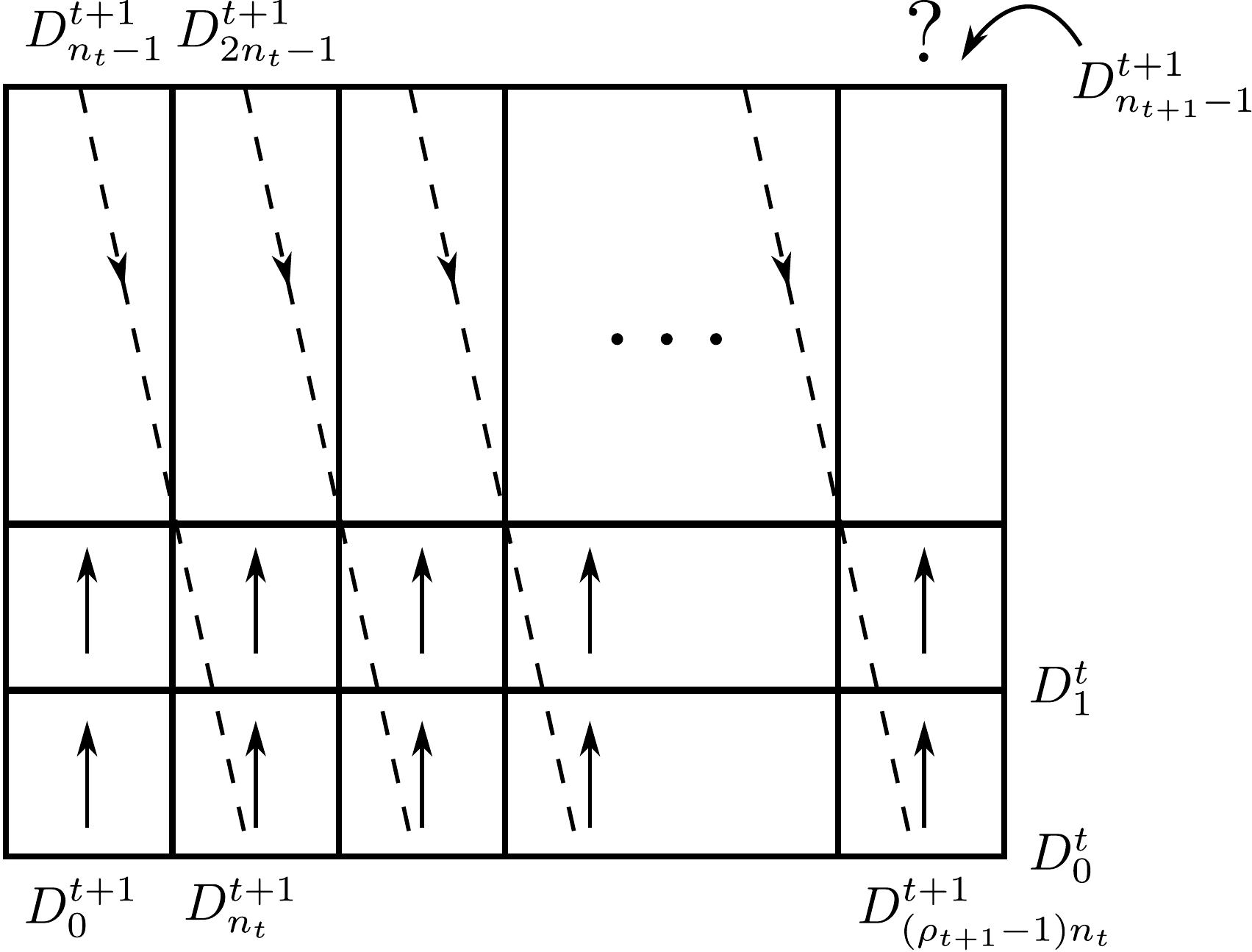}
\caption{At stage $t+1$ the dynamics of $T$ is  defined  on the complement of
$D^{t+1}_{n_{t+1}-1}$.} \label{pic1prime}
\end{figure}

Let $H$ be an Abelian group. Given the $(n_t)$-odometer $T$ we
will be working with so called {\em Morse cocycles} (see e.g.\
\cite{Gu}) $\va:X\setminus\{-\ov{1}\}\to H$.

\def\uv{\underline{v}}
\def\ov#1{\overline{#1}}
\def\wt{\widetilde}
\def\epv {{$\mbox{}$\hfill ${\Box}$\vspace*{1.5ex} }}
\def\vsp{\medskip}
\def\NN{\mathbb{N}}
\def\PP{\mathbb{P}}
\def\RR{\mathbb{R}}
\def\ZZ{\mathbb{Z}}
\def\Ker{{\rm Ker}}
%symbol dla zewn. sumy prostej albo produkt
\def\kart{\times}

The map $\va$ is determined by {\em input data}, namely, by a
collection of sequences $(b^{(t)}_j)_{j=0}^{\rho_t-2}$, $t\geq1$. To define $\va$, first, we define
auxiliary maps $\psi_t$ defined partially on $X$ (see formulas~(\ref{ff1}) and~(\ref{ff2}) below) with values in
$H$ and then we set
$$
\va(x)=\psi_t(x)
$$
for $t$ such that $x$ belongs to the domain of $\psi_t$ (Lemma~\ref{konstrukcja_phi}~(a) yields the
correctness of this definition).

Assume that  for any $t\in\N$ we are given a finite sequence $
b^{(t)}_0,\ldots,b^{(t)}_{\rho_t-2}\in H$. By induction on $t\ge 1$, we will now define:
\begin{itemize}
\item a sequence $a^{(t)}_0,\ldots,a^{(t)}_{n_t-2}\in H$,
\item values $\psi_t(x)\in H$ for $x\in X\setminus
D_{n_{t+1}-1}^{{t+1}}$.
\end{itemize}

 For $t=1$ we set: $a^{(1)}_j=b^{(1)}_j$ for $j=0,\ldots,\rho_1-2$
and
\beq\label{ff1}
\psi_1(x)=\left\{\begin{array}{ll}a^{(1)}_i&x\in D^{1}_i,
i=0,\ldots,n_1-2\\
b^{(2)}_j&x\in D^{2}_{(j+1)n_1-1},\; j=0,\ldots,\rho_2-2.
\end{array}
\right.\eeq

If $t\ge 1$ and  $a^{(t)}_j$ and $\psi_t$ are already defined, we set:\\
$a^{(t+1)}_{jn_t+s}=a^{(t)}_s$ for $s=0,\ldots,n_t-2$,
$j=0,\ldots,\rho_{t+1}-1$,\\
$a^{(t+1)}_{jn_t+n_t-1}=b^{(t+1)}_j$ for $j=0,\ldots,\rho_{t+1}-2$,
\beq\label{ff2}
\psi_{t+1}(x)=\left\{\begin{array}{ll}a^{(t+1)}_i&x\in
D^{{t+1}}_i,
i=0,\ldots,n_{t+1}-2\\
b^{(t+2)}_j&x\in D^{{t+2}}_{(j+1)n_{t+1}-1},\;
j=0,\ldots,\rho_{t+2}-2.
\end{array}
\right.
\eeq

We denote $\Sigma_t=a^{(t)}_0+\ldots+a^{(t)}_{n_t-2}$ for $t\ge 1$.
This construction can be visualized  at Figure~2.

\begin{figure}[ht]
\centering
\includegraphics[width=250pt]{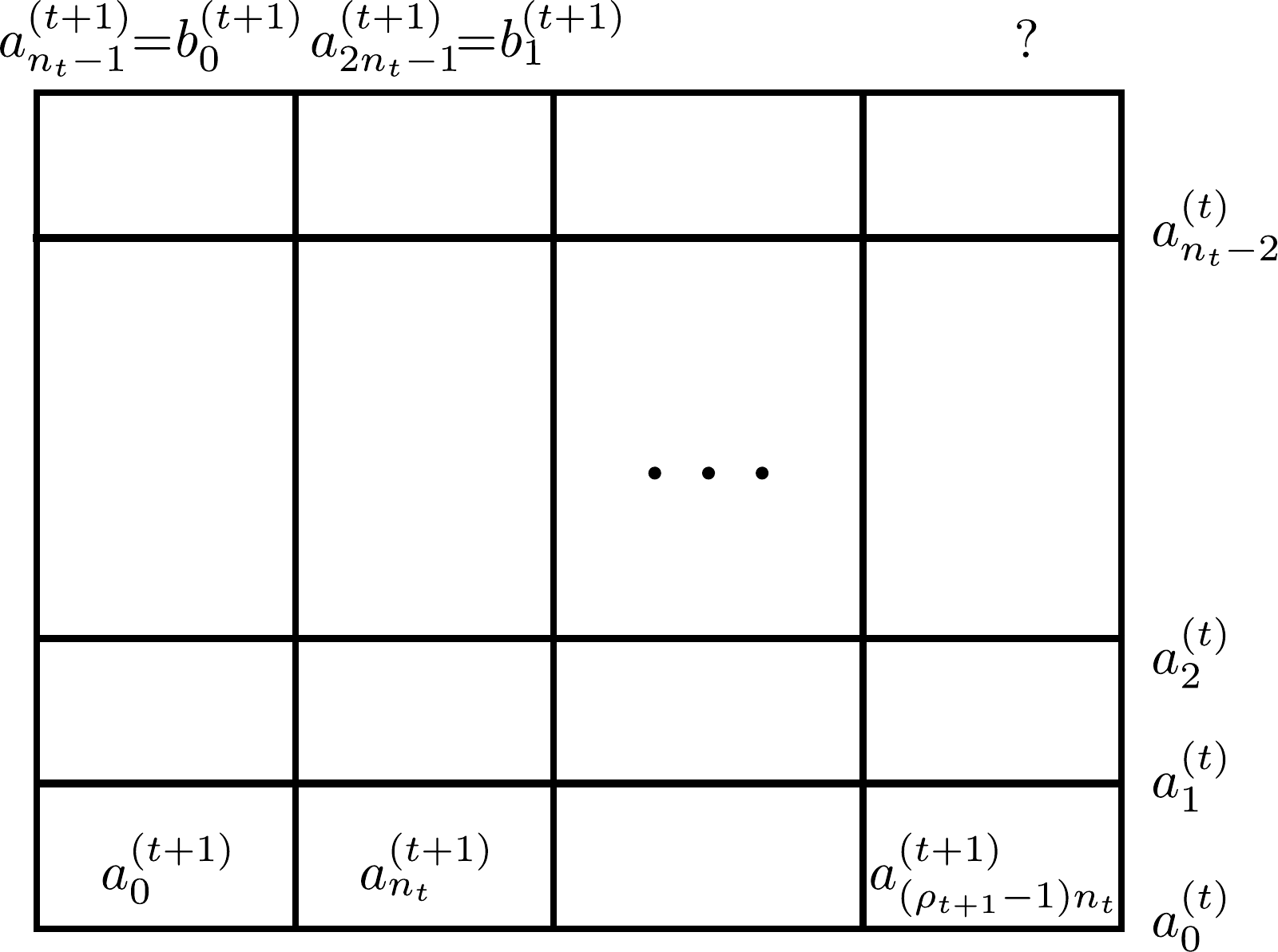}
\caption{At stage $t+1$ of the construction, we have $
a^{(t+1)}_{jn_t+s}=\va|_{D^{t+1}_{jn_t+s}}=\va|_{D^t_s}=a^{(t)}_s$
for $j=0,1,\ldots,\rho_{t+1}-1$ and $s=0,1,\ldots,n_t-2$. We
define $\va$ on $D^{t+1}_{jn_t-1}$, $j=1,2,\ldots,\rho_{t+1}-1$ by
setting $a^{(t+1)}_{jn_t-1}=b^{(t+1)}_{j-1}$; $\va$ remains
undefined on $D^{t+1}_{n_{t+1}-1}$.} \label{pic2prime}
\end{figure}

\begin{Lemma}\label{konstrukcja_phi} Under the notation above:

{\rm (a)} $\psi_t$ equals the restriction of $\psi_{t+1}$ to
$X\setminus D_{n_{t+1}-1}^{t+1}$.

{\rm (b)} Let $i\ge 1$, $0\le u\le \rho_{t+1}-i-1$. Then
$$
\psi_t^{(in_t)}(x):=\psi_t(x)+\ldots+\psi_t(T^{in_t-1}x)=i\Sigma_t+(b^{(t+1)}_u+\ldots
+b^{(t+1)}_{u+i-1})
$$
for $x\in   C^{t+1}_u$.

%{\rm (c)}
%$$
%\mu(\{x\in X:
%\sum_{i=0}^{jn_t}\psi_t(T^{i}(x))=j\Sigma_t+\sum_{s=k}^{k+j-1}b^{(t+1)}_s\})=\frac{1}{\rho_{t+1}}
%$$
%and those sets, for $k=0,...,\rho_{t+1}-j-1$ cover a part of $X$
%of $\mu$-measure $1-\frac{j}{\rho_{t+1}}$.

{\rm (c)} If $b^{(t)}_0+\ldots+b^{(t)}_{\rho_t-2}=0$, for any $t\ge
1$ then $\Sigma_t=0$ for any $t\geq1$. \end{Lemma}

{\bf Proof.} (a) Let $x\in D_k^{{t+1}}$ for some $k\le n_{t+1}-2$.
Write $k=jn_t+r$, $0\le r\le n_t-1$, $0\le j\le \rho_{t+1}-1$.
Then $D_k^{{t+1}}\subset D^t_r$. If $r\le n_t-2$ then
$$
\psi_{t}(x)=a^{(t)}_r=a^{(t+1)}_{jn_t+r}=\psi_{t+1}(x). $$

If $r=n_t-1$ then $j\le \rho_{t+1}-2$ and
$$
\psi_{t}(x)=b^{(t+1)}_j=a^{(t+1)}_{jn_t+n_t-1}=\psi_{t+1}(x).
$$

(b) follows directly from the definition of $\psi_t^{(in_t)}(\cdot)$ (see Figure~\ref{pic3} for an explanation how to compute $\psi_t^{(n_t)}(x)$).

(c) follows by induction:
$$\begin{array}{l}\Sigma_{t+1}=a^{(t+1)}_0+\ldots+a^{(t+1)}_{n_{t+1}-2}=\\
\sum_{k=0}^{n_t-2}a^{(t+1)}_k +
a^{(t+1)}_{n_t-1}+\sum_{k=0}^{n_t-2}a^{(t+1)}_{k+n_t}+a^{(t+1)}_{2n_t-1}+\ldots
+a^{(t+1)}_{(\rho_{t+1}-1)n_t-1}+\sum_{k=0}^{n_t-2}a^{(t+1)}_{k+(\rho_{t+1}-1)n_t}=\\
\rho_{t+1}\Sigma_t+a^{(t+1)}_{n_t-1}+a^{(t+1)}_{2n_t-1}+\ldots
+a^{(t+1)}_{(\rho_{t+1}-1)n_t-1}=\rho_{t+1}\Sigma_t+b^{(t)}_0+\ldots+b^{(t)}_{\rho_{t+1}-2}.
\end{array}
$$

\begin{figure}[ht]
\centering
\includegraphics[width=250pt]{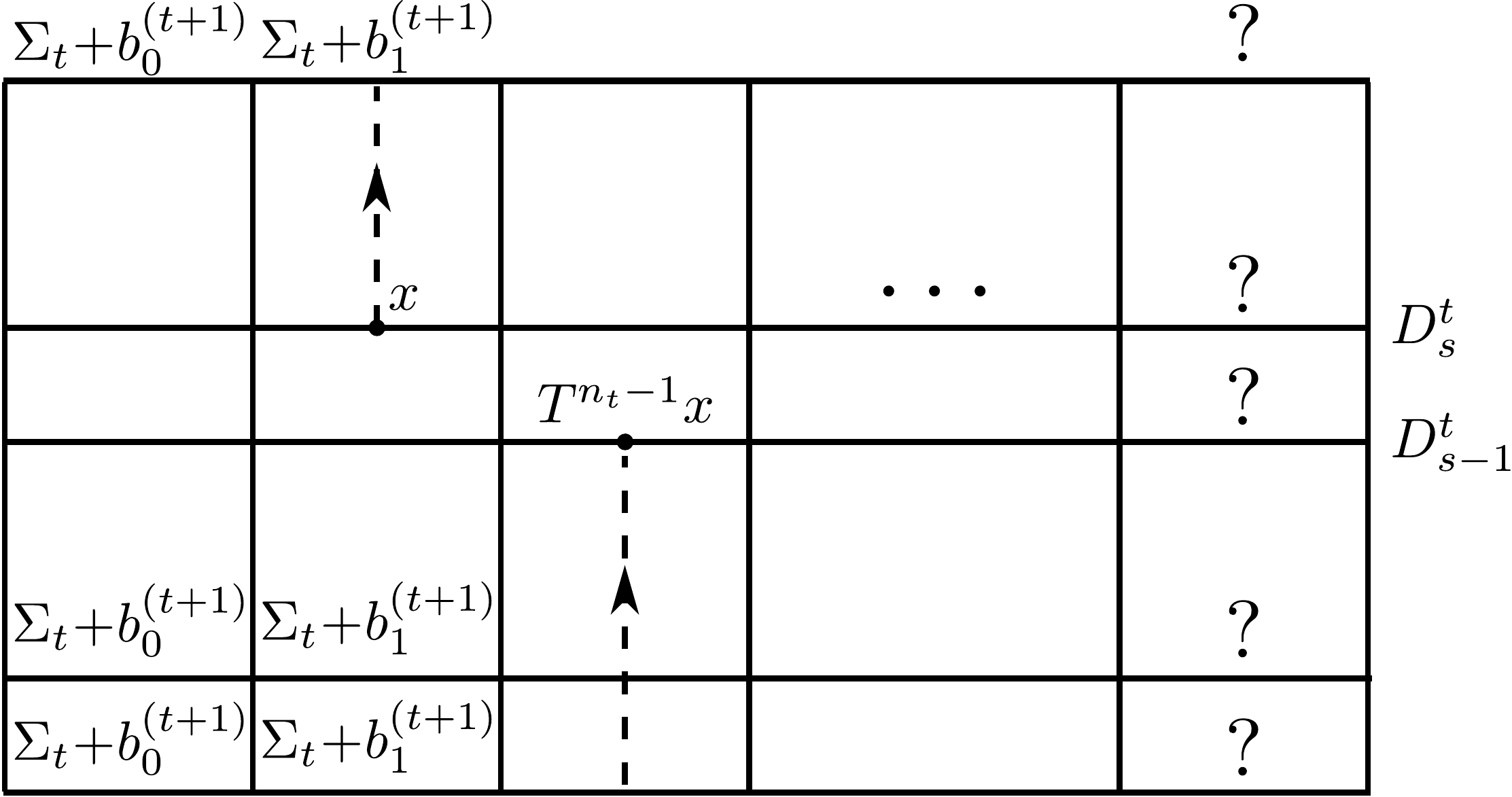}
\caption{The values taken by $\va^{(n_t)}$ viewed at stage $t+1$;
$\va^{(n_t)}$ is constant on each of the first $\rho_{t+1}-1$ columns
of ${\cal D}^t$. If $x\in D^{t+1}_{jn_t+s}$
($j=0,1,\ldots,\rho_{t+1}-2$, $s=0,1,\ldots,n_t-1$) then $
\va^{(n_t)}(x)=\sum_{i=0}^{n_t-1}\va(T^ix)
=a^{(t)}_s+\ldots+a^{(t)}_{n_t-2}+b^{(t+1)}_j+a_0^{(t)}+
\ldots+a^{(t)}_{s-1}=\Sigma_t+b^{(t+1)}_j.$} \label{pic3}
\end{figure}

%%%%%DOT�D APPENDIX
%%% MAIN CONSTRUCTION W NOOWEJ WERSJI
Recall that $K=G^\perp$ is an algebraic coupling of $\Z_{k_1},\ldots,\Z_{k_N}$.
Our aim is to define a Morse cocycle $\Phi=(\va_1,...,\va_N):X\setminus\{-\ov{1}\}\rightarrow
\ZZ_{k_1}\times\ldots\times\ZZ_{k_N}$ so that
(\ref{ba}),~(\ref{ba2}) and~(\ref{co3}) are satisfied.

We will additionally assume now that $n_1>1$ and
\beq\label{nn1}\rho_{t+1}=n_t^{N+1}\rho'_{t+1}\eeq for $t\ge
1$~\footnote{To proceed with the construction in which $\Phi$ satisfies
the conditions (\ref{ba}),~(\ref{ba2}) and~(\ref{co3}), it is
sufficient to take $\rho_{t+1}=n_t^{N+1}$ for all $t\geq1$. We
introduced the factor $\rho'_{t+1}$ because we need in Section 3.5
the condition that $n_t$ is divisible by a given number for $t$ large enough,
see~(\ref{podzielnosc}) and~(\ref{podzielnosc1}) below. Setting
$\rho'_{t+1}=t$ for all $t\geq1$ guarantees this condition.}.

For each $t\geq1$ choose a sequence
$(c_{1,i,t},c_{2,i,t},\ldots,c_{N,i,t})_{i=0}^ {n_t-1}\subset K$ such
that
$$
\frac{1}{n_t}\left|\{0\le i\le n_t-1:
(c_{1,i,t},\ldots,c_{N,i,t})=(c_1,\ldots,c_N)\}\right|\to\frac{1}{|
K|}
$$
for every $(c_1,\ldots,c_N)\in K$.\vsp

%\underline{Some additional notation.} Given
%$\underline{c}=(c_1,...,c_N)\in K$ let
%$$I_{\underline{c},t}=\{0\le i\le n_t-1: (c_{1,i,t},...,c_{N,i,t})=\underline{c}\}$$ and
%$$f_{\underline{c},t}=\sharp I_{\underline{c},t}.$$

%Moreover, for $c_l\in\ZZ_{k_l}$ we denote
%$$
%J_{c_l,t}=\{0\le i\le n_t-1: c_{l,i,t}=c_l\}.
%$$ Observe that $\sum_{\underline{c}\in
%K}f_{\underline{c},t}=\sum_{c_l\in\ZZ_{k_l}}\sharp J_{c_l,t}=n_t$.

%For $i<j$ we set $[i,j]=\{i,i+1,\ldots,j\}$.
Given a real number $x$ we denote by $[x]$ the greatest integer
less than or equal to $x$.

We define the input data for determining
$\va_l:X\setminus\{-\ov{1}\}\rightarrow \ZZ_{k_l}$ to be the
sequences $ b^{(t)}_{0,l},\ldots,b^{(t)}_{\rho_t-2,l} $ for $t\in
\NN$ as follows.

For $l>1$: {\small $$ b^{(t+1)}_{j,l}=\left\{\begin{array}{ll}
c_{l,s,t}& \mbox{\rm if}\;
j=sn_t^{N}\rho'_{t+1}+rn_t^{l-1}\;\mbox{\rm for some}\;0\le s\le n_t-1, 0\le r\le n^{N+1-l}_t\rho'_{t+1}-1, \\
-n_tc_{l,[\frac{s-1}{n_t^{N-l}\rho'_{t+1}}],t}& \mbox{\rm if}\;
j=sn_t^{l}-1\;\mbox{\rm for some}\;
1\le s\le n_t^{N+1-l}\rho'_{t+1}-1,\\
-n_tc_{l,n_t-1,t}& \mbox{\rm if}\;j=n_t^{N+1}\rho'_{t+1}-2,\\
 0&\mbox{\rm otherwise}\end{array}\right.
$$}
and for $l=1$:
$$
b^{(t+1)}_{j,1}=\left\{\begin{array}{ll}
c_{1,\left[\frac{j}{n_t^N\rho'_{t+1}}\right],t}& \mbox{\rm if}\;
n_t\;\mbox{{\rm does not divide}}\; j+1\;\mbox{{\rm and}}\;j\ne n_t^{N+1}\rho'_{t+1}-2, \\
-(n_t-1)c_{1,\left[\frac{j}{n_t^N\rho'_{t+1}}\right],t}& \mbox{\rm
if}\;
n_t|j+1,\\
-(n_t-2)c_{1,n_t-1,t}&  \mbox{\rm if}\;j=n_t^{N+1}\rho'_{t+1}-2.
\end{array}\right.
$$

We give an example of the sequences defined this way in the table
below. In this table $N=3$, $n_t=3$, $\rho'_{t+1}=1$. We omit the
indices $t$ and $t+1$ and we mark only the nonzero terms. To save
space we write $c^j_l$ instead of $c_{l,j,t}$.

\def\hspl{\hspace{-5pt}}
\def\hsp{\hspace{-9pt}}
\bigskip

\noindent {\tiny
\begin{tabular}{|c|c|c|c|c|c|c|c|c|c|c|c|c|c|c|c|c|c|c|c|c|c|c|c|c|c|c|c|}
\hline
 $j$ \hsp &\hspl   0\hsp &\hspl  1\hsp &\hspl  2\hsp &\hspl  3\hsp &\hspl  4\hsp &\hspl  5\hsp &\hspl  6\hsp &\hspl  7\hsp &\hspl  8\hsp &\hspl  9\hsp &\hspl  10\hsp &\hspl  11\hsp &\hspl  12\hsp &\hspl  13\hsp &\hspl  14\hsp &\hspl  15\hsp &\hspl  16\hsp &\hspl  17\hsp &\hspl  18\hsp &\hspl  19\hsp &\hspl  20\hsp &\hspl  21\hsp &\hspl  22\hsp &\hspl  23\hsp &\hspl  24\hsp &\hspl  25\hsp &\hspl  26\\
\hline
$b_{j,3}$\hsp &\hspl   $c^0_3$ \hsp &\hspl  \hsp &\hspl  \hsp &\hspl  \hsp &\hspl  \hsp &\hspl  \hsp &\hspl  \hsp &\hspl  \hsp &\hspl  \hsp &\hspl  $c^0_3$\hsp &\hspl  \hsp &\hspl  \hsp &\hspl  \hsp &\hspl  \hsp &\hspl  \hsp &\hspl  \hsp &\hspl  \hsp &\hspl  \hsp &\hspl  $c^0_3$\hsp &\hspl  \hsp &\hspl  \hsp &\hspl  \hsp &\hspl  \hsp &\hspl  \hsp &\hspl  \hsp &\hspl  \hsp &\hspl  $-3c^0_3$\\
\hline $b_{j,2}$\hsp &\hspl   $c^0_2$ \hsp &\hspl  \hsp &\hspl
\hsp &\hspl  $c^0_2$\hsp &\hspl  \hsp &\hspl  \hsp &\hspl
$c^0_2$\hsp &\hspl  \hsp &\hspl  $-3c^0_2$\hsp &\hspl  $c^0_2$\hsp
&\hspl  \hsp &\hspl  \hsp &\hspl  $c^0_2$\hsp &\hspl  \hsp &\hspl
\hsp &\hspl  $c^0_2$\hsp &\hspl  \hsp &\hspl  $-3c^0_2$\hsp &\hspl
$c^0_2$\hsp &\hspl  \hsp &\hspl  \hsp &\hspl  $c^0_2$\hsp &\hspl  \hsp &\hspl  \hsp &\hspl  $c^0_2$\hsp &\hspl  \hsp &\hspl  $-3c^0_2$\\
\hline $b_{j,1}$\hsp &\hspl  $c^0_1$\hsp &\hspl  $c^0_1$\hsp
&\hspl $-2c^0_1$\hsp &\hspl  $c^0_1$\hsp &\hspl  $c^0_1$\hsp
&\hspl $-2c^0_1$\hsp &\hspl  $c^0_1$\hsp &\hspl  $c^0_1$\hsp
&\hspl $-2c^0_1$ \hsp &\hspl  $c^0_1$\hsp &\hspl  $c^0_1$\hsp
&\hspl  $-2c^0_1$\hsp &\hspl  $c^0_1$\hsp &\hspl  $c^0_1$\hsp
&\hspl  $-2c^0_1$\hsp &\hspl  $c^0_1$\hsp &\hspl  $c^0_1$\hsp
&\hspl  $-2c^0_1$
\hsp &\hspl  $c^0_1$\hsp &\hspl  $c^0_1$\hsp &\hspl  $-2c^0_1$\hsp &\hspl  $c^0_1$\hsp &\hspl  $c^0_1$\hsp &\hspl  $-2c^0_1$\hsp &\hspl  $c^0_1$\hsp &\hspl  $c^0_1$\hsp &\hspl  $-2c^0_1$\\
\hline
\end{tabular}}

\bigskip

\noindent {\tiny
\begin{tabular}{|c|c|c|c|c|c|c|c|c|c|c|c|c|c|c|c|c|c|c|c|c|c|c|c|c|c|c|c|}
\hline
 $j$ \hsp &\hspl   27\hsp &\hspl  28\hsp &\hspl  29\hsp &\hspl  30\hsp &\hspl  31\hsp &\hspl  32\hsp &\hspl  33\hsp &\hspl  34\hsp &\hspl  35\hsp &\hspl  36\hsp &\hspl  37\hsp &\hspl  38\hsp &\hspl  39\hsp &\hspl  40\hsp &\hspl  41\hsp &\hspl  42\hsp &\hspl  43\hsp &\hspl  44\hsp &\hspl  45\hsp &\hspl  46\hsp &\hspl  47\hsp &\hspl  48\hsp &\hspl  49\hsp &\hspl  50\hsp &\hspl  51\hsp &\hspl  52\hsp &\hspl  53\\
\hline
$b_{j,3}$\hsp &\hspl   $c^1_3$ \hsp &\hspl  \hsp &\hspl  \hsp &\hspl  \hsp &\hspl  \hsp &\hspl  \hsp &\hspl  \hsp &\hspl  \hsp &\hspl  \hsp &\hspl  $c^1_3$\hsp &\hspl  \hsp &\hspl  \hsp &\hspl  \hsp &\hspl  \hsp &\hspl  \hsp &\hspl  \hsp &\hspl  \hsp &\hspl  \hsp &\hspl  $c^1_3$\hsp &\hspl  \hsp &\hspl  \hsp &\hspl  \hsp &\hspl  \hsp &\hspl  \hsp &\hspl  \hsp &\hspl  \hsp &\hspl  $-3c^1_3$\\
\hline $b_{j,2}$\hsp &\hspl   $c^1_2$ \hsp &\hspl  \hsp &\hspl
\hsp &\hspl  $c^1_2$\hsp &\hspl  \hsp &\hspl  \hsp &\hspl
$c^1_2$\hsp &\hspl  \hsp &\hspl  $-3c^1_2$\hsp &\hspl  $c^1_2$\hsp
&\hspl  \hsp &\hspl  \hsp &\hspl  $c^1_2$\hsp &\hspl  \hsp &\hspl
\hsp &\hspl  $c^1_2$\hsp &\hspl  \hsp &\hspl  $-3c^1_2$\hsp &\hspl
$c^1_2$\hsp &\hspl  \hsp &\hspl  \hsp &\hspl  $c^1_2$\hsp &\hspl  \hsp &\hspl  \hsp &\hspl  $c^1_2$\hsp &\hspl  \hsp &\hspl  $-3c^1_2$\\
\hline $b_{j,1}$\hsp &\hspl  $c^1_1$\hsp &\hspl  $c^1_1$\hsp
&\hspl $-2c^1_1$\hsp &\hspl  $c^1_1$\hsp &\hspl  $c^1_1$\hsp
&\hspl $-2c^1_1$\hsp &\hspl  $c^1_1$\hsp &\hspl  $c^1_1$\hsp
&\hspl $-2c^1_1$ \hsp &\hspl  $c^1_1$\hsp &\hspl  $c^1_1$\hsp
&\hspl  $-2c^1_1$\hsp &\hspl  $c^1_1$\hsp &\hspl  $c^1_1$\hsp
&\hspl  $-2c^1_1$\hsp &\hspl  $c^1_1$\hsp &\hspl  $c^1_1$\hsp
&\hspl  $-2c^1_1$
\hsp &\hspl  $c^1_1$\hsp &\hspl  $c^1_1$\hsp &\hspl  $-2c^1_1$\hsp &\hspl  $c^1_1$\hsp &\hspl  $c^1_1$\hsp &\hspl  $-2c^1_1$\hsp &\hspl  $c^1_1$\hsp &\hspl  $c^1_1$\hsp &\hspl  $-2c^1_1$\\
\hline
\end{tabular}}

\bigskip

\noindent {\tiny
\begin{tabular}{|c|c|c|c|c|c|c|c|c|c|c|c|c|c|c|c|c|c|c|c|c|c|c|c|c|c|c|}
\hline
 $j$ \hsp &\hspl   54\hsp &\hspl  55\hsp &\hspl  56\hsp &\hspl  57\hsp &\hspl  58\hsp &\hspl  59\hsp &\hspl  60\hsp &\hspl  61\hsp &\hspl  62\hsp &\hspl  63\hsp &\hspl  64\hsp &\hspl  65\hsp &\hspl  66\hsp &\hspl  67\hsp &\hspl  68\hsp &\hspl  69\hsp &\hspl  70\hsp &\hspl  71\hsp &\hspl  72\hsp &\hspl  73\hsp &\hspl  74\hsp &\hspl  75\hsp &\hspl  76\hsp &\hspl  77\hsp &\hspl  78\hsp &\hspl  79\\
\hline
$b_{j,3}$\hsp &\hspl   $c^2_3$ \hsp &\hspl  \hsp &\hspl  \hsp &\hspl  \hsp &\hspl  \hsp &\hspl  \hsp &\hspl  \hsp &\hspl  \hsp &\hspl  \hsp &\hspl  $c^2_3$\hsp &\hspl  \hsp &\hspl  \hsp &\hspl  \hsp &\hspl  \hsp &\hspl  \hsp &\hspl  \hsp &\hspl  \hsp &\hspl  \hsp &\hspl  $c^2_3$\hsp &\hspl  \hsp &\hspl  \hsp &\hspl  \hsp &\hspl  \hsp &\hspl  \hsp &\hspl  \hsp &\hspl    $-3c^2_3$\\
\hline $b_{j,2}$\hsp &\hspl   $c^2_2$ \hsp &\hspl  \hsp &\hspl
\hsp &\hspl  $c^2_2$\hsp &\hspl  \hsp &\hspl  \hsp &\hspl
$c^2_2$\hsp &\hspl  \hsp &\hspl  $-3c^2_2$\hsp &\hspl  $c^2_2$\hsp
&\hspl  \hsp &\hspl  \hsp &\hspl  $c^2_2$\hsp &\hspl  \hsp &\hspl
\hsp &\hspl  $c^2_2$\hsp &\hspl  \hsp &\hspl  $-3c^2_2$\hsp &\hspl
$c^2_2$\hsp &\hspl  \hsp &\hspl  \hsp &\hspl  $c^2_2$\hsp &\hspl  \hsp &\hspl  \hsp &\hspl  $c^2_2$\hsp &\hspl   $-3c^2_2$\\
\hline $b_{j,1}$\hsp &\hspl  $c^2_1$\hsp &\hspl  $c^2_1$\hsp
&\hspl $-2c^2_1$\hsp &\hspl  $c^2_1$\hsp &\hspl  $c^2_1$\hsp
&\hspl $-2c^2_1$\hsp &\hspl  $c^2_1$\hsp &\hspl  $c^2_1$\hsp
&\hspl $-2c^2_1$ \hsp &\hspl  $c^2_1$\hsp &\hspl  $c^2_1$\hsp
&\hspl  $-2c^2_1$\hsp &\hspl  $c^2_1$\hsp &\hspl  $c^2_1$\hsp
&\hspl  $-2c^2_1$\hsp &\hspl  $c^2_1$\hsp &\hspl  $c^2_1$\hsp
&\hspl  $-2c^2_1$
\hsp &\hspl  $c^2_1$\hsp &\hspl  $c^2_1$\hsp &\hspl  $-2c^2_1$\hsp &\hspl  $c^2_1$\hsp &\hspl  $c^2_1$\hsp &\hspl  $-2c^2_1$\hsp &\hspl  $c^2_1$\hsp &\hspl    $-c^2_1$\\
\hline
\end{tabular}}
\bigskip

In order to obtain the sequences in case  $\rho'_{t+1}>1$ one
should repeat each of the three  ``patterns'' above $\rho'_{t+1}$ times.
By a pattern we mean a part of the table containing $c^i_l$ with
fixed $i$. Of course, the last pattern has to be treated in a
slightly modified way because of the lack of the column
$\rho_{t+1}-1$.

For each $l=1,\ldots,N$, let $\va_l:X\setminus\{-\ov{1}\}\rightarrow\ZZ_{k_l}$ be defined
according to the above rule by the input data $(b^{(t)}_{j,l})_{j=0}^{\rho_{t}-2}$, $t\geq1$.

It is clear that the sequences
$(b^{(t+1)}_{j,l})_{j=0,\ldots,\rho_{t+1}-2}$ satisfy the
assumption of the condition~(c) in Lemma~\ref{konstrukcja_phi}.
Thus, thanks to Lemma~\ref{konstrukcja_phi}~(b), we calculate
$\va_l^{(in_t^k)}(x)$ for $x\in X$ as follows. Assume that $x$
belongs to the $s$th column  $C^{t+1}_s$ and $s\le
\rho_{t+1}-in_t^{k-1}-1$. Then, by
Lemma~\ref{konstrukcja_phi}~(b),
$\va_l^{(in_t^k)}(x)=\sum_{j=s}^{s+in_t^{k-1}-1}b^{(t+1)}_{j,l}$,
that is, we  sum up the $in_t^{k-1}$ consecutive   elements in the
$l$th row of the table (i.e.\ in the row $b^{(t+1)}_{j,l}$,
starting from the column $s$).

We are mainly interested in the ``average'' behaviour of $
(\va_1^{(n_t)}(x),\va_2^{(n_t^2)}(x),\ldots,\va_N^{(n_t^N)}(x)) $. In
the following table we present the values of $
(\va_1^{(n_t)}(x),\va_2^{(n_t^2)}(x),\va_3^{(n_t^3)}(x)) $ in the
case considered above. The column marked by $j$ contains the
values for $x\in C^{t+1}_j$.

\bigskip

\noindent {\tiny
\begin{tabular}{|c|c|c|c|c|c|c|c|c|c|c|c|c|c|c|c|c|c|c|c|c|c|c|c|c|c|c|c|}
\hline
 $j$ \hsp &\hspl   0\hsp &\hspl  1\hsp &\hspl  2\hsp &\hspl  3\hsp &\hspl  4\hsp &\hspl  5\hsp &\hspl  6\hsp &\hspl  7\hsp &\hspl  8\hsp &\hspl  9\hsp &\hspl  10\hsp &\hspl  11\hsp &\hspl  12\hsp &\hspl  13\hsp &\hspl  14\hsp &\hspl  15\hsp &\hspl  16\hsp &\hspl  17\hsp &\hspl  18\hsp &\hspl  19\hsp &\hspl  20\hsp &\hspl  21\hsp &\hspl  22\hsp &\hspl  23\hsp &\hspl  24\hsp &\hspl  25\hsp &\hspl  26\\
\hline
$\va_3^{(n_t^3)}$\hsp &\hspl   $c^0_3$ \hsp &\hspl $c^0_3$ \hsp &\hspl $c^0_3$ \hsp &\hspl $c^0_3$ \hsp &\hspl $c^0_3$ \hsp &\hspl $c^0_3$  \hsp &\hspl $c^0_3$ \hsp &\hspl $c^0_3$ \hsp &\hspl $c^0_3$ \hsp &\hspl  $c^0_3$\hsp &\hspl $c^0_3$ \hsp &\hspl$c^0_3$  \hsp &\hspl $c^0_3$ \hsp &\hspl $c^0_3$ \hsp &\hspl$c^0_3$  \hsp &\hspl $c^0_3$  \hsp &\hspl $c^0_3$ \hsp &\hspl $c^0_3$ \hsp &\hspl  $\star$\hsp &\hspl $\star$ \hsp &\hspl $\star$ \hsp &\hspl $\star$  \hsp &\hspl$\star$  \hsp &\hspl $\star$ \hsp &\hspl $\star$  \hsp &\hspl $\star$ \hsp &\hspl  $\star$\\
\hline $\va_2^{(n_t^2)}$\hsp &\hspl   $c^0_2$ \hsp &\hspl
$c^0_2$\hsp &\hspl $c^0_2$ \hsp &\hspl $c^0_2$\hsp &\hspl
$c^0_2$\hsp &\hspl $c^0_2$\hsp &\hspl $\star$\hsp &\hspl$\star$
\hsp &\hspl $\star$\hsp &\hspl $c^0_2$\hsp &\hspl $c^0_2$\hsp
&\hspl $c^0_2$\hsp &\hspl $c^0_2$\hsp &\hspl$c^0_2$ \hsp
&\hspl$c^0_2$ \hsp &\hspl $\star$\hsp &\hspl $\star$\hsp &\hspl
$\star$\hsp &\hspl
$c^0_2$\hsp &\hspl$c^0_2$  \hsp &\hspl $c^0_2$ \hsp &\hspl  $c^0_2$\hsp &\hspl $c^0_2$ \hsp &\hspl $c^0_2$ \hsp &\hspl  $\star$\hsp &\hspl $\star$ \hsp &\hspl  $\star$\\
\hline $\va_1^{(n_t)}$\hsp &\hspl  $c^0_1$\hsp &\hspl $c^0_1$\hsp
&\hspl $\star$\hsp &\hspl  $c^0_1$\hsp &\hspl $c^0_1$\hsp &\hspl
$\star$\hsp &\hspl  $c^0_1$\hsp &\hspl $c^0_1$\hsp &\hspl $\star$
\hsp &\hspl  $c^0_1$\hsp &\hspl $c^0_1$\hsp &\hspl $\star$\hsp
&\hspl  $c^0_1$\hsp &\hspl $c^0_1$\hsp &\hspl $\star$\hsp &\hspl
$c^0_1$\hsp &\hspl $c^0_1$\hsp &\hspl $\star$
\hsp &\hspl  $c^0_1$\hsp &\hspl  $c^0_1$\hsp &\hspl $\star$\hsp &\hspl  $c^0_1$\hsp &\hspl  $c^0_1$\hsp &\hspl $\star$\hsp &\hspl  $c^0_1$\hsp &\hspl  $c^0_1$\hsp &\hspl  $\star$\\
\hline
\end{tabular}}
\bigskip

\noindent {\tiny
\begin{tabular}{|c|c|c|c|c|c|c|c|c|c|c|c|c|c|c|c|c|c|c|c|c|c|c|c|c|c|c|c|}
\hline
 $j$ \hsp &\hspl   27\hsp &\hspl  28\hsp &\hspl  29\hsp &\hspl  30\hsp &\hspl  31\hsp &\hspl  32\hsp &\hspl  33\hsp &\hspl  34\hsp &\hspl  35\hsp &\hspl  36\hsp &\hspl  37\hsp &\hspl  38\hsp &\hspl  39\hsp &\hspl  40\hsp &\hspl  41\hsp &\hspl  42\hsp &\hspl  43\hsp &\hspl  44\hsp &\hspl  45\hsp &\hspl  46\hsp &\hspl  47\hsp &\hspl  48\hsp &\hspl  49\hsp &\hspl  50\hsp &\hspl  51\hsp &\hspl  52\hsp &\hspl  53\\
\hline
$\va_3^{(n_t^3)}$\hsp &\hspl   $c^1_3$ \hsp &\hspl $c^1_3$ \hsp &\hspl $c^1_3$ \hsp &\hspl $c^1_3$ \hsp &\hspl $c^1_3$ \hsp &\hspl $c^1_3$  \hsp &\hspl $c^1_3$ \hsp &\hspl $c^1_3$ \hsp &\hspl $c^1_3$ \hsp &\hspl  $c^1_3$\hsp &\hspl $c^1_3$ \hsp &\hspl$c^1_3$  \hsp &\hspl $c^1_3$ \hsp &\hspl $c^1_3$ \hsp &\hspl$c^1_3$  \hsp &\hspl $c^1_3$  \hsp &\hspl $c^1_3$ \hsp &\hspl $c^1_3$ \hsp &\hspl  $\star$\hsp &\hspl $\star$ \hsp &\hspl $\star$ \hsp &\hspl $\star$  \hsp &\hspl$\star$  \hsp &\hspl $\star$ \hsp &\hspl $\star$  \hsp &\hspl $\star$ \hsp &\hspl  $\star$\\
\hline $\va_2^{(n_t^2)}$\hsp &\hspl   $c^1_2$ \hsp &\hspl
$c^1_2$\hsp &\hspl $c^1_2$ \hsp &\hspl $c^1_2$\hsp &\hspl
$c^1_2$\hsp &\hspl $c^1_2$\hsp &\hspl $\star$\hsp &\hspl$\star$
\hsp &\hspl $\star$\hsp &\hspl $c^1_2$\hsp &\hspl $c^1_2$\hsp
&\hspl $c^1_2$\hsp &\hspl $c^1_2$\hsp &\hspl$c^1_2$ \hsp
&\hspl$c^1_2$ \hsp &\hspl $\star$\hsp &\hspl $\star$\hsp &\hspl
$\star$\hsp &\hspl
$c^1_2$\hsp &\hspl$c^1_2$  \hsp &\hspl $c^1_2$ \hsp &\hspl  $c^1_2$\hsp &\hspl $c^1_2$ \hsp &\hspl $c^1_2$ \hsp &\hspl  $\star$\hsp &\hspl $\star$ \hsp &\hspl  $\star$\\
\hline $\va_1^{(n_t)}$\hsp &\hspl  $c^1_1$\hsp &\hspl $c^1_1$\hsp
&\hspl $\star$\hsp &\hspl  $c^1_1$\hsp &\hspl $c^1_1$\hsp &\hspl
$\star$\hsp &\hspl  $c^1_1$\hsp &\hspl $c^1_1$\hsp &\hspl $\star$
\hsp &\hspl  $c^1_1$\hsp &\hspl $c^1_1$\hsp &\hspl $\star$\hsp
&\hspl  $c^1_1$\hsp &\hspl $c^1_1$\hsp &\hspl $\star$\hsp &\hspl
$c^1_1$\hsp &\hspl $c^1_1$\hsp &\hspl $\star$
\hsp &\hspl  $c^1_1$\hsp &\hspl  $c^1_1$\hsp &\hspl $\star$\hsp &\hspl  $c^1_1$\hsp &\hspl  $c^1_1$\hsp &\hspl $\star$\hsp &\hspl  $c^1_1$\hsp &\hspl  $c^1_1$\hsp &\hspl  $\star$\\
\hline
\end{tabular}}
\bigskip

\noindent {\tiny
\begin{tabular}{|c|c|c|c|c|c|c|c|c|c|c|c|c|c|c|c|c|c|c|c|c|c|c|c|c|c|c|}
\hline
 $j$ \hsp &\hspl   54\hsp &\hspl  55\hsp &\hspl  56\hsp &\hspl  57\hsp &\hspl  58\hsp &\hspl  59\hsp &\hspl  60\hsp &\hspl  61\hsp &\hspl  62\hsp &\hspl  63\hsp &\hspl  64\hsp &\hspl  65\hsp &\hspl  66\hsp &\hspl  67\hsp &\hspl  68\hsp &\hspl  69\hsp &\hspl  70\hsp &\hspl  71\hsp &\hspl  72\hsp &\hspl  73\hsp &\hspl  74\hsp &\hspl  75\hsp &\hspl  76\hsp &\hspl  77\hsp &\hspl  78\hsp &\hspl  79\\
\hline
$\va_3^{(n_t^3)}$\hsp &\hspl   $c^2_3$ \hsp &\hspl $c^2_3$ \hsp &\hspl $c^2_3$ \hsp &\hspl $c^2_3$ \hsp &\hspl $c^2_3$ \hsp &\hspl $c^2_3$  \hsp &\hspl $c^2_3$ \hsp &\hspl $c^2_3$ \hsp &\hspl $c^2_3$ \hsp &\hspl  $c^2_3$\hsp &\hspl $c^2_3$ \hsp &\hspl$c^2_3$  \hsp &\hspl $c^2_3$ \hsp &\hspl $c^2_3$ \hsp &\hspl$c^2_3$  \hsp &\hspl $c^2_3$  \hsp &\hspl $c^2_3$ \hsp &\hspl  $\star$\hsp &\hspl $\star$ \hsp &\hspl $\star$ \hsp &\hspl $\star$  \hsp &\hspl$\star$  \hsp &\hspl $\star$ \hsp &\hspl $\star$  \hsp &\hspl $\star$ \hsp &\hspl  $\star$\\
\hline $\va_2^{(n_t^2)}$\hsp &\hspl   $c^2_2$ \hsp &\hspl
$c^2_2$\hsp &\hspl $c^2_2$ \hsp &\hspl $c^2_2$\hsp &\hspl
$c^2_2$\hsp &\hspl $c^2_2$\hsp &\hspl $\star$\hsp &\hspl$\star$
\hsp &\hspl $\star$\hsp &\hspl $c^2_2$\hsp &\hspl $c^2_2$\hsp
&\hspl $c^2_2$\hsp &\hspl $c^2_2$\hsp &\hspl$c^2_2$ \hsp
&\hspl$c^2_2$ \hsp &\hspl $\star$\hsp &\hspl $\star$\hsp &\hspl
$\star$\hsp &\hspl
$c^2_2$\hsp &\hspl$c^2_2$  \hsp &\hspl $c^2_2$ \hsp &\hspl  $c^2_2$\hsp &\hspl $c^2_2$ \hsp &\hspl  $\star$\hsp &\hspl $\star$ \hsp &\hspl  $\star$\\
\hline $\va_1^{(n_t)}$\hsp &\hspl  $c^2_1$\hsp &\hspl $c^2_1$\hsp
&\hspl $\star$\hsp &\hspl  $c^2_1$\hsp &\hspl $c^2_1$\hsp &\hspl
$\star$\hsp &\hspl  $c^2_1$\hsp &\hspl $c^2_1$\hsp &\hspl $\star$
\hsp &\hspl  $c^2_1$\hsp &\hspl $c^2_1$\hsp &\hspl $\star$\hsp
&\hspl  $c^2_1$\hsp &\hspl $c^2_1$\hsp &\hspl $\star$\hsp &\hspl
$c^2_1$\hsp &\hspl $c^2_1$\hsp &\hspl $\star$
\hsp &\hspl  $c^2_1$\hsp &\hspl  $c^2_1$\hsp &\hspl $\star$\hsp &\hspl  $c^2_1$\hsp &\hspl  $c^2_1$\hsp &\hspl $\star$\hsp &\hspl  $c^2_1$\hsp &\hspl  $\star$\\
\hline
\end{tabular}}
\bigskip

We do not give explicitly the values in the places marked by
$\star$ because the frequency of $\star$'s in a row of the table
tends to~0 as $t\to \infty$\footnote{This is guaranteed by the
condition $n_t^N\rho'_{t+1}/\rho_{t+1}\rightarrow 0$, see
(\ref{nn1}).}.

Similarly, we analyze the values of $\va_k^{(n_t^l)}(x)$ for $k\neq
l$ and we note that in such a case the frequency of zeros tends to 1
as $t\rightarrow\infty$.

We summarize our construction in the following proposition.

\begin{Prop}\label{czestotliwosc} {\rm (a)} Given $l\in\{1,\ldots,N\}$ and $i\in \NN$,
$$
\mu\big(\{x\in X:\:\va_l^{(in^l_t)}(x)=i\va_l^{(n^l_t)}(x)\}\big)\to 1
\;\;\mbox{as}\;\; t\to \infty.$$

{\rm (b)} Given $\underline{c}\in K$,
$$
\mu\big(\{x\in X:\:
(\va_1^{(n_t)}(x),\va_2^{(n_t^2)}(x),\ldots,\va_N^{(n_t^N)}(x))=\underline{c}\}\big)\to
\frac{1}{|K|}\;\;\mbox{as}\;\; t\to \infty.$$

{\rm (c)} If $k\neq l$ then
$$
\mu\big(\{x\in X:\va_k^{(n_t^l)}(x)=0\}\big)\to
1\;\;\mbox{as}\;\; t\to \infty.$$\end{Prop} \epv

Thus $\Phi=(\va_1,\ldots,\va_N)$ constructed above has  the desired properties
(\ref{ba}),~(\ref{ba2}) and~(\ref{co3}).

%%%% DOT�D MAIN CONSTRUCTION W NOOWEJ WERSJI

\subsection{How many constructions can be done over the same
odometer?}\label{howmany} Assume that $N\geq2$. Suppose  that we
have a sequence of periods
$\underline{k}^{(i)}=(k_1^{(i)},\ldots,k^{(i)}_N)$, $i\geq1$. Let
$K_i\subset\Z_{k_1^{(i)}}\times\ldots\times\Z_{k_N^{(i)}}$ be
algebraic couplings and suppose that we want to realize
$G_i:=K_i^{\perp}$, $i\geq1$, as $(N,p)$-periodic rigidity groups
for a common $p\in E(\beta\N)$ (with the~``$\perp$'' depending on
$i\geq1$, see Section~\ref{prepa}). To this end assume that $T$ is
the odometer determined by the sequence $(n_t)_{t\geq1}$ satisfying
for each $i\geq1$ (cf.~(\ref{podzielnosc}))
\beq\label{podzielnosc1}
{\rm lcm}(k_1^{(i)},\ldots,k^{(i)}_N)|n_t\;\;\mbox{for all}\;\;t\geq
t_i.\eeq Moreover, we assume that~(\ref{nn1}) holds. Then we see
that the construction of the relevant cocycle (taking values in
$\Z_{k_1^{(i)}}\times\ldots\times\Z_{k_N^{(i)}}$) in
Section~\ref{construction} can be carried out over this fixed $T$
for each $i\geq1$. This yields a sequence of natural numbers $(n_t)_{t\geq1}$ and a
sequence of unitary operators $(U_i)_{i\geq1}$ acting on Hilbert
spaces ${\cal H}_i$, such that for each $i\geq1$
 \beq\label{co88}
\lim_{t\to\infty} U_i^{j_1n_t+\ldots+j_Nn_t^N}=
\left\{\begin{array}{lll}
Id&\mbox{if}&(j_1,\ldots,j_N)\in K_i^\perp\\
0&\mbox{otherwise.}&\end{array}\right.\eeq

Apply now Part~2 of the proof of Proposition~\ref{plimitpol} to
conclude that we can find a subsequence $(m_s)$ of $(n_t)$ such
that the convergence~(\ref{co88}) can be replaced by
${\rm IP}$-convergence along ${\rm FS}((m_s))$ for each $i\geq1$. Using Lemma~\ref{plimit}, we have
proved the following result.

\begin{Prop}\label{kl} Assume that $N\geq2$.
Assume moreover that $G_i\subset \Z_{k_1^{(i)}}\oplus
\ldots\oplus\Z_{k_N^{(i)}}$ is a subgroup satisfying the
$(\ast)$-property, $i\geq1$. Then there exists $p\in E(\beta\N)$
such that $G_i$ is an $(N,p)$-periodic rigidity group for each
$i\geq1$.
\end{Prop}

\subsection{$N$-rigidity groups are subgroups of ${\cal P}_N$
of finite index (proof of Theorem~E)}\label{finiteindexsec}

Recall from Section 3.1 that $N$-rigidity groups are preimages of
$N$-periodic rigidity groups. Now, we  give a purely algebraic
description of rigidity groups as subgroups of $\Z^N$, see also
Remark~\ref{rozj11}.

%Due to Theorem~\ref{p-group-main} and the definition of
%$N$-rigidity groups (as preimages of $N$-periodic rigidity
%groups), we will be able now to give a purely algebraic
%description of rigidity groups as subgroups of $\Z^N$, see also
%Remark~\ref{rozj11}.

For $i=1,\ldots,N$, we denote by $e_i$ the $i$th standard basis
vector $(0,\ldots,0,1,0,\ldots,0)$ in $\Z^N$. Given a sequence
$\underline{k}=(k_1,\ldots,k_N)$ of natural numbers, we denote by
$$
\pi_{\underline{k}}:\Z^N\rightarrow\Z_{k_1}\oplus\ldots\oplus\Z_{k_N}
$$
the canonical projection $\pi_{k_1}\times\ldots\times\pi_{k_N}$.
In view of Theorem~\ref{p-group-main} and~(\ref{ppp2}), we obtain the following result.

\begin{Prop}\label{sta1}
A subgroup $G$ of $\Z^N$ is an $N$-rigidity group if and only  if
there exists a sequence $\underline{k}=(k_1,\ldots,k_N)$ of natural
numbers such that:

{\rm (a)}  $G$ has the $(\ast)$-property, that is, for any
$i=1,\ldots,N$ and any element $g=(g_1,\ldots,g_N)\in G$: if
$k_j|g_j$ for each $j\neq i$, then $k_i|g_i$,

{\rm (b)}
$G=\pi_{\underline{k}}^{-1}(\pi_{\underline{k}}(G))$.\end{Prop}

It follows that the problem of full description of rigidity groups
contained in ${\cal P}_{\leq N}$ is reduced to the description of
subgroups of $\Z^N$ satisfying~(a) and~(b) in Proposition~\ref{sta1}.
\medskip

{\em Proof of Theorem~E.} If $G$ is an $N$-rigidity group
then, by~(b), there is a group isomorphism
$$
\Z^N/G\cong
(\Z_{k_1}\oplus\ldots\oplus\Z_{k_N})/\pi_{\underline{k}}(G)
$$
which implies that the index of $G$ in $\Z^N$ is finite.

To prove the converse, assume that $G$ has finite index in $\Z^N$
and, for $i=1,\ldots,N$, let $k_i$ be the smallest natural number such that $k_ie_i\in
G$. We show that the conditions~(a) and~(b) are satisfied with
$\underline{k}=(k_1,\ldots,k_N)$.

(a) Assume that $g=(g_1,\ldots,g_N)\in G$ and let $k_j|g_j$ for each
$j\neq 1$. It follows that
$$
(0,g_2,\ldots,g_N)\in k_1\Z\oplus\ldots\oplus k_N\Z={\rm
ker}(\pi_{\underline{k}}).
$$ Moreover, $k_1\Z\oplus\ldots\oplus k_N\Z \subset G$ by the definition of
$k_1,\ldots,k_N$. Then
$$G\ni g-(0,g_2,\ldots,g_N)=(g_1,0,\ldots,0)=g_1e_1.$$ It follows that
$k_1$ divides $g_1$ by the choice of $k_1$. The same argument
works when the index~1 is replaced by any other $i\in\{1,\ldots, N\}$.

(b) Let $x\in \pi_{\underline{k}}^{-1}(\pi_{\underline{k}}(G))$.
Then $x-g\in{\rm ker}(\pi_{\underline{k}})$ for some $g\in G$. But
${\rm ker}(\pi_{\underline{k}})\subset G$, thus  $x\in G$.

We have completed the proof of Theorem~E.\koniec

\begin{Cor}\label{sta2} If $G$ is a rigidity subgroup of $\Z^N$
and $\sigma:\Z^N\rightarrow\Z^N$ is an injective $\Z$-endomorphism then $\sigma(G)$ is also a rigidity subgroup of $\Z^N$.
In particular, this holds if $\sigma$ is a $\Z$-automorphism of
$\Z^N$.\end{Cor}

\begin{Cor}\label{sta3}An intersection of finitely many
rigidity subgroups of $\Z^N$ is a rigidity subgroup. A subgroup
$H$ of a rigidity group $G$ is a rigidity group if and only if $H$
is of finite index in $G$.
\end{Cor}

We now pass to the proof of Corollary~F. We start with a ``separation''
lemma.

\begin{Lemma}\label{sta4} Assume that
$H$ is a subgroup of $\Z^N$ and $Q_1,\ldots,Q_t\in\Z^N\setminus H$.
Then there exists a rigidity subgroup  $G\subset \Z^N$ such that
$H\subset G$ and $Q_1,\ldots,Q_t\notin G$.\end{Lemma}
\begin{proof}
In view of Corollary~\ref{sta3}, we can assume that $t=1$. Let
$P_1,\ldots,P_s$ be a $\Z$-basis of $H$. Then $s\leq N$ and there
exist $P_{s+1},\ldots,P_N\in\Z^N$ such that the set
$\{P_1,\ldots,P_N\}$ is a $\Q$-basis of $\Q^N$. This means that
$P_1,\ldots,P_N$ are independent over $\Q$ (and over $\Z$ as well)
and the subgroup generated by $P_1,\ldots,P_N$ has finite index in
$\Z^N$.

For $M\in\N$ we let $G^{(M)}$ denote  the group generated by
$$
P_1,\ldots,P_s,MP_{s+1},\ldots,MP_{N}.
$$
Clearly, each $G^{(M)}$ has finite index in $\Z^N$, thus it is a
rigidity subgroup, and contains $H$. It is enough to prove that
there exists $M$ such that $Q_1\notin G^{(M)}$.

Suppose that $Q_1\in G^{(1)}$. Then there exist $a_1,\ldots,a_N\in\Z$
such that
$$
Q_1=a_1P_1+\ldots+a_sP_s+a_{s+1}P_{s+1}+\ldots+a_NP_N.
$$
Since $Q_1\notin H$, at least one of $a_{s+1},\ldots,a_N$ is
nonzero.

Let $\ov{M}\in\N$ be a number greater than the maximum of the
absolute values of $a_{s+1},...,a_N$. If $Q_1\in G^{(\ov{M})}$
then there exist $b_1,...,b_N\in\Z$ such that
$$
Q_1=b_1P_1+\ldots+b_sP_s+b_{s+1}\ov{M}P_{s+1}+\ldots+b_N\ov{M}P_N.
$$
By the choice of $\ov{M}$, the sequences $(a_{s+1},\ldots,a_N)$ and
$(b_{s+1}\ov{M},\ldots,b_N\ov{M})$ are different and we get a
contradiction with the independence of $P_1,\ldots,P_N$. \end{proof}

\vspace{2ex}

{\em Proof of Corollary~F.} \ Either apply  Lemma~\ref{sta4} to
$H=\Z P_1+\ldots+\Z P_s$  or observe that $P_{s+1},\ldots,P_N$ do not
belong to the subgroup generated by
$P_1,\ldots,P_s,2P_{s+1},\ldots,2P_N$ and this subgroup has finite
index in $\Z^N$. Now,  apply Theorem~E.\koniec

\subsection{Every finitely generated group of polynomials is a
group of global rigidity (proof of Theorem~G)} Assume that $G\subset{\cal P}_{\leq N}$
is an arbitrary subgroup such that in $G$ we can find a polynomial
of degree~$N$.

\begin{Lemma}\label{sta11} There exists a sequence $G_i\subset
{\cal P}_{\leq N}$, $i\geq1$, of subgroups of finite index in
${\cal P}_{\leq N}$ such that $G=\bigcap_{i\geq1}G_i$.\end{Lemma}
\begin{proof} Let us write ${\cal P}_{\leq N}\setminus
G=\{R_1,R_2,\ldots\}$. Using Lemma~\ref{sta4}, for each $i\geq1$,
we can find $G_i\subset {\cal P}_{\leq N}$ of finite index such
that $G\subset G_i$ and $R_i\notin G_i$. Clearly,
$G=\bigcap_{i\geq1}G_i$.\end{proof}

Assume that $\underline{k}^{(i)}=(k_1^{(i)},\ldots,k^{(i)}_N)$ is
the period of $G_i$, see Section 3.1. Then apply
Proposition~\ref{kl} (to $\pi_{\underline{k}^{(i)}}(G_i)$) to
obtain that there exist $p\in E(\beta\N)$ and $U_i\in{\cal
U}({\cal H}_i)$, $i\geq1$, such that
$$
p\,\text{-}\!\lim_{n\in\N} U_i^{P(n)}=Id$$ for $P\in G_i$ and
$p\,\text{-}\!\lim_{n\in\N} U_i^{R(n)}=0$ for the remaining $R\in{\cal
P}_{\leq N}\setminus G_i$, $i\geq1$. Set $U=U_1\oplus
U_2\oplus\ldots$ We have \beq\label{ml221} p\,\text{-}\!\lim_{n\in\N}
U^{P(n)}=Id\;\;\mbox{if and only if}\;\; P\in \bigcap_{i\geq1}
G_i.\eeq

Taking into account Lemma~\ref{sta11} and ~(\ref{ml221}) we have completed the proof of Theorem~G.

V.\ Bergelson, Department of Mathematics, The Ohio State
University, Columbus, OH 43210, USA

 vitaly@math.ohio-state.edu

 \vspace{1ex}

S. Kasjan, Faculty of Mathematics and Computer Science, Nicolaus
Copernicus University,  Chopin street 12/18, 87-100 Toru\'n,
Poland

skasjan@mat.umk.pl

\vspace{1ex}

 M.\ Lema\'nczyk,
Faculty of Mathematics and Computer Science, Nicolaus Copernicus
University,  Chopin street 12/18, 87-100 Toru\'n, Poland

mlem@mat.umk.pl

\end{document}